\documentclass[
twoside,
]{article}

\newif\ifhyperref
\newif\ifbackrefs
\backrefstrue
\newif\iftodoenv
\newif\iflinenumbers
\newif\ifindexinmargin

\newif\ifnotes
\newif\ifipad{}

\hyperreftrue
\notestrue

\usepackage[only,bigtriangleup]{stmaryrd}
\usepackage{xspace}
\usepackage{wrapfig}
\usepackage{dingbat}
\usepackage{verbatim}  
\iftodoenv
  \usepackage[marginpar]{todo}
  
\else
  \usepackage[hide]{todo}

\fi

\usepackage{array}
\usepackage{colortbl}

\usepackage{paralist}
\usepackage{verbatim}
\usepackage[normalem]{ulem}
\iflinenumbers
\RequirePackage{lineno}
\newenvironment{xequation}{%
  \begin{linenomath*}
    \begin{equation}}
    {%
    \end{equation}
  \end{linenomath*}
}
\newenvironment{xequation*}{%
  \begin{linenomath*}
    \begin{equation*}}
    {%
    \end{equation*}
  \end{linenomath*}
}
\else

\newenvironment{xequation*}{%
  \begin{equation*}}
  {%
  \end{equation*}
}
\fi                             
\usepackage{graphicx}
\usepackage{color}
\usepackage{amscd}
\usepackage{amsmath}
\usepackage{amssymb}
\usepackage{amsopn}
\usepackage{mathabx}
\usepackage{amsthm}
\usepackage{paralist}
\usepackage{makeidx}
\usepackage{url}
\title{The sharp for the Chang model is small}
\author{William J. Mitchell\thanks{I would like to thank the Mitlag
    Leffler Institute, where this work was conceived while the author
    was resident at the program \emph{Mathematical Logic: Set theory and model
      theory} in 2009, and the Fields Institute, where much of
    the work on this paper was done while the 
    author participated in the \emph{Thematic Program on Forcing and
      its Applications} in Fall 2012.}}
\date{\today}
\usepackage{tikz}
\usetikzlibrary{matrix,arrows,decorations.pathmorphing,snakes}

\usepackage[all]{xy}
\ifhyperref
\usepackage[%
  \ifipad dvips\else pdftex\fi, 
  \ifbackrefs
    backref=page
  \fi
]{hyperref}
\fi
\ifnotes
\newcommand\boxederrormessage[1]{\par\noindent\fbox{\vbox{#1}}}
\else
\newcommand\boxederrormessage[1]{}
\fi

\newenvironment{myinparaenum}{\begin{inparaenum}[\upshape(i)]
}{\end{inparaenum}}
\newenvironment{case}[2]{\smallskip\par\noindent\textsc{Case #1} (#2).}{\smallskip}
\newtheorem{theorem}{Theorem}[section]
\newtheorem{lemma}[theorem]{Lemma}
\newtheorem{proposition}[theorem]{Proposition}
\newtheorem{corollary}[theorem]{Corollary}

\newtheorem{claim}{Claim}[theorem]
\newtheorem*{claim*}{Claim}

\theoremstyle{definition}
\newtheorem{definition}[theorem]{Definition}
\theoremstyle{remark}

\newtheorem{conjecture}[theorem]{Conjecture}
\newtheorem{question}[theorem]{Question}

\newcommand{\vg}{virtual gap construction}
\newcommand{\below}[2]{#1{\|_{#2}}}
\newcommand{\etarestrict}{\restrict}
\newcommand{\xer}{\etarestrict\vec\eta}

\newcommand{\ufFromExt}[2]{(#1)_{#2}}
\newcommand{\gn}[1]{\ulcorner\, #1\,\urcorner}
\newcommand{\forceKappa}{\bar\kappa} 

\newcommand{\modelsCi}{\models_{\!\!\!\!\!\lower.7ex\hbox{$\scriptstyle\chang_\iota$}}}

\newcommand{\xre}{\etarestrict\vec\eta{}}
\newcommand{\xrep}{\etarestrict\vetap}
\newcommand{\vetap}{\vec{\eta'}}

\newcommand{\LS}{limit suitable}

\DeclareMathOperator{\supp}{supp}   
\DeclareMathOperator{\DEF}{Def}

\DeclareMathOperator{\add}{add}

\newcommand{\scutdown}{{\uparrow}}

\DeclareMathOperator{\otp}{otp}

\newcommand{\ords}{\Omega}
\newcommand\gkeq{\leftrightarrow}
\newcommand\mgkeq{/{\gkeq}}
\newcommand{\ngkeq}{\mathbin{\hbox to 0pt{\hspace*{.5em}/\hss}\gkeq}}

\DeclareMathOperator{\ps}{\mathcal{P}}
\newcommand\reals{\mathcal{R}}

\newcommand{\RK}{K(\reals)}

\newcommand{\forces}{\Vdash}
\newcommand{\decides}{\parallel}
\newcommand{\ndecides}{\nparallel}

\renewcommand{\phi}{\varphi}

\newcommand{\ecut}{{\uptodownarrow}} 
\newcommand{\cut}{{\vert}}
\DeclareMathOperator{\len}{length}
\newcommand{\sing}[1]{\{#1\}}
\newcommand{\set}[1]{\{\,#1\,\}}
\newcommand{\xset}[2]{\set{#1\mid#2}}
\newcommand{\pair}[1]{\langle#1\rangle}
\newcommand{\seq}[1]{\pair{\,#1\,}}
\newcommand{\xseq}[2]{\seq{#1\mid#2}}
\newcommand{\card}[1]{|#1|}
\DeclareMathOperator{\Crit}{Crit}
\DeclareMathOperator{\crit}{crit}
\DeclareMathOperator{\ult}{Ult}

\newcommand{\tless}{\mathbin{\triangleleft}}

\newcommand{\restrict}{{\upharpoonright}}
\DeclareMathOperator{\range}{range}
\DeclareMathOperator{\domain}{domain}
\DeclareMathOperator{\dirlim}{dir\,lim}

\DeclareMathOperator{\gen}{Gen}

\DeclareMathOperator{\cof}{cf}

\newcommand\chang{\mathbb{C}}
\makeindex{}
\begin{document}
\maketitle\
\begin{quotation}
  This paper is dedicated to the memory of Rich Laver and
  Jim Baumgartner, who  I treasured  as friends, colleagues and
  exemplars  since we were graduate students together.
\end{quotation}
\begin{abstract}
  Woodin has shown that if there is a measurable Woodin cardinal then
  there is, in an appropriate sense, a sharp for the Chang model.  We
  produce, in a weaker sense, a sharp for the Chang model using only
  the existence of a cardinal $\kappa$ having
  an extender of length $\kappa^{+\omega_1}$.
\end{abstract}

\setcounter{tocdepth}{8}        
\tableofcontents
\iflinenumbers
\setpagewiselinenumbers
\modulolinenumbers[3]
\linenumbers
\fi
\section{Introduction}

The Chang model, introduced in \cite{Chang1971Sets-constructi}, is the
smallest model of ZF set theory which contains all countable sequences
of ordinals.    It may be constructed  as $L({^{\omega}\ords})$,
that is, by imitating  the recursive definition of the $L_{\alpha}$
hierarchy,
setting $\chang_{0}=\emptyset$ and
$\chang_{\alpha+1}=\DEF^{\chang_{\alpha}}(\chang_{\alpha})$, but
modifying the definition for limit ordinals $\alpha$ by setting 
$\chang_{\alpha}=[\alpha]^{<\omega_1}\cup\bigcup_{\alpha'<\alpha}\chang_{\alpha}$. 
Alternatively it may be constructed, as did Chang, by replacing the
use of first order logic in the definition of $L$ with the
infinitary logic $L_{\omega_1,\omega_1}$.
We write $\chang$ for the Chang model.

Clearly the Chang model contains the set $\reals$ of reals, and hence is an
extension of $L(\reals)$.    Kunen \cite{Kunen1973A-model-for-the} has shown
that the axiom of choice fails in the
Chang model whenever there are uncountably many measurable cardinals; 
in particular the theory of $\chang$ may vary, even when the set of reals
is held fixed.    We show that in the presence of sufficiently large
cardinal strength this is not true.
An earlier unpublished result of Woodin states that if
there is a Woodin limit of  Woodin cardinals,  then there is a sharp for the
Chang model.   Our result is not strictly comparable to Woodin's,
since although ours
uses a much smaller cardinal, Woodin's notion of a sharp is
stronger, and his result gives the sharp for a stronger model.
Perhaps the most striking aspect of the new result is its
characterization of the size of the Chang model.   Although the Chang
model, like $L(\reals)$, can have arbitrary large  cardinal strength
coded into the reals,  the large cardinal strength of $\chang$
relative to $L(\reals)$, even in the presence of large cardinals in
$V$, is at most  $o(\kappa)=\kappa^{+\omega_1}+1$.

The next three definitions describe our notion of a sharp for $\chang$.
Following this definition and a formal statement of our theorem, we
will more specifically discuss the differences between our result and
that of Woodin.

As with traditional sharps, the sharp for the Chang model asserts the
existence of a closed, unbounded class $I$ of indiscernibles.    
The  conditions on $I$ are given in Definition~\ref{def:Csharp}, following
two preliminary definitions:

\begin{definition}
  \label{def:suitable}
  Say that a subset
  $B$ of  a closed class $I$ is \emph{suitable} if
  \begin{inparaenum}[(a)]\item
    $B$ is countable and closed,
  \item  every member of $B$ which is a
    limit point of $I$ of countable cofinality is also a limit point
    of $B$, and
  \item $B$ is  closed under  immediate predecessors in $I$.
  \end{inparaenum}

  We say that suitable sequences  $B$ and $B'$ are \emph{equivalent} if they
  have the same order type and, writing $\sigma\colon B\to B'$  for the order isomorphism, 
  $\forall\kappa\in B\; \sigma(\kappa)\in\lim(I)\iff\kappa\in\lim(I)$ . 
\end{definition}

Note that if $B$ is suitable and $\beta'$ is the successor of  $\beta$ in $B$, then
either $\beta'$ is the successor of $\beta$ in $I$, or else $\beta'$ is a limit
member of $I$ and $\cof(\beta')>\omega$.
Indeed clauses~(b) and~(c)
of the definition of a suitable sequence are equivalent to the
assertion that every gap in $B$, as a subset of $I$, is capped by a
member of $B$ which is a limit point of $I$ of uncountable cofinality.

\begin{definition}
  \label{def:restrictedFormula}
  Suppose that  $T$ is a collection of constants and functions with
  domain in $[\kappa]^{n}$ for some $n<\omega$.   Write 
  $\mathcal{L}_{T}$ for the language of set theory augmented with symbols denoting the members of $T$. A \emph{restricted formula} in the language $\mathcal{L}_T$   is  a
  formula $\phi$ such that every variable occurring inside an argument of a function in $T$ is free in $\phi$.
\end{definition}
\begin{definition}
  \label{def:Csharp}
  We say that there is a \emph{sharp for the Chang model} $\chang$ if there
  is a closed unbounded class $I$ of ordinals and a set $T$ of
  functions having the following three properties:

  \begin{enumerate}
  \item
    Suppose that $B$ and $B'$ are equivalent suitable sets, and let
    $\phi(B)$  be a restricted formula.  Then 
    \begin{equation*}
      \chang\models \phi(B)\iff \phi(B').
    \end{equation*}
  \item
    Every member of $\chang$ is of the form $\tau(B)$ for some
    term $\tau\in T$  and some suitable sequence $B$.
  \item
    If $V'$ is any universe of ZF set theory such that $V'\supseteq V$
    and $\reals^{V'}=\reals^{V}$ then, for all restricted formulas
    $\phi$
    \begin{equation*}
      \chang^{V'}\models
      \phi(B)\iff
      \chang^{V}\models\phi(B).
    \end{equation*}
    for any  $B\subseteq I$ which is suitable in both $V$ and $V'$.
  \end{enumerate}
\end{definition}
Note, in clause~3, that $\chang^{V'}$ may be larger than
$\chang^{V}$.   
A sequence $B$ which is suitable  in $V$ may not be suitable in
$V'$, as a limit member of $B$ may have uncountable cofinality in $V$
but countable cofinality in $V'$.   However the class $I$, as well as
the theory, will be the same in the two models.

The sharp defined here is somewhat provisional, as is suggested by the
gap between the upper and lower bounds in Theorem~\ref{thm:main}.  The
major consequence of $0^\sharp$ which is shared by this notion of a
sharp is the existence of nontrivial embeddings of $\chang$:

\begin{proposition}\label{thm:sharpEmbed}
  Suppose that $I$ is a class satisfying Definition~\ref{def:Csharp}
  and $\sigma\colon I\to I$ is an increasing map which 
  \begin{myinparaenum}
  \item is continuous at limit points of cofinality $\omega$, and for all $\kappa\in B$ 
  \item $\sigma(\min(I\setminus(\kappa+1)))=\min(I\setminus(\sigma( \kappa)+1))$ and
  \item $\sigma(\kappa)$ is a limit point of $I$ if and only if
    $\kappa$ is a limit point of $I$.
  \end{myinparaenum}
  Then $\sigma$
  can be extended to an elementary embedding $\sigma^*\colon\chang\to\chang$.\qed
\end{proposition}

Definition~\ref{def:Csharp} is not strong enough to imply the converse, that any
elementary embedding $\sigma^*\colon \chang\to\chang$ is generated by
some such map $\sigma\colon I\to I$, and it does not  imply that the
embeddings $\sigma^*$
are unique.   Note, for example, that if a sharp for $\chang$ is
given, according to Definition~\ref{def:Csharp}, by $I$ and $T$ then 
$I'=\set{\kappa_{\nu\cdot\omega_1}\mid \nu\in\ords}$ 
also satisfies the definition, using the set
$T'=T\cup\set{t_{\alpha}\mid\alpha<\omega_1}$ of terms where
$t_{\alpha}(\kappa_{\omega_1\cdot\nu})=\kappa_{\omega_1\cdot\nu+\alpha}$.
However, the restriction to $I'$ of the embedding $i^*\colon\chang\to\chang$ induced by the
embedding 
$i\colon\kappa_{\omega_1\cdot\nu+\alpha}\mapsto\kappa_{\omega_1\cdot(\nu+1)+\alpha}$
does not satisfy the hypothesis of Proposition~\ref{thm:sharpEmbed}.
It is
likely that this deficiency will eventually  be resolved by a
characterization of the ``minimal sharp", that is, of the weakest large
cardinal (or the smallest mouse) which yields a sharp
in the sense of Definition~\ref{def:Csharp}. 

Recall that a traditional sharp, such as $0^{\sharp}$, may be viewed
in either of two different ways:  as a closed and unbounded
class of indiscernibles which generates the full (class) model, or as
a mouse with a final extender on its sequence which is an ultrafilter.

From the first viewpoint, perhaps the most striking difference between
$0^{\sharp}$ and our sharp for $\chang$ is the need for external terms in order to
generate $\chang$ from the indiscernibles.
From the second viewpoint,
regarding the sharp as a mouse, the sharp for the Chang model  involves two modifications:
\begin{enumerate}
\item
  For the purposes of this paper, a \emph{mouse}  will always be a
  mouse over the reals, that is, an extender model of the form $J_{\alpha}(\reals)[{\cal E}]$.
\item The final extender of the mouse which represents the sharp of
  the Chang model will be a proper extender, not an ultrafilter.
\end{enumerate}

It is still unknown  how large the  final extender must be.   We show
that its length is somewhere in the range from $\kappa^{+(\omega+1)}$
to $\kappa^{+\omega_1}$, inclusive:
\begin{theorem}[Main Theorem]
  \label{thm:main}
  \begin{enumerate}
  \item\label{item:main-lower} Suppose that there is no mouse
    $M=J_{\alpha}(\reals)[{\cal E}]$ with a final extender $E={\cal
      E}_{\gamma}$ with critical point $\kappa$ and length $\kappa^{+(\omega+1)}$ in $J_{\alpha}(\reals)[{\cal E}]$ such that  $\cof^V(\len(E))>\omega$.
    Then $K(\reals)^{\chang}$, 
    the core model  over the reals as defined in
    the Chang model, is an iterated ultrapower (without drops) of
    $K(\reals)^{V}$; and hence there is no sharp for the
    Chang model.

  \item\label{item:main-upper}
    Suppose that there is a model $L(\reals)[\mathcal{E}]$ which
    contains all of the reals  and has an extender $E$ of length
    $(\kappa^{+\omega_1})^{L(\reals)[\mathcal{E}]}$, where $\kappa$ is the critical point of $E$.
    Then
    there is a sharp for     $\chang$.
  \end{enumerate}
\end{theorem}

This problem was suggested by Woodin in a conversation at the
Mittag-Lefler Institute in 2009, in which he observed that there
was an immense  gap between the hypothesis needed for his sharp, and
easily obtained 
lower bounds such as a model with a single measure.
At the time I conjectured that the same argument might show that any
extender model would provide a similar lower bound,  but James
Cummings and Ralf Schindler, in the same conversation, 
pointed out that  Gitik's results suggest that it would fail at  an
extender of length $\kappa^{+(\omega+1)}$.

I would also like to thank Moti Gitik, for suggesting his
forcing for the proof of clause~2 and
explaining its use.    
I have generalized his forcing to add new sequences of arbitrary
countable length.   I have also made substantial but, I believe,
inessential changes to the presentation; I hope that he will recognize
his forcing in my presentation.    Many of the arguments in this
paper, indeed almost all of those which do not directly involve either the
generalization of the forcing or the application to the Chang model,
are due to Gitik.

\subsection{Comparison with Woodin's sharp}
Our notion of a sharp for $\chang$  differs from that of Woodin in
several ways.  We will discuss them in roughly increasing order of importance.

\begin{compactenum}
\item 
  The
  theory of our sharp can depend on the set of reals, while the theory of
  Woodin's sharp does not;    however this is due to the large
  cardinals involved, rather than the definition of the sharp.
  Woodin's proof that the theory of $L(\reals)$ is invariant under set
  forcing also shows that the theory of our sharp  stabilizes in the
  presence of a class of Woodin cardinals.
\end{compactenum}
\smallskip

Two differences which might seem to be weaknesses in our model are
actually only differences in presentation.

\begin{compactenum}
  \setcounter{enumi}{1}
\item 
  Woodin's indiscernibles are defined to be indiscernible in 
  the infinitary language $L_{\omega_1,\omega_1}$, whereas we use
  only first order logic.   However  the two languages are
  equivalent in this context: since $\chang$ is closed under countable sequences and
  $\chang_{\alpha}\prec \chang$ whenever $\alpha$ is a member of the
  class $I$ of indiscernibles, the existence of our sharp implies that
  any formula of $L_{\omega_1,\omega_1}$ is equivalent to a formula of
  first order logic having a parameter which is a countable sequence
  of ordinals.

\item 
  For Woodin's sharp, any two subsequences of $I$ are indiscernible,
  while for our sharp only ``suitable''  subsequences are considered.
  The requirement of suitability could be eliminated by replacing $I$ with the class
  of 
  limit points of $I$ of uncountable cofinality, and making a
  corresponding addition to the class
  $T$ of terms, but it seems that doing so would ultimately 
  lose information about the structure
  of the sharp.  This point is discussed further in
  Subsection~\ref{sec:suitability-required}. 
\end{compactenum}
\smallskip

The final two differences are significant.    The first can probably
be removed,  while
the second is basic and explains the difference in the hypotheses used:

\begin{compactenum}
  \setcounter{enumi}{3}
\item 
  The notion of restricted formulas is entirely absent from
  Woodin's results: he allows the  terms from $T$ to be  used as full elements of the
  language.   We believe that our need for restricted formulas is due to
  the choice of terms and will eventually
  be removed by a more complete analysis  resolving the
  question about the size of the minimal mouse needed to give a sharp for $\chang$. 
  If this conjecture turns out to be incorrect then its failure ould be
  a major weakness in our
  notion of a sharp.
\item 
  Woodin has observed, in a personal communication, that his sharp
  actually is a sharp for a much stronger model, namely the smallest model
  which  contains all countable sequences of ordinals and the
  stationary filter on the set $\ps_{\omega_1}([\lambda]^{\omega})$
  for every $\lambda$.   Thus  our constructions do not conflict, but instead
  describe sharps for different models, and this
  explains the difference in the hypotheses needed. 
\end{compactenum}

Woodin has observed (private communication) that some of the gap
between the two sharps can
be filled by modifying the construction of this paper to use the least
mouse $M$ 
over the reals such that $M$ has  infinitely many Woodin cardinals below the 
extenders needed for the conclusion of
Theorem~\ref{thm:main}(\ref{item:main-upper}).   This would
give a version of our sharp which can be coded by a set
$X\subseteq\reals$ having the following property: Suppose that $V'$ is
any inner model of $V$ such that $X\cap V'\in V'$.  Then
$X\cap V'$ codes the corresponding sharp
for the Chang model of $V'$.   Woodin regards this as the ``true
sharp''; however it seems that the better terminology would be to
regard this not as the analog of the sharp operator, but as the analog
of the $M_\omega$ mouse operator.

Future work, and the publication of Woodin's work on his sharp, will
be needed to better comprehend the possibilities of extensions of 
sharps  for Chang-like models in analogy with the
extended theory related to $0^{\sharp}$.  At the same time, as points~3 and~4 above
make clear, further work is needed towards clarifying the basic notion
of a sharp for the Chang model as presented in this paper.

\subsection{Some basic facts about $\chang$}
\label{sec:basic-facts}

As pointed out earlier, the Axiom of Choice fails in $\chang$
if there are infinitely many measurable cardinals.    However, the
fact that $\chang$ is closed under countable sequences implies that
the axiom of Dependent Choice holds, and this is enough to avoid most of the serious 
pathologies which can occur in a model without choice.   
For life without Dependent Choice, see for example
\cite{Gitik2012Violating-the-s}, which gives a model with surjective
maps from $\ps(\aleph_{\omega})$ onto an arbitrarily large cardinal
$\lambda$ without any need for large cardinals.

\begin{todoenv} {(7/25/14) --- idea --- A question: what can be said
    about the covering lemma in the absence of choice.  Does it apply
    to this model?}
\end{todoenv}

The same argument that shows that every member of $L$ is ordinal definable implies that every member of
$\chang$ is definable in $\chang$ using a countable sequence of
ordinals as parameters.

In the proof of part~1 of Theorem~\ref{thm:main} we make use of 
the core model $K(\reals)$ inside of $\chang$, and in the absence of the Axiom
of Choice this requires some justification.   In large part the Axiom
of Choice can be avoided in the construction and theory of this core
model, since the core model itself is well ordered (after using countably complete forcing to map the reals onto $\omega$).   However one 
application of the Axiom of Choice falls outside of this situation:
the use of Fodor's pressing down lemma, the proof of which requires
choosing closed unbounded sets as witnesses that the sets where the
function is constant are all nonstationary.    This lemma is needed in
the construction of $K(\reals)$ in order to prove that  the comparison of pairs of mice by iterated ultrapowers  always terminates.     
However, this is not a problem 
in the construction of $K(\reals)$ in $\chang$, as we can apply Fodor's lemma in  the universe $V$, which satisfies the Axiom of Choice, to verify that all
comparisons terminate.

The proof of the covering lemma involves other uses of Fodor's lemma;
however we  do not use the covering lemma.

\subsection{Notation}
\label{sec:notation}

We use generally standard set theoretic notation.    We use $\ords$
to mean the class of all ordinals, and frequently treat $\ords$ itself
as an ordinal.    If $h$ is a function, then  we use $h[B]$ for the range of $h$ on $B$,
$h[B]=\set{h(b)\mid b\in B}$.   We write $[X]^{\kappa}$ for the set of
subsets of $X$ of size $\kappa$.

In forcing,  we use $p< q$ to mean that $p$ is stronger than $q$.  
The notation $p\decides\phi$ means that the
condition $p$ decides $\phi$, that is, either $p\forces \phi$ for
$p\forces\lnot \phi$. 
If $P$ is a forcing order and $s\in P$,  then we write $\below{P}{s}$
for the forcing below $s$, that is, the restriction of $P$ to  $\set{t\in P\mid t\leq s}$.

If $E$ is an extender, then we write $\supp(E)$ for the support, or
set of generators, of $E$.  Typically we take this to be  the interval
$[\kappa,\len(E))$ where $\kappa$ is the critical point of $E$;
however we frequently make use of the restriction of $E$ to a
nontransitive\footnote{We regard $\supp(E)=[\kappa,\lambda)$ as
  ``transitive'' despite its omission of ordinals less than
  $\kappa$.   We could equivalently, but slightly less conveniently, use
  $\supp(E)=\len(E)$.} 
set of generators: that is, if $S\subseteq\supp(E)$ then
we write $E\ecut S$ for the restriction of $E$ to $S$, so $\ult(V,E\ecut
S)\cong\set{i^{E}(f)(a)\mid f\in V\land a\in[S]^{<\omega}}$.   We
remark that $\ult(V,E\ecut S)=\ult(V,\bar E)$, where $\bar E$ is the
\emph{transitive collapse of $E\ecut S$}, that is, the extender 
obtained from $E\ecut S$ by using the transitive collapse $\sigma\colon
[\kappa,\len(\bar E))\cong \supp(E)\cap
\set{i^E(f)(a)\mid a\in[S]^{<\omega}}$ and setting
the ultrafilter
$(\bar E)_{\alpha}=E_{\sigma^{-1}(\alpha)}$.
In cases where the $E\ecut S\notin M$ but the 
transitive collapse $\bar E\in M$, we frequently describe
constructions as using $E\ecut S$ when the actual construction inside
$M$ must use  $\bar E$.  Such use will not always be explicitly
stated. 

We write $\ufFromExt{E}{a}$ for the ultrafilter $\set{x\subseteq
  H_{\crit(E)}\mid a\in i^E(x)}$.

We make extensive use of the core model over the reals, $K(\reals)$.
However we make no (direct) use of fine structure, largely because we
make no attempt to use the weakest hypothesis which could be treated
by our argument.    The reader will need to be familiar with extender
models, but only those weaker than strong cardinal, that is, without the
complications of overlapping extenders and iteration trees.     For
our purposes, a mouse will be an extender model $M=J_{\alpha}(\reals)[\mathcal{E}]$, where
$\reals$ is the set $\ps(\omega)$ of reals and $\mathcal{E}$ is a
sequence of extenders, and it generally can be
assumed to be a model of Zermelo set theory (and therefore equal to $L_{\alpha}(\reals)[\mathcal{E}]$).

The ultrafilters in a mouse $M$ over the reals, including those
appearing as components of an extender, are all complete over sets of
reals.  That is, if $U$ is an ultrafilter and $f\colon
X\to\ps(\reals)$ for some $X\in U$ then there is a set $a\subseteq \reals$
such that $\set{x\in X\mid f(x)=a}\in U$.   This implies the needed
instances of the Axiom of Choice:
\begin{proposition}
  \label{thm:enoughAC}
  Suppose that $U$ is an ultrafilter and $X\in U$. Then
  \begin{enumerate}
  \item there is a well orderable  $X'\subseteq X$ such that
    $X'\in U$,
    and
  \item if $f$ is a function such that $\set{x\in X\mid
      f(x)\not=\emptyset}\in U$ then there is a function $g$ such that
    $\set{x\in X\mid g(x)\in f(x)}\in U$.
  \end{enumerate}
\end{proposition}
\begin{proof}
  Every element of $M$ is ordinal definable from a real parameter.  If
  $x\in M$, then let $\phi_x$ be the least formula $\phi$, with ordinal
  parameters, such that $(\exists r\in\reals)\forall
  z\;(\phi(z,r)\iff z=x)$, and let $R_x=\set{r\in\reals\mid\forall
    z\;(\phi_x(z,r)\iff z=x)}$.    For the first clause, there is
  $R\subseteq \reals$ such that $X'=\set{x\in X\mid R_x=R}\in U$.
  Thus, if $r$ is any member of $R$, then every member of $X'$ is
  ordinal definable from $r$.

  The proof of the second clause is similar, using $R\subseteq \reals$
  such that $\set{x\in X\mid \bigcup_{z\in f(x)}R_z = R}$.
\end{proof}

If $M=J_{\alpha}(\reals)[\mathcal{E}]$ is a mouse then we write
$M\cut\gamma$ for $J_{\gamma}(\reals)[\mathcal{E}\restrict\gamma]$,
that is, for the cut off of $M$ at $\gamma$ without including the
active final extender $\mathcal{E}_{\gamma}$ if there is one.    This
is most commonly used as $N\cut\ords$, where   $N$
is the final model of an iteration of length $\ords$ and $\ords^{N}>\ords$.

\section{The Lower bound}
\label{sec:lower}
The proof of Theorem~\ref{thm:main}(1), giving a lower bound to
the large cardinal strength of a sharp for the Chang model, is  a
straightforward application of a technique 
of Gitik (see the proof of Lemma~2.5 for $\delta=\omega$ in \cite{gitik-mitchell.ext-indisc}).

\begin{proof}[Proof of Theorem~\ref{thm:main}(1)]
  The proof of the lower bound   uses iterated ultrapowers to
  compare $\RK$ with $\RK^{\chang}$.  Standard methods show that
  $\RK^{\chang}$ is not moved in this comparison, so  there is an
  iterated ultrapower $\seq{M_\nu\mid \nu\le\theta}$, For some
  $\theta\leq\ords$, such that  $M_0=\RK$ and $M_\theta=\RK^\chang$.  
  This iterated ultrapower is defined by setting
  \begin{myinparaenum}\item
    $M_\alpha=\dirlim\set{M_{\alpha'}\mid \alpha''<\alpha'<\alpha}$ for
    sufficiently large $\alpha''<\alpha$ if $\alpha$ is a  limit
    ordinal, and 
  \item
    $ M_{\alpha+1}=\ult(M^*_{\alpha},E_\alpha)$, where 
    $E_{\alpha}$ is the least extender in $M_{\alpha}$ which is not in
    $\RK^{\chang}$ and $M^*$ is equal to $M_{\alpha}$ unless
    $E_{\alpha}$ is not a full extender in $M_{\alpha}$, in which case
    $M^*_{\alpha}$ is the largest initial segment of $M_\alpha$ in which
    $E_{\alpha}$ is a full extender.
  \end{myinparaenum}

  We want to show that 
  \begin{myinparaenum}
  \item 
    this does not drop, that is,  $M^{*}_{\alpha}=M_{\alpha}$ for all $\alpha$, and 
  \item 
    $M_{\theta}=\RK^{\chang}$. 
  \end{myinparaenum}
  
  If either of these is false, then $\theta=\ords$
  and there is a closed unbounded class $C$ of ordinals $\alpha$ such that
  $\crit(E_\alpha)=\alpha=i_{\alpha}(\alpha)$. Since $o(\kappa)<\ords$ for all $\kappa$
  it follows that there is a stationary class $S\subseteq C$ of ordinals of
  cofinality $\omega$ such that $i_{\alpha',\alpha}(E_{\alpha'})=E_{\alpha}$ for all
  $\alpha'<\alpha$ in $S$.  Fix $\alpha\in S\cap\lim(S)$; we will
  show that  the hypothesis of
  Theorem~\ref{thm:main}(1) implies that $E_\alpha\in\chang$, contradicting the choice of $E_\alpha$.

  To this end, let  $\vec\alpha=\seq{\alpha_n\mid n\in\omega}$ be an increasing sequence of ordinals in $S$ such that $\bigcup_{n\in\omega}\alpha_n=\alpha$. We call a sequence $\seq{\beta_n\mid n\in\omega}$ a \emph{thread} for the generator $\beta$ of $E_\alpha$ if $\beta_n=i^{-1}_{\alpha_n,\alpha}(\beta)$ for all sufficiently large $n<\omega$.  
  The technique of Gitik used in \cite[Lemma~2.3]{gitik-mitchell.ext-indisc}  gives a formula $\phi$ such that $\phi(\vec\alpha,\vec\beta,\beta)$ holds if and only 
  if $\beta<\kappa^{+\omega}$ and $\vec\beta$ is a thread for $\beta$.
  Since all of the threads are in $\chang$, this implies that
  $E_\alpha\restrict\kappa^{+\omega}\in\chang$. If
  $\gamma=\len(E_\alpha)<\kappa_\alpha^{+(\omega+1)}$ then this
  construction can be extended to all of $E_\alpha$ by using
  $\seq{i^{-1}_{\alpha_n}(\len(E_\alpha))\mid n\in\omega}$ as an additional
  parameter.  
  But the hypothesis of Theorem~\ref{thm:main}(\ref{item:main-lower}) implies that
  $\len(E_\alpha)<(\kappa_\alpha^{+\omega})^\chang$, so
  $E_\alpha\in\chang$, contradicting the definition of $E_\alpha$. 
  
  It follows that no sharp for $\chang$ exists, as otherwise the embedding given
  by Proposition~\ref{thm:sharpEmbed}  would make an iterated ultrapower of $\RK$
  non-rigid. 
\end{proof}

\section{The upper bound}
\label{sec:upperbound}
The proof of Theorem~\ref{thm:main}(\ref{item:main-upper}) will take
up the rest of this paper except for the final 
Section~\ref{sec:questions}, which poses some open questions.

The hypothesis of Theorem~\ref{thm:main}(\ref{item:main-upper}) is stronger than
necessary:   our construction of the sharp for $\chang$  uses only a
sufficiently strong mouse  over the
reals, that is, a model  $M=J_{\gamma}(\reals)[\mathcal{E}]$ 
where $\mathcal{E}$ is an iterable extender sequence.  

At this point we describe a general procedure for constructing a  sharp from a mouse. 
For this purpose we will assume that $M$ is a mouse satisfying the
following conditions: 
\begin{myinparaenum}
\item $\card M=\card\reals$, definably over $M$, indeed
\item
  there is an onto function $h\colon \reals\to M$ which is the
  union of an increasing $\omega_1$ sequence of functions in  $M$, and 
\item $M$ has a last $(\kappa,\kappa^{+\omega_1})$-extender, $E\in M$.
\end{myinparaenum}
We can easily find such a mouse from the hypothesis of
Theorem~\ref{thm:main}(\ref{item:main-upper}) by choosing a model $N$ of
the form   $J_{\gamma}(\reals)[\mathcal{E}]$ with the last two properties and
letting $M$ be  the transitive collapse of the  Skolem hull of
$\reals\cup\omega_1$ in $N$.
In Definition~\ref{def:Nsequence},  at the start of
section~\ref{sec:mainlemma}, 
we will make additional and
more precise  assumptions on $M$ which are used in the proof of the
Main Theorem.

We remark that we could assume the Continuum Hypothesis by 
generically adding a map $g$ mapping $\omega_1$ onto the reals.
Doing so would not add any new countable sequences and hence would not
affect the Chang model.   Indeed we could use
$J_{\gamma}[g][\mathcal{E}]$ for the mouse $M$ instead of
$J_{\gamma}(\reals)[\mathcal{E}]$, so that $M$ satisfies the Axiom of
Choice and the Continuum Hypothesis, along with all of the properties we require of $M$.
We do not do so (though we will need to generically add such a
map $g$ near the end of the proof) but the reader certainly may, if
desired, assume that this has been done.

The following simple observation is basic to the construction:
\begin{proposition}\label{thm:Mcc}
  The mouse $M$ is closed
  under countable subsequences.
\end{proposition}
\begin{proof}
  By the assumption (b) on $M$, any countable subset $B\subseteq M$ is
  equal to $h[b]$ for a function $h\in M$ and set $b\subset\reals$.
  Since $M$ contains all reals, and
  any countable set of reals can be coded by a single real, $b\in M$
  and thus $B\in M$.
\end{proof}

As in the case of $0^{\sharp}$, we obtain the sharp for the Chang model by iterating the final extender $E$ out of the universe:

\begin{definition}\label{def:IFromM}
  We write $i_{\alpha}\colon M_0=M\to M_{\alpha}=\ult_{\alpha}(M,E)$.  
  In particular $M_{\ords}$ is the result of  iterating $E$ out of the
  universe, so that $i_{\ords}(\kappa)=\ords$.

  Let $\kappa=\crit(E)$.  We write
  $\kappa_{\nu}=i_{\nu}(\kappa)$ 
  and $I=\set{\kappa_\nu\mid\nu\in\ords}$.      We say that an
  ordinal $\beta$ is a  \emph{generator
  belonging to $\kappa_{\nu}$} if $\beta=i_{\nu}(\bar\beta)$ for some 
  $\bar{\beta}\in[\kappa,\kappa^{+\omega_1})$
\end{definition}
Note that the set of generators belonging to $\kappa_\nu$ is a subset
of $\supp(i_{\nu}(E))$, that is, it is a set of generators for the
extender $i_{\nu}(E)$ on $\kappa_\nu$ in $M_\nu$.
 Every member of $M_{\ords}$ is equal to 
$i_{\ords}(f)(\vec \beta)$ for some function $f\in M$ with domain
$\kappa^{\card{\beta}}$ and some finite sequence $\vec \beta$  of
generators for members of $I$.
The following observation follows from this fact together with
Proposition~\ref{thm:Mcc}:

\begin{proposition}\label{thm:countable-sets-generators}
  Suppose that $N\supseteq M_{\ords}\cut\ords$ is a model of set
  theory which contains 
  all countable sets of generators.  Then $\chang^{N}=\chang$.
\end{proposition}
\begin{proof}
  It is sufficient to show that $N$ contains all countable sets of
  ordinals, but that is immediate since every countable set  $B$ of
  ordinals has the form
  $B=\set{i_{\ords}(f_{n})(\vec\beta_n)\mid n\in\omega}$, where each
  $f_n$ is a function in $M$ and each $\vec\beta_n$ is a finite
  sequence of generators.  Since the sequence $\seq{f_n\mid n\in\omega}$ is
  in $M\subseteq N$ by Proposition~\ref{thm:Mcc}, the sequence
  $\seq{i_{\ords}(f_n)\restrict\lambda\mid n\in\omega}\in
  M_{\ords}\cut\ords\subseteq N$ for
  $\lambda>\sup\bigcup_{n\in\omega}\vec\beta_n$,  and the sequence
  $\seq{\vec\beta_n\mid n\in\omega}$ is in $N$ by assumption.  Thus $B\in
  N$. 
\end{proof}

Clearly the class $I$ gives a sharp for the model $M_\ords\cut\ords$
in the sense of Definition~\ref{def:Csharp} (with suitable sequences
from $I$
replaced by finite sequences), but it is not at all clear that  $I$ 
gives a sharp for $\chang$ as well.
We show starting in
Section~\ref{sec:proof-start} that it does give a sharp when defined
using the mouse specified there.
\begin{conjecture}\label{thm:optimalmouseconjecture}
  If $M$ is the minimal mouse for which this procedure yields a sharp
  for $\chang$, then the core model $K(\reals)^{\chang}$ of the Chang
  model is given by 
  an iteration $k$, without drops,  of $M_{\ords}\cut\ords$.
\end{conjecture}
This mouse $M$ (which we will refer to as the ``optimal'' mouse) would
then give ``the'' sharp for $\chang$.   A verification of this
conjecture would presumably determine the correct large cardinal strength of the sharp, and
remove some of the weaknesses which have been remarked on in our results.

\subsection{Why is suitability required?}
\label{sec:suitability-required}
Two major weaknesses of the results of this
paper were pointed out earlier:  the need for restricted formulas and suitable sequences.   We expressed the hope that the need for restricted
formulas will be eliminated by strengthening these results to use
the minimal mouse.   In this subsection we make a brief digression to
look at the question of suitability.   Nothing in this subsection is
required for the proof of
Theorem~\ref{thm:main}(\ref{item:main-upper}) and nothing in this
subsection will be referred to again except for the statement of
Theorem~\ref{thm:modified-suitable}.

Say that a mouse $M$ is \emph{correct for the Chang model} if there is
an iteration $k\colon M_{\ords}\to\RK^{\chang}$, without drops, such
that $k[\kappa_{\nu}]\subset\kappa_{\nu}$ for all $\nu\in\ords$ and
$k(\kappa_\nu)>\kappa_\nu$ for all $\nu\in\ords$ of uncountable
cofinality.

Such a mouse must be the minimal mouse which is not a member of  $\chang$, since
otherwise the minimal such mouse would be a member of $M$ and the
iteration $k$ would either drop or go beyond $\ords$.  The converse is
not known, but it seems probable that the minimal mouse is correct and that $i\restrict
I=\set{(\kappa_\nu,k(\kappa_\nu)\mid\nu\in\ords}  $ is a 
class of indiscernibles for $\chang$.

Now suppose that $M$ is correct for $\chang$, and  
say that a sequence $\vec\alpha$ is Prikry for $\vec \beta$ if each is an
increasing $\omega$ sequence and there is a sequence of measures $U_n\in M_{\ords}$ on
$\beta_n$  such that $\vec\alpha$ satisfies the Mathias genericity
condition: for all $x\subseteq\sup(\vec\beta)$
in $M_{\ords}$, 
for all but finitely many  $n\in\omega$, we have $\alpha_n\in x$ if and
only if $x\cap\beta_n\in U_n$.

Note that we are not asserting here that $\vec\alpha$ is actually generic
over $M_{\ords}$, as neither $\vec\beta$ nor the sequence of measures
need be in $M_\ords$.

We write $\vec\lambda<^*\vec \eta$ if $\lambda_n<\eta_n$ for all but
finitely many $n$.

\begin{proposition}
    Suppose $\vec\nu$ and $\vec\mu$ are 
    increasing $\omega$-sequences of ordinals with
    $\vec\nu<^*\vec\mu$ and $\sup(\vec{\nu})=sup(\vec\mu)$.   Then
    $\seq{k(\kappa_{\nu_n})\mid n\in\omega}$ and 
    $\seq{\kappa_{\nu_n+1}\mid n\in\omega}$ are each Prikry for
    $\seq{k(\kappa_{\mu_n})\mid n\in\omega}$.  Furthermore,
    no sequence $\vec\alpha$ in the interval 
    $\seq{k(\kappa_{\nu_n})\mid n\in\omega}
    <^*\vec\alpha<^*\seq{\kappa_{\nu_n+1}\mid n \in\omega}$ 
    is Prikry for     $\seq{k(\kappa_{\mu_n})\mid n\in\omega} $.
\end{proposition}
\begin{proof}
  To see that  $\seq{\kappa_{\nu_n+1}\mid n\in\omega}$ is Prikry for
  $\seq{k(\kappa_{\mu_n})\mid n\in\omega}$, use $U_n=k\circ
    i_{\nu_n+1,\mu_n}(U'_n)$ where
    $U'_n=\set{x\subseteq\kappa_{\nu_n+1}\mid \kappa_{\xi_n+1}\in k(x)}$.
    To see that $\seq{k(\kappa_{\nu_n})\mid n\in\omega}$ is Prikry for
    $\seq{k(\kappa_{\mu_n})\mid n\in\omega}$, use  $U_n=k\circ
    i_{\mu_n}(\ufFromExt{E}{\kappa})$. 

    For the final sentence, observe that $\seq{k\circ
      i_{\ords}(f)(k(\kappa_{\nu}))\mid f\in M}$ is cofinal in
    $\kappa_{\nu+1}$ for all $\nu\in\ords$.   It follows that if 
    $\seq{k(\kappa_{\nu_n})\mid n\in\omega}
    <^*\vec\alpha<^*\seq{\kappa_{\nu_n+1}\mid n  \in\omega}$
    then there is a function $f\in M$ such that
    $k\circ i_{\ords}(f)(k(\kappa_{\nu_n}))>\alpha_n$ for all $n\in\omega$ such
    that $\alpha_n< \kappa_{\nu_n+1}$, so
    $x=\set{\nu\mid(\exists\nu'<\nu) \;k\circ i_{\ords}(f)(\nu')>\nu}$ witnesses that
    $\vec\alpha$ is not Prikry for $\seq{\kappa_{\mu_n}\mid n\in\omega}$.
\end{proof}
\begin{corollary}
  \label{thm:counterexample}
  Suppose that $B$ and $B'$ are two countable closed subsets of $I$
  such that for all formulas $\phi$
  of set theory (with no extra terms) $\chang\models\phi(k\restrict
  B)\iff\phi(k\restrict B')$.
  
  Then, writing $B=\seq{\kappa_{\nu_{\xi}}\mid\xi<\alpha}$ and
  $B'=\seq{\kappa_{\nu'_{\xi}}\mid\xi<\alpha'}$, we have
  $\alpha=\alpha'$,   $(\forall\xi<\alpha)\;(\cof(\nu_\xi)=\omega\iff\cof(\nu'_{\xi})=\omega)$,
  and 
  for all but finitely many $\xi<\alpha$
  \begin{enumerate}
  \item $\nu_{\xi+1}=\nu_{\xi}+1$ if and only if
    $\nu'_{\xi+1}=\nu'_{\xi}+1$, and 
  \item $\nu_{\xi}$ is a limit ordinal if and only if $\nu'_{\xi}$
    is
    a limit ordinal.
  \end{enumerate}
\end{corollary}
\begin{proof}
  Only the two numbered assertions are problematic.   For the first
  assertion, suppose to the contrary that $\seq{\xi_n\mid n\in\omega}$
  is an increasing subsequence of $\alpha$ such that
  $\nu_{{\xi_n}+1}=\nu_{{\xi_n}}+1$ but
  $\nu'_{{\xi_{n}}+1}>\nu'_{\xi_n}+1$.
  Let  $\phi(k\restrict B)$ be the formula asserting that there is no
  sequence $\vec\alpha$ which is Prikry for
  $\seq{k(\kappa_{\nu_{\xi_n}+1})\mid n\in\omega}$ such that 
  $\seq{\kappa_{\nu_{\xi_n}}\mid
    n\in\omega}<^*\vec\alpha<^*\seq{\kappa_{\nu_{\xi_n+1}}\mid n\in\omega}$ for each
  $n\in\omega$.   Then $\phi$
  is  true of $B$ but false of $B'$.

  For the second assertion, observe that 
  $\nu_{\xi_n}$ is a limit ordinal for all but finitely many
  $n\in\omega$ if and only if 
  there are $<^*$-cofinally many
  sequences $\vec\alpha<^* \seq{\kappa_{\nu_{\xi_n}}\mid n\in\omega}$
  which are Prikry for $\seq{k(\kappa_{\nu_{\xi_n}})\mid n\in\omega}$.
\end{proof}

On its face this Corollary is vacuous: it applies only to (and only
conjecturally to) the optimal sharp for the Chang model, which itself
only conjecturally exists.   However it is an important  motivation
for the technique we use to prove the Main Theorem and  gives important
information about the structure of the sharp of the Chang model.
First, the \emph{gaps} in a sequence $B$, that is, the maximal
intervals of $I\setminus B$, are important.    Second, (assuming as we
do that no gaps have a least upper bound of cofinality $\omega$) the
only important characteristic of the gaps is whether their upper bound
is a limit point or a successor point of $I$.   Finally,  individual
gaps are not important---only 
infinite sets of gaps.

Indeed, in Subsection~\ref{sec:finite-exceptions} we will outline a
proof of Theorem~\ref{thm:modified-suitable} below, which strengthens
Theorem~\ref{thm:main}(\ref{item:main-upper}) to show that the class
$I$ of indiscernibles of given by the proof of that theorem satisfies
the converse of the conclusion  of
Corollary~\ref{thm:counterexample}.%

\begin{definition}
  \label{def:weaklySuitable}
  Call  a sequence $B\subseteq I$ \emph{weakly suitable} if  $B$
  is a countable and closed, and 
  $B\cap\lambda$ is unbounded in $\lambda$ whenever $\lambda\in B$
  and $\cof(\lambda)=\omega$.

  Suppose that $B=\seq{\lambda_\nu\mid \nu<\alpha}$ and
  $B'=\seq{\lambda'_\nu\mid \nu<\alpha'}$, enumerated in increasing
  order,  are weakly suitable.    We say that $B$ and $B'$ are 
  \emph{equivalent} if $\alpha=\alpha'$,
  $(\forall\nu<\alpha)\;(\cof(\lambda_\nu)=\omega\iff\cof(\lambda'_\nu)=\omega)$,
  and with at most finitely many exceptions the following hold for all $\nu<\alpha$:
  \begin{myinparaenum}
  \item $\lambda_{\nu_+1}=\min(I\setminus\lambda_\nu+1)$ if and only
    if $\lambda'_{\nu+1}=\min(I\setminus\lambda_{\nu}+1)$, and
  \item $\lambda_\nu$ is a limit member of $I$ if and only if
    $\lambda'_\nu$ is a limit member of $I$.
  \end{myinparaenum}
\end{definition}
\begin{theorem}\label{thm:modified-suitable}
  If $B$ and $B'$ are equivalent weakly suitable sequences then
  $\chang\models\phi(B)\iff\phi(B')$ for any restricted formula $\phi$
  in our language.
\end{theorem}

\subsection{Definition of the set $T$ of terms.}
The next definition gives the set of terms we will use to construct the sharp.
This list should be regarded as preliminary, as a better understanding
of the Chang model will undoubtedly suggest a more felicitous choice.

\begin{definition}
  \label{def:Istardef}\label{def:terms}
  The members of the set $T$ of terms of our language for the sharp of $\chang$ are those 
  obtained by compositions
  of the following set of basic terms:
  \begin{enumerate}
  \item \label{item:CfType}For each function $f\colon {^n\kappa}\to
    \kappa$ in $M$ for some $n\in\omega$,
    there is a term $\tau$ such that $\tau(z)=i_{\ords}(f)(z)$  for
    all $z\in {^n\ords}$.
  \item\label{item:CindiscBelong} For each $\bar\beta$ in the interval
    $\kappa\leq\bar\beta<(\kappa^{+\omega_1})^{M}$ there is a term $\tau$ such
    that $\tau(\kappa_{\nu}) = i_{\nu}(\bar\beta)$ for all $\nu\in\ords$.

  \item\label{item:Cinfinite} Suppose 
    $\seq{\tau_n\mid n\in\omega}$  is an $\omega$-sequence of compositions of terms from the previous two
    cases, and $\domain(\tau_n)\subseteq{^{k_n}\Omega_n}$.   Then
    there is a term $\tau$ such that $\tau(\vec a)=\seq{\tau_n(\vec
      a\restrict k_{n})\mid n\in\omega}$ for all  $\vec 
    a\in{^{\omega}\Omega}$.
  \item \label{item:CsuccType}
    For each formula $\phi$, there is a  term $\tau$
    such that  if $\iota$ is an ordinal and $y$ is a countable sequence
    of terms for members of $\chang_{\iota}$ then
    \begin{equation*}
      \tau(\iota, y)=
      \xset{x\in    \chang_{\iota}}{\chang_{\iota}\models
        \phi(x, y)}.
    \end{equation*}
  \end{enumerate}
\end{definition}

\begin{proposition}
  \label{thm:terms-suffice}
  For each $z\in\chang$ there is a term $\tau\in M$ and a suitable
  sequence $B$ such that $\tau(B)=z$.
\end{proposition}
\begin{proof}
  First we observe that any ordinal $\nu$ can be written in the form
  $\nu=i_{\ords}(f)(\vec \beta)$ for some $f\in M$ and finite sequence
  $\vec \beta$ of generators.     Each generator $\beta$ belonging to
  some $\kappa_\xi\in I$ is equal to $i_{\xi}(\bar \beta)$ for some
  $\bar\beta\in\left[\kappa,(\kappa^{+\omega_1})^{M}\right)$,
  and thus is denoted by a term $\tau(\kappa_\xi)$
  built from
  clause~\ref{item:CindiscBelong}.    Thus any finite sequence of
  ordinals is denoted by an expression using terms of
  type~\ref{item:CfType} and~\ref{item:CindiscBelong}.   Since $M$ is closed under
  countable sequences, adding terms of type~\ref{item:Cinfinite} adds
  in all  countable sequences of ordinals.

  Finally, any set $x\in\chang$ has the form
  $\set{x\in\chang_{\iota}\mid \chang_{\iota}\models \phi(x,y)}$ for some
  $\iota,\phi$ and $y$ as in clause~\ref{item:CsuccType}.    Thus a
  simple recursion on $\iota$ shows that every member of $\chang$ is
  denoted by a term from clause~\ref{item:CsuccType}.
\end{proof}

The terms of clause~\ref{item:CindiscBelong} force the limitation 
to restricted formulas in
Theorem~\ref{thm:main}(\ref{item:main-upper}), since the domain of these terms is exactly the
class $I$ of indiscernibles.  It is possible that a more natural set
of terms would enable this restriction to be removed, but this would
depend on a precise understanding of the iteration $k\colon M_\ords\to
K(\reals)^{\chang}$  from Subsection~\ref{sec:suitability-required}.

Proposition~\ref{thm:terms-suffice} actually exposes a probable
weakness in our current state of understanding of the Chang model.   This proposition
corresponds to the property of  $0^{\sharp}$ that every ordinal $\alpha$
is definable is using as parameters  members of the class $I$ of
indiscernibles.    In the case of $0^{\sharp}$ this is only true if
the parameters are allowed to include members of $I\setminus\alpha+1$.
In contrast, Proposition~\ref{thm:terms-suffice} says that $\alpha$
is alway denoted by a term $\tau(B)$ with
$B\in[I\cap(\alpha+1)]^{\omega}$.    Possibly a more polished set of
terms, obtained  through a more careful
analysis of  the fine structure of the models and the iteration
$k$,  would yield definability
properties more like those of  $0^{\sharp}$. 

\subsection{Outline of the proof}
\label{sec:proof-start}

Proposition~\ref{thm:countable-sets-generators} suggests  a possible
strategy for the proof of
Theorem~\ref{thm:main}(\ref{item:main-upper}):  find  a generic
extension of $M_{\ords}\cut\ords$ which  contains all 
countable sequences of generators.     There are good reasons
why this is likely to be impossible, beginning with the problem of
actually constructing a generic set for a class sized model.  
Beyond that, many of the known forcing constructions used to add countable sequences of
ordinals require large cardinal strength far stronger than that
assumed in the hypothesis of Theorem~\ref{thm:main}, and give models with
properties which are known to imply the existence of submodels having strong
large cardinal strength.
However, two considerations suggest that this last problem may be less
serious than it first  appears.    First,  there can be  much more large cardinal strength in
the Chang model than is apparent from the actual extenders present in $\RK^{\chang}$, 
since much of the large cardinal strength in $V$ is
encoded in the set of reals.      Second, 
many properties known to imply large cardinal strength
are false in the Chang model not because of the lack of such
strength, but because of the failure of the
Axiom of Choice.
Results involving the size of the power set of singular cardinals, for
example, are irrelevant to the Chang model since the power set is not
(typically) well ordered there.

We avoid the problem of constructing generic extensions for class sized
model by working with submodels generated by countable subsets of $I$,
and we find that in fact none of the large cardinal structure in $V$
survives the passage to the Chang model beyond that given in the
hypothesis to Theorem~\ref{thm:main}.

\begin{definition}
  If $B\subseteq I$ and $\gen_B$ is the set of generators belonging to members of $B$ then we write
  \begin{equation*}
    M_{B} = \{i_{\Omega}(f)(b)\mid  f\in M \land b\in[\gen_B]^{<\omega}\}.
  \end{equation*}
  
  If $B$ is closed, and in particular if it is suitable, then we write $\chang_{B}$ for the Chang model 
  evaluated using the ordinals of $M_B\cut\ords$ and all countable
  sequences of these ordinals.
\end{definition}

Note
that $M_B$ is not transitive: it is a submodel of $M_{\Omega}$, and 
$i_{\ords}: M\to M_\ords$ is  
the canonical embedding  $M\to M_B$ for any $B\subseteq I$.  It is not
obvious even that the model $\chang_B$ can be regarded as a subset of
$\chang$; the proof of this is a part of the proof of the main lemma. 
The definition of $\chang_{B}$ does imply that if $B$ and $B'$ are
closed subsets of $I$ with the same
order type then $\chang_{B}\cong
\chang_{B'}$.   In particular, if $\otp(B)=\alpha+1$ then, setting
$B(\alpha+1)=\set{\kappa_\nu\mid \nu<\alpha+1}$, 
$\chang_{B}\cong \chang_{B(\alpha+1)}$, which in turn is equal to the
$\kappa_{\alpha+1}$st stage $\chang_{\kappa_{\alpha+1}}$ of the
recursive definition of the Chang model as stated at the beginning of
this paper.

The motivation for our work begins with the observation that  $M_{B}\cut
\ords\prec M_{B'}\cut\ords\prec M_{\ords}\cut\ords$ whenever 
$B\subseteq B'\subseteq I$.   
Corollary ~\ref{thm:counterexample} refutes any suggestion that this 
necessarily extends to the models $\chang_B$ and $\chang_{B'}$,
however it also  motivates
Definition~\ref{def:limitsuitable} below.  

Corollary~\ref{thm:counterexample}  says that we must take account of the gaps in $B$.
To be precise, we will say that a \emph{gap} in $B$ is a maximal 
nonempty 
interval in $I\setminus B$.    For $B$ either suitable or \LS{},
every gap in $B$ is  headed by a limit point $\lambda$ of 
$I$  which is a member of $B\cup\sing{\ords}$ and has uncountable cofinality.

\begin{definition}
  \label{def:limitsuitable}
  A subset $B$ of $I$ is \emph{\LS} if
  \begin{myinparaenum}
  \item\label{item:LS-suitable} its closure $\bar B$ is suitable,  and 
    \label{item:LS-gaps} every gap in $B$ is an interval of the
    form $[\lambda,\delta)$ where
  \item $\delta$ is either $\ords$ or a member of $B$ which is a limit
    point of $I$ of uncountable cofinality, 
  \item if $\lambda\not=\emptyset$, then $\lambda=\sup(\sing{0}\cup
    B\cap\delta)$, and  
  \item\label{item:LS-top}
    $\lambda=\kappa_{\nu+\omega}$ for some $\nu\in\ords$.
  \end{myinparaenum}

  Two \LS{} sets $B$ and $B'$ are said to be \emph{equivalent} if they have
  the same order type and they have gaps in the same locations.
  For a \LS\  set $B$, which is never closed (except for
  $B=\emptyset$), we write
  $\chang_{B}=\bigcup\set{\chang_{B'}\mid B'\subset B\land B'\text{ is suitable}}$. 
  That is, for limit suitable sets $B$ the model is constructed, like
  $\chang_B$ for suitable $B$, by construction over the
  (nontransitive) set of ordinals of $M_B$, but using only those
  countable sets of ordinals which are in $\chang_{B'}$ for some
  suitable $B'\subset B$. 
\end{definition}

The use of $\kappa_{\nu+\omega}$ in the final Clause~(\ref{item:LS-top}) is for
convenience: our arguments would still be valid if it were only required
that $\lambda$ be a limit member of $I$ of countable cofinality which
is not a member of $B$.

Note that if $B$ is a \LS\ sequence then $\chang_{B}$ is not closed
under countable sequences; in particular $B$ is not a member of $\chang_B$.
Thus if $\delta$ is the head of a gap of $B$ then $\chang_{B}$
believes (correctly)  that $\delta$  has
uncountable cofinality.

Theorem~\ref{thm:main}(\ref{item:main-upper}) will follow  from the following
lemma: 

\begin{lemma}[Main Lemma]
  \label{thm:mainlemma}
  Suppose $B\subset I$ is \LS. Then 
  $\chang_{B}$ is isomorphic to an elementary substructure of $\chang$
  via the map defined by
  $\tau^{\chang_B}(\vec\beta)\mapsto\tau^\chang(\vec\beta)$ for any
  term $\tau\in T$ and any $\vec \beta$ which is a countable sequence of
  generators for members of some suitable $B'\subset B$. 
\end{lemma}
The elementarity holds for all restricted formulas.  The proof will be by an induction over pairs 
$(\iota,\phi)$, where $\iota\in M_B\cap \ords+1$, and $\phi$ is a
formula of set theory; and  the induction hypothesis implies that
the map 
\begin{equation*}
  \set{z\in\chang_\iota^{\chang_B}\mid \chang_\iota^{\chang_B}\models\phi(z,\vec\beta)}\mapsto
  \set{z\in\chang_\iota\mid\chang_\iota\models\phi(z,\vec\beta)}
\end{equation*}
is well defined. 
To see that Lemma~\ref{thm:mainlemma} suffices to prove
Theorem~\ref{thm:main}(2), observe that 
any suitable set $B$ can be extended to a \LS\ set defined by the equation
\begin{equation*}
  B'=B\cup\set{\kappa_{\nu+n}\mid \kappa_{\nu}\in B\land n\in\omega}, 
\end{equation*}
that is, by by adding the next $\omega$-sequence from $I$ at the foot
of each gap of $B$ and to the top of $B$.  Now
let $B_0$ and $B_1$ be two equivalent suitable sets.  Then their \LS\
extensions $B'_0$ and $B'_1$ are also equivalent, having the same
order type and having gaps in the corresponding places, so 
$\chang_{B'_0}\cong \chang_{B'_1}$.    Then for any restricted formula
$\phi$ we have
\begin{align*}
  \chang\models \phi(B_0)&\iff \chang_{B'_0}\models\phi(B_0)\\
                         &\iff \chang_{B'_1}\models\phi(B_1)\iff \chang\models\phi(B_1).
\end{align*}

\section{The Proof of the Main Lemma}
\label{sec:mainlemma}

At this point we fix a mouse $M$ to be used for 
the proof of the Main Lemma~\ref{thm:mainlemma}.   Some basic
properties of $M$ have already been 
sketched at the start of Section~\ref{sec:upperbound}, and
Definition~\ref{def:Nsequence} below gives  more specific
requirements.

For this section, $B\subseteq I$ is a limit suitable sequence and  $\zeta=\otp(B)$.
The main tool used for the proof is the forcing $P(\vec
E\restrict\zeta)\mgkeq$, to be  defined inside $M$, and a
$M_B$-generic set $G\subseteq i_{\ords}(P(\vec
E\restrict\zeta)\mgkeq)$ to be constructed inside $V[h]$ for a 
generic Levy collapse map  
$h\colon\omega_1\cong \reals$.    The model $M_B[G]$ will include all
its countable subsets, and $\chang_{B}$ will be definable as a
submodel of $M_B[G]$.

The
forcing  is essentially due to Gitik (see, for example,
\cite{Gitik2002Blowing-up-powe}) and the technique for constructing
the $M_B$-generic set $G$ is from Carmi Merimovich
\cite{Merimovich2007Prikry-on-exten}.
Gitik's forcing was designed to make the Singular Cardinal Hypothesis
fail at a cardinal of cofinality 
$\omega$ by adding many Prikry sequences, each of which is (in our
context) a sequence of generators  for cardinals in  $B$.    Thus it
would do what we need for the case when  
$\otp(B)=\omega$,  but  needs to be adapted to work for sequences $B$
of arbitrary countable length.   To this end we modify
Gitik's forcing by using  ideas introduced by Magidor  in
\cite{magidor.changecf} to adapt
Prikry forcing in order to to add sequences of indiscernibles of
length longer  
than $\omega$.   This adds some complications to Gitik's forcing, but
on the other hand much  of the complication of Gitik's
work is avoided since we do not need to know whether cardinals in
the interval $(\kappa^{+}, \kappa^{+\omega_1})$ are collapsed,  and
hence  we can omit his preliminary forcing.

Our forcing is based on a sequence $\vec E$ of extenders, derived from the
last extender $E$ of $M$.   We begin by defining this sequence, and at the
same time  specify  what properties we require of the chosen mouse $M$.

\begin{definition}\label{def:Nsequence}
  We  define an increasing sequence, $\seq{N_\nu\mid \nu<\omega_1}$
  of submodels of $M$.
  We write $E_{\nu}$ for $E\ecut N_{\nu}$,
  the restriction of $E$ to the ordinals in $N_{\nu}$,  
  we write $\pi_{\nu}\colon \bar N_\nu\to N_\nu$ for the Mostowski
  collapse of $N_\nu$, and we write $\bar E_\nu$ for
  $\pi_{\nu}^{-1}[E_\nu]=\pi_{\nu}^{-1}(E)\ecut\bar N_\nu$.

  We require that the $\reals$-mouse $M$ and the sequence
  $\seq{N_\nu\mid \nu<\omega_1}$   satisfy the following conditions: 
  \begin{compactenum}
  \item $M$ is a model of Zermelo set theory such that
    $\reals\subset M$, 
    $\card{M}=\card{\reals}$, and 
    $\cof(\ords\cap M)=\omega_1$. 
  \item  $\len(E)=(\kappa^{+\omega_1})^{M}$. 
  \item If $\nu'<\nu<\omega_1$ then $(N_{\nu'},E_{\nu'})\prec (N_{\nu},E_{\nu})\prec
    (M,E)$. 
  \item $^{\kappa}N_{\nu}\cap{ M}\subseteq N_{\nu}$.

  \item\label{item:cardNnuSubsetNnu}
    $\card {\bar N_\nu}^{M}\subset N_\nu$.
  \item\label{item:Nseq-doublepluss} $\card {\bar
      N_0}^{M}=(\kappa^{++})^{M}$, and if $\nu>0$ then
    $\card {\bar N_\nu}^{M}=\sup_{\nu'<\nu}(\card{\bar N_{\nu'}}^{++})^{M}$.
  \item \label{item:Eseq-MIsUnionN}
    $M=\bigcup_{\nu<\omega_1}N_{\nu}$.
  \end{compactenum}
\end{definition}

Clauses~5 and~6 are needed for the proof of
Proposition~\ref{thm:limitSuitableC_Bdefinable}. 

We will work primarily with the extenders $E_{\nu}$ rather than with their
collapses $\bar E_{\nu}$, because this makes  it easier to keep track
of the generators.   However it should be noted that  $E_{\nu}$ may
not be a member of $\ult(M,E)$, 
so further justification is needed for  many of the claims we
wish to make about being able to 
carry out constructions inside $M$.     Since we never actually use
more than countably many of the extenders $E_{\nu}$ at any one time,
the following observation will provide such justification:

\begin{proposition}\label{thm:EnuinM}  The following are all members
  of $\ult(M,E_{\nu})$, for any $\nu<\omega_1$:
  \begin{itemize}
   \item \label{item:Nseq-powerset}$\ps(\bigcup_{\nu'<\nu}\bar     N_{\nu'})$
   \item the extender $\bar E_{\nu'}$, and the map $\pi_{\nu''}^{-1}\circ \pi_{\nu'}\colon
     \supp(\bar E_{\nu'})\to\supp(\bar E_{\nu''})$, for each $\nu'<\nu''<\nu$
   \item the direct limit of the set $\set{\supp(\bar E_{\nu'})\mid
       \nu'<\nu''<\nu}$ along the maps $\pi_{\nu''}^{-1}\circ
       \pi_{\nu'}$, as well as with the injection maps from
       $\supp(E_{\nu'})$ into this direct limit
  \end{itemize}\qed
\end{proposition}
Since $\ult(M,E_{\nu})=\ult(M,\bar E_{\nu})$, this proposition allows
us to regard the direct limit as a code inside $M$ for the extender
$E_{\nu}$ together with its system of
subextenders $E_{\nu'}$ for $\nu'<\nu$.

The hypothesis of Theorem~\ref{thm:main} is more than sufficient to
find a 
mouse $M$ and sequence $\vec N$ of submodels satisfying
Definition~\ref{def:Nsequence}: this can be done  by first defining models $M'$ and
$\seq{N'_{\nu}\mid 
  \nu<\omega_1}$ satisfying all of the conditions except
Clause~\ref{item:Eseq-MIsUnionN}, and then taking $M$ to be the
transitive collapse of $\bigcup_{\nu<\omega_1} N'_{\nu}$.    The
conditions on $M$ are, in turn, much stronger than is needed to carry out
this construction.  In view of the fact that there is
no clear reason to believe that the actual strength needed is greater
that $o(E)=\kappa^{+(\omega+1)}$, it does not seem useful  to
complicate the argument in order to determine the minimal mouse for which
the present argument works.

\medskip{}

We are now ready to begin the proof of Lemma~\ref{thm:mainlemma}.
Following Gitik we define, in subsections~\ref{sec:absolutelyfinal} and~\ref{sec:PForder}, a Prikry type forcing
$P(\vec F)$ depending on a sequence $\vec F$ of extenders.
Subsections~\ref{sec:PFproperties} and~\ref{sec:prikry} develop the
properties of this forcing, and Subsection~\ref{sec:gkeqDef} describes an equivalence
relation $\gkeq$ on its set of conditions.
Subsection~\ref{sec:generic_set} constructs an $M_B$-generic subset of
$i_{\ords}(P(\vec E\restrict\zeta)\mgkeq)$, and
subsection~\ref{sec:proof-main-lemma} uses this construction to prove
Lemma~\ref{thm:mainlemma} under the additional assumption that
$\kappa=\kappa_0\in B$.   Finally, 
Subsection~\ref{sec:finite-exceptions} deals with the special case
$\kappa\notin B$ and indicates how the same  technique can be used to prove
Theorem~\ref{thm:modified-suitable}. 

\subsection[The main forcing]{The forcing $P(\vec F)$}
\label{sec:absolutelyfinal}

Throughout the definition of the forcing, from
Subsections~\ref{sec:absolutelyfinal} through \ref{sec:gkeqDef},  we work entirely inside  the
mouse $M$; in particular all cardinal calculations are carried out
inside $M$.   We are
interested in defining $P(\vec E\restrict\zeta)$, but for the purposes
of the recursion used in the definition we allow $\vec F$ to be any suitable
sequence of extenders.   We will not give a definition of the
notion of a \emph{suitable 
  sequence} of extenders.   All the sequences used in this section are
suitable: specifically, all of the sequences $\vec E\restrict
{\xi}$ for $\xi<\omega_1$ are suitable,  all of the ultrafilters
$\ufFromExt{E}{\vec E\restrict\xi}=\set{X\subseteq H^{M}_{\kappa}\mid \vec E\restrict\xi \in i^{E}(X)}$
concentrate on suitable sequences, and     furthermore, if $\vec F$ is
suitable then so is  $\vec F\restrict[\gamma_0,\tau)$ for any
$0\leq\gamma_0\leq\tau\leq\len(\vec F)$. 

Before starting the definition of the forcing, we give a brief
discussion of its design, techniques and origin.

The constructed generic extension of $M_B$ will have the form 
$
  M[G]=M[\vec\kappa, \vec h] 
$,
where  
  $\vec \kappa =\seq{\forceKappa _\gamma\mid\gamma\leq\zeta}$
  enumerates $B\cup\sing{\ords}$ and 
  $\vec  h=\seq{h_{\nu,\nu'}\mid \zeta\geq\nu>\nu'}$ 
  is a sequence of functions
  $h_{\nu,\nu'}\colon  [\forceKappa _\nu,\forceKappa _{\nu}^{+})\to
  \forceKappa _{\nu}$.  Each of the functions $h_{\nu,\nu'}$ is,
  individually,  Cohen generic over $M$.

  The purpose of this forcing is to provide what
  we will call ``standard forcing names'' for the generators belonging
  to members of $B$.  Specifically, consider 
  $\ords=\kappa_{\ords}\in M_B$ and suppose
  $\beta=i_{\bar\nu}(\bar\beta)$ is a generator belonging to
  $\kappa_{\bar\nu}=\forceKappa_{\nu}\in B$.   The construction of the
  $M_B$-generic set $G$ will determine  an ordinal 
  $\bar\xi\in[\kappa,\kappa^{+})$ such that
  $\beta=h_{\zeta,\nu}(i_{\ords}(\bar\xi))$, and this will be used as
  a name in $M$, with parameters $\nu$ and $\bar\xi$, for
  the generator $\beta$ in $M_B$.
  Since $M$ 
  is closed under countable sequences, this will give a name
  for any countable
  sequence of generators, and this in turn will give, via
  clause~\ref{item:CsuccType} of Definition~\ref{def:terms}, a name
  for any member of $\chang_{B}$.

  The problem comes from the fact that the forcing $P(\vec E\restrict\zeta)$ only uses
  the extenders $E_{\nu}$ for $\nu<\zeta$.  The raw use of the
  iteration $\seq{i_{\xi}\mid\xi\in\ords}$ would specify that 
  $i_{\ords}(\bar\beta)$, for $\bar\beta\in[\kappa,\kappa^{+})$,
  should be assigned the indiscernibles 
  $\set{i_{\bar\nu}(\bar\beta)\mid
    \kappa_{\bar\nu}=\forceKappa_{\nu}\in B}$; however this would
  establish names only for the generators  $i_{\bar 
    \nu}(\bar\beta)$ such that 
  $\bar\beta\in\bigcup_{\nu<\zeta}\supp(E_{\nu})$.  To get around this
  problem we need to have a way to slip any ordinal 
  $i_{\bar\nu}(\bar \beta)$, for $\kappa_{\bar\nu}=\forceKappa_{\nu}\in
  B$ and 
  $\bar\beta\in[\kappa,\kappa^{+\omega_1})$, 
  into the generic set as a 
  substitute for some $i_{\bar{\nu}}(\bar\beta')$ with $\bar\beta'\in\bigcup_{\nu<\zeta}\supp(E_\nu)$.

  The trick is to design the forcing to disassociate the
  indiscernibles added by the Prikry component of the forcing from
  any particular ordinal for which it is an indiscernible.  
  We follow Gitik
  \cite{Gitik2002Blowing-up-powe,Gitik2005No-bound-for-th,
    Gitik2010Prikry-type-for,Gitik2012Violating-the-s}
  in using three successive stages to do so.  

 The first stage involves mixing Cohen forcing in  with the
 Prikry forcing.      For any apparent indiscernible
 $h_{\gamma,\gamma'}(\xi)=\xi'$ determined by the
 generic set $G$, there are conditions in $G$ which
 assign the value via a Cohen condition as well as conditions which assign
 it via a Prikry condition.
 In particular, there is no function in $M_B[G]$ which
 assigns uniform indiscernibles to any subset of
 $[\kappa_\ords,\kappa_{\ords}^{+\omega_1})$ of size greater than $\ords=\kappa_{\ords}$.

 The second stage involves the use of
 $[\kappa_\ords,\kappa_{\ords}^{+})$ as the domain of
 $h_{\zeta,\nu}$, rather than
 $\bigcup_{\nu<\zeta}\supp(i_{\ords}(E_{\nu}))$.
 This is accomplished by using, in the Prikry component of the forcing,
 functions 
 $a=a^{s,\zeta}_{\zeta,\nu}$ which map  a subset of $[\kappa_{\ords},\kappa_{\ords}^{+})$
 of size $\ords$ into $\supp(i_{\ords}(E_{\nu}))$.  
 The atomic non-direct extension will use a function $a'$, taken from
 a member of the ultrafilter $\ufFromExt{i_\ords(E_\nu)}{a}$. 
 The function $a'$ could be regarded as a Prikry indiscernible for
 $a$; however it will be recorded in the
 extension only via a Cohen condition $f_{a, a'}$
 defined by $f(\xi)=a'(\xi')$, where $\xi'\in\domain(a')$ corresponds
 to $\xi\in\domain(a)$.

 The effect of this is that if $\alpha\in i_{\ords}(\supp(E_0))$ and $s$ is a condition including
 $a^{s,\zeta}_{\zeta,\nu}(\xi)=\alpha$ for each $\nu<\zeta$, then the 
 sequence $\vec\beta=\seq{h_{\zeta,\nu}(\xi)\mid \nu<\zeta}$ in $M_B[G]$ 
 will be a Prikry sequence for the ultrafilter
 $\ufFromExt{i_{\ords}(E_0)}{\alpha}$;  however there will be no
 association, or at least no explicit association, with the ordinal
 $\beta$ as distinguished from
 any other member of $\set{\beta'\in[\kappa_{\ords},\kappa_{\ords}^{+\omega_1})\mid
   \ufFromExt{i_{\ords}(E_0)}{\beta'}=\ufFromExt{i_{\ords}(E_0)}{\beta}}$,
   which will for typical $\beta$ be unbounded in
   $\supp(i_{\ords}(E_\nu))$ for each $\nu\leq\zeta$. 
 
The ambiguity introduced by the second stage allows the third, and
final, stage in the disassociation of the Prikry 
 conditions, via the equivalence relation $\gkeq$  introduced in
 Subsection~\ref{sec:gkeqDef}.   Gitik uses this equivalence relation
 to ensure that the final  forcing has the $\kappa^{++}$-chain condition and
 hence does not collapse $\kappa^{++}$. 
 We do not care whether the cardinals 
 $\forceKappa_{\nu}^{++}$ are collapsed in $M_B[G]$, but we need to use
 the  equivalence relation in order to construct a generic set $G$ which gives
 standard forcing names to  all
 generators $i_{\bar\nu}(\bar\beta)$ belonging to
 $\forceKappa_{\nu}=\kappa_{\bar\nu}\in B$.   
 This may be regarded as a way of making the notions of ``no
 association'' versus ``no explicit
 association'' in the last paragraph more precise.  As an example of a
 non-explicit association, 
 suppose that 
 $\ufFromExt{E}{\beta'}\not=\ufFromExt{E}{\beta}$ for all
 $\beta'<\beta$.  
Then $E_{\beta}$ is necessarily associated with the least of the
Prikry sequences for the ultrafilter $\ufFromExt{E}{_{\beta}}$.
Thus, in this case, the
 association, though not explicit, is unavoidable.
 The equivalence relation $\gkeq$ will allow us to determine, for any 
 ordinal $\bar\beta\in
 [\kappa,\kappa^{+\omega_1})$,
 sequences $\seq{\bar\beta_{\nu}\mid\nu<\zeta}$ with
 $\bar\beta_\nu\in\supp(E_\nu)$ such that the Prikry
 sequence $\seq{i_{\nu}(\bar\beta)\mid\kappa_{\nu}\in B}$ 
 induced by the iteration $i$ can be
 substituted in the constructed generic set for the sequence
 $\seq{i_{\nu}(\bar\beta_\nu)\mid \kappa_\nu\in B}$ which would be assigned
 by the iteration $i_{\ords}$  as the indiscernibles associated with
 $\seq{i_{\ords}(\bar\beta_\nu)\mid\nu<\zeta}$. 

\subsubsection{Definition of the forcing: Overview}
\begin{definition}\label{def:overview}

The conditions of $P(\vec F)$ are functions $s$ satisfying the
following conditions:
\begin{enumerate}
    \item The domain of $s$ is a finite subset of $\zeta+1$ with $\zeta\in\domain(s)$.
    \item Each value $s(\tau)$ of $s$ is a  member of  the set $P^*_\tau$
      of  quadruples 
    \begin{equation*}
s(\tau)=(\forceKappa^{s,\tau},\vec F^{s,\tau},z^{s,\tau},\vec
A^{s,\tau}).
\end{equation*}
 satisfying the following conditions:
\begin{enumerate}
    \item $\vec F^{s,\tau}$ is a suitable sequence  $\vec F^{s,\tau}=\seq{F^{s,\tau}_{\nu}\mid
  \gamma_0\leq\nu<\tau}$ of extenders, where
$\gamma_0=\max(\domain(s)\cap \tau)+1$, or $\gamma_0=0$
if $\tau=\min(\domain(s))$.
\item 
 $\forceKappa^{s,\tau}$ is the critical point of the extenders in $\vec F^{s,\tau}$.
 \item $z^{s,\tau}$ is a tableau of functions giving information about
   the functions $h_{\nu,\nu'}$.  This tableau will  be fully
   specified in Definition~\ref{def:tableau}. 
 \item $\vec A^{s,\tau}$ is a sequence of sets $A^{s,\tau}_\nu\in
   U^{s,\tau}_\nu$, for $\gamma_0\le\nu<\tau$.  The definition of the ultrafilter 
    $U^{s,\tau}_\nu$ will be given in  Definition~\ref{def:A}.  
\end{enumerate}
\end{enumerate}

The two partial orders on $P(\vec F)$,  a direct extension order
$\le^*$ and a forcing
order $\le$, will be defined in Subsection~\ref{sec:PForder}. 
\end{definition}

\subsubsection{Definition of the forcing: the tableau $z=z^{s,\tau}$}
The third component $z^{s,\tau}$ of $s(\tau)$ is a tableau which is represented
in  Figure~\ref{fig:1}.    

The following definition specifies the
members of this tableau:

\begin{definition}\label{def:tableau}
Suppose that $\tau\in\domain(s)$, and set
$\gamma_0=\sup(\domain(s)\cap\tau)+1$, or $\gamma_0=0$ if
$\tau=\min(\domain(s))$.  The tableau $z=z^{s,\tau}$ includes 
\begin{enumerate}
    \item for each pair $(\gamma,\nu)$ of ordinals
with  $\tau\ge\gamma\geq\gamma_0>\nu\geq0$, a function
$f^{z}_{\gamma,\nu}$ and 
\item for each pair  $(\gamma,\nu)$ with
$\tau\ge\gamma>\nu\geq\gamma_0$, a pair of functions
$(a^{z}_{\gamma,\nu}, f^{z}_{\gamma,\nu})$.
\end{enumerate}
For each pair $\gamma,\nu$ the function $f^z_{\gamma,\nu}=f^{s,\tau}_{\gamma,\nu}$ is a
slightly modified Cohen function:
\begin{enumerate}
    \item $\domain(
      f^{z}_{\gamma,\nu})\subseteq[\forceKappa^{z},(\forceKappa^z)^{+})$
      and $\card{\domain(f^{z}_{\gamma,\nu})}\leq\forceKappa^z$.
\item Each of the values $f^z_{\gamma,\nu}(\xi)$ of $f^z_{\gamma,\nu}$
  has  one of the two following forms:
\begin{enumerate}
\item $f^{z}_{\gamma,\nu}(\xi)=\xi'\in \forceKappa^z_{\tau}$, or
\item\label{item:fpeculiar} $f^{z}_{\gamma,\nu}(\xi)=h_{\gamma',\nu}(\xi')$ for some
  $\gamma'$ in the interval $\gamma>\gamma'>\nu$ and some
  $\xi'\in\forceKappa^z_{\tau}$.
\end{enumerate}
\end{enumerate}
The functions $a^{z}_{\gamma,\nu}=a^{s,\tau}_{\gamma,\nu}$ satisfy the following conditions:
\begin{enumerate}
\item 
  $\domain( a^{z}_{\gamma,\nu})\subseteq[\forceKappa^{z},(\forceKappa^z)^{+})$
  and
  $\card{\domain(a^{z}_{\gamma,\nu})}\leq\forceKappa^z$.    
\item
  $\range(a^{z}_{\gamma,\nu})\subseteq\supp(F^{s,\tau}_{\nu})$. 
\item $\domain(a^{z}_{\gamma,\nu})\cap\domain(f^z_{\gamma,\nu})=\emptyset$, 
\item\label{item:asubset} If $\tau\geq\gamma>\gamma'>\nu$ then
    $a^{z}_{\gamma,\nu}\subseteq a^{z}_{\gamma',\nu}$.
    
\end{enumerate}
\end{definition}

\begin{figure}[t]
\renewcommand{\arraystretch}{1.25}
\begin{equation*}
    \begin{array}{c|ccccccc} 
      &0&\cdots&\gamma_0-1&\gamma_0&\cdots&\gamma&\cdots \\
      \hline \tau &
      f^{z}_{\tau,0}&\dots&f^{z}_{\tau,\gamma_0-1}&(a^{z}_{\tau,\gamma_0},f^{z}_{\tau,\gamma_0})&\dots&
      (a^{z}_{\tau,\gamma},f^{z}_{\tau,\gamma})&\dots
      \\    
      \vdots&\vdots&&\vdots&\vdots&&\vdots&
      \\
      \gamma& f^{z}_{\gamma,0}& \dots&f^{z}_{\gamma,\gamma_0-1}&
      (a^{z}_{\tau,\gamma_0},f^z_{\gamma,\gamma_0})&\dots&&
      \\
      \vdots&\vdots&&\vdots&\vdots&&&
      \\
      \gamma_0+1& f^{z}_{\gamma_0+1,0}& \dots&f^{z}_{\gamma_0+1,\gamma_0-1}&
      (a^{z}_{\gamma_0+1,\gamma_0},f^{z}_{\gamma_0+1,\gamma_0})& &&
      \\
      \gamma_0& f^{z}_{\gamma_0,0}&\dots&f^{z}_{\gamma_0,\gamma_0-1}&&&& \\
    \end{array}
  \end{equation*}
 \caption{The third component $z^{s,\tau}$ of $s(\tau)$.   The
   element at row $\alpha$ and column $\beta$ is used to determine
   $h_{\alpha,\beta}$.   In the case of the top row, this
   determination is direct; for the other rows this is indirect, via
   their use in  defining the ultrafilters $U^{s,\tau}_{\gamma}$ from which the sets
   $A^{s,\tau}_{\alpha}$ are taken.}
\label{fig:1}
\end{figure}

 The $(\gamma,\nu)$ entry in the tableau, whether a 
function $f^{z}_{\gamma,\nu}$ or a pair of functions
$(a^{z}_{\gamma,\nu}, f^{z}_{\gamma,\nu})$, will ultimately be used to determine
the values of the Cohen function $h_{\gamma,\nu}$.

 The functions $f^{s,\tau}_{\tau,\nu}$ in the first row of $z$ directly
determine $h_{\tau,\nu}$.  The functions $f^{s,\tau}_{\gamma,\nu}$ in the remaining rows, with
$\gamma<\tau$, 
indirectly help to determine $h_{\gamma,\nu}$ via the Prikry style
forcing: they restrict the possible values of $s'(\gamma)$ in
conditions $s'\leq s$.

The first form for the function $f_{\gamma,\nu}$ is the usual form for
a Cohen condition and asserts that 
$h_{\gamma,\nu}(\xi)=\xi'$;  or, more specifically,  if $s$ is a
condition with  $f^{s,\tau}_{\tau,\gamma}(\xi)=\xi'$,
then $s\forces \dot h_{\tau,\gamma}(\xi)=\xi'$.    The second form, 
the value $f^{s,\tau}_{\tau,\gamma}(\xi)=h_{\gamma',\nu}(\xi')$,  
of $f(\xi)$  may be taken as a formal expression: it specifies that  the value of 
the name $h_{\tau,\nu}(\xi)$ is given by
\begin{align}
  \text{if $s\forces \dot h_{\gamma',\nu}(\xi')=\xi''$ then}\quad&\quad s\forces\dot
  h_{\tau,\nu}(\xi)=\xi'',\label{eq:fpeculiareval}\\
  \text{if $s\forces\xi'\notin \domain(\dot h_{\gamma',\nu})$ then}\quad&\quad
  s\forces \dot h_{\tau,\nu}(\xi)=0, \text{ and}\notag\\
  \text{otherwise}
 \quad&\quad  s\ndecides\dot h_{\tau,\nu}(\xi).\notag
\end{align}
This definition requires recursion on $\tau$, using the fact that  ``$s\forces
\dot h_{\gamma',\nu}(\xi')=\xi''$'' depends only on $s\restrict\gamma'+1$.
In the first of these three cases, $s\forces \dot h_{\gamma',\nu}(\xi')=\xi''$, we will
regard the forms $f_{\tau,\nu}^{z}(\xi)=\xi''$ and
$f_{\tau,\nu}^{z}(\xi)=h_{\gamma',\nu}(\xi')$ as being identical.
\medskip{}

The functions $a^{z}_{\gamma,\nu}$ are included in order to generate the Prikry
indiscernibles. If $a^{s,\tau}_{\tau,\nu}(\xi)=\alpha$, then
$h_{\tau,\nu}(\xi)$ in the generic extension will 
be a Prikry indiscernible for the ultrafilter
$(F^{s,\tau}_{\nu})_{\alpha}=\set{x\in \ps(\forceKappa) \mid \alpha\in
  i^{F^{s,\tau}_{\nu}}(x)}$.

This completes the definition of the tableau $z^{s,\tau}$.   

\subsubsection{The forcing: the ultrafilters $U^{s,\tau}_{\gamma}$ and
  sets $A^{s,\tau}_{\gamma}$.}
We continue the definition of $P(\vec F)$ by specifying the
requirements for the final coordinate $A^{s,\zeta}$ for a quadruple
$w=s(\zeta)\in P^*_{\zeta}$.  Definition~\ref{def:A}  uses recursion on 
$\zeta$ to define the following for each    
for $\gamma<\zeta$: 
\begin{enumerate}
\item 
 a set $P^{*}_{\zeta,\gamma}$,  of which $A^{w}_{\gamma}$
is  a subset,
\item    a restriction operation
$w\scutdown\gamma$, which maps $w\in P^*_{\zeta}$ to a quadruple 
$w\scutdown\gamma\in P^{*}_{\zeta,\gamma}$, and 
\item an ultrafilter
$U^{w}_{\gamma}\subset\ps(P^{*}_{\zeta,\gamma})$.
\end{enumerate}

These will complete the definition of the set
$P^*_{\gamma}=P^*_{\gamma,\gamma}$, and hence of the set of conditions
of the forcing $P(\vec F)$.

In addition to $w\scutdown\gamma$ we use a second  restriction operator
$z\restrict[\gamma_0,\gamma]$, which may be applied to a tableau $z$
of the form of either Figure~\ref{fig:1} or~\ref{fig:member-of-A}.
This operator retains the rows of $z$ with indices in the interval
$[\gamma_0,\gamma]$ and discards the rows above these; thus if
$w=(\forceKappa^{w},\vec F^{w},z^{w},\vec A^{w})\in P^*_{\zeta}$, then 
$(\forceKappa^w,\vec F^{w}\restrict\gamma,
z^{w}\restrict[\gamma_0,\gamma], \vec A^{w}\restrict\gamma)\in P^*_{\gamma}$.

\begin{definition}
  \label{def:A}
  We assume as a recursion hypothesis that $P^{*}_{\tau}$ and
  $P^{*}_{\tau,\gamma}$
  have been defined for all
  $\gamma\le\tau<\zeta$.  
  If $\zeta\geq\gamma$ then 
  the members of $P^{*}_{\zeta,\gamma}$ are quadruples
  \begin{equation}
    w=(\forceKappa^{w},\vec F^{w}, z^{w}, \vec A^{w})\label{eq:c}
  \end{equation}
  satisfying the following conditions:
  \begin{enumerate}
  \item The tableau $z^w$ has the form of Figure~\ref{fig:member-of-A}.
  \item $w\restrict[\gamma_0,\gamma]=(\forceKappa^{w},\vec
    F^{w},z^{w}\restrict[\gamma_0,\gamma],\vec A^{w}) \in P^*_{\gamma}$.
  \item The functions $a^{z}_{\nu,\nu'}$ for $\tau\geq\nu>\gamma\geq\nu'$
    satisfy the conditions in Definition~\ref{def:tableau}, except
    that $a^{z}_{\tau,\nu'}$ has range contained in
    $[\forceKappa_\tau,(\forceKappa_\tau)^{+\omega_1})$.
  \end{enumerate}

Note that $P^*_{\tau,\tau}=P^{*}_{\tau}$.

Suppose that $\tau\leq\zeta$, $w\in P^*_{\tau,\gamma}$ and $\gamma'<\gamma$.  Then
$w\scutdown\gamma'$ is the quadruple
\begin{equation*}
w\scutdown\gamma'=(\forceKappa^{w},\vec
F^{w}\restrict\gamma', z^{w}\scutdown\gamma',\vec A^{w}\scutdown\gamma')\in
P^*_{\tau,\gamma'}
\end{equation*}
defined by recursion on $\gamma$ as follows:
\begin{compactenum}
\item $z^{w}\scutdown\gamma'$ is equal to the tableau obtained by
  deleting from $z^{w}$ all columns with index greater than $\gamma'$ and
  deleting the functions $f^{z}_{\nu,\nu'}$ from all rows with index
  greater than $\gamma'$.  Thus
  $(z^{w}\scutdown\gamma')\restrict[\gamma_0,\gamma']=z\restrict[\gamma_0,\gamma']$
  but the rows with index $\nu>\gamma'$ retain only the
  functions  $a^{w}_{\nu,\nu'}$ for $\gamma_0\leq\nu'<\nu\leq\gamma$.
\item
  $\vec
  A^{w}\scutdown\gamma'=\seq{A^{w}_{\gamma''}\scutdown\gamma'\mid
    \gamma_0\leq\gamma''\leq\gamma'}$ where
  $A^{w}_{\gamma''}\scutdown\gamma'=\set{w'\scutdown\gamma'\mid w'\in A^{w}_{\gamma''}}$.
\end{compactenum}
Note that this  definition  also applies for  $w\in P_{\tau}^*$, since
$P^*_{\tau}=P^*_{\tau,\tau}$. 
Finally, the ultrafilter $U^{s,\tau}_{\gamma}$ is defined as
\begin{equation}\label{eq:n}
  U^{s,\tau}_{\gamma}=\ufFromExt{F^{s,t}_{\gamma}}{s(\tau)\scutdown\gamma}=\set{X\subseteq P^*_{\tau,\gamma}\mid
    s(\tau)\scutdown\gamma\in i^{F^{s,\tau}_{\gamma}}(X)}.
\end{equation}

\end{definition}

\begin{figure}[]
\renewcommand{\arraystretch}{1.25}
\begin{equation*}
  \begin{array}{c|cccccc} 
    &0&\cdots&\gamma_0-1&\gamma_0&\cdots&\gamma \\
    \hline 
    \tau & & &  &a^{z}_{\tau,\gamma_0}&\dots& a^{z}_{\tau,\gamma}
    \\
    \vdots&&&&\vdots&&\vdots
    \\
    \gamma+1&&&&a^{z}_{\gamma+1,\gamma_0}& \dots&a^{z}_{\gamma+1,\gamma}
    \\[2pt]
    \arrayrulecolor{gray}\cline{2-7}\arrayrulecolor{black}
    \gamma& f^{z}_{\gamma,0}& \dots&f^z_{\gamma,\gamma_0-1}&(a^{z}_{\gamma,\gamma_0},f^z_{\gamma,\gamma_0})&\dots&
    \\
    \vdots&\vdots&&\vdots&\vdots&&
    \\
    \gamma_0+1& f^{z}_{\gamma_0+1,0}& \dots&f^{z}_{\gamma_0+1,\gamma_0-1}&
    (a^{z}_{\gamma_0+1,\gamma_0},f^{z}_{1,\gamma_0})& &
    \\
    \gamma_0& f^{z}_{\gamma_0,0}&\dots&f^{z}_{\gamma_0,\gamma_0-1}&&& \\
  \end{array}
\end{equation*}
  \caption{The tableau $z^w$ of a member of
    $A_{\gamma}^{s,\tau}\subseteq P^*_{\tau,\gamma}$. 
    The entry in row $\alpha$ and column $\beta$ is used in the
    determination of   $h_{\alpha,\beta}$.}
\label{fig:member-of-A}
\end{figure}

This completes the definition of the set of conditions for the forcing
$P(\vec F)$.



\subsection{The partial orderings of $P(\vec F)$.}
\label{sec:PForder}

Since $P(\vec F)$  is a
Prikry type forcing notion, we need to define both a direct extension
order $\leq^*$ and  a forcing order $\leq$.  
We will begin by defining the one-step extension, $\add(s,w)\leq s$,
 which is the atomic extension adding a new ordinal to the domain
of $s$.   We will then define the direct extension order $\leq^{*}$, which
will be the restriction of $\le$ to conditions $s'\leq s$ with $\domain(s')=\domain(s)$.
The forcing extension $\leq$ is then the smallest transitive relation
extending ${\leq^*}$ such that 
$\add(s,w)\leq s$ for all
$w\in\bigcup_{\tau\in\domain(s)}\bigcup_{\gamma}A^{s,\tau}_{\gamma}$. 

\subsubsection{The one-step extension}
\label{sec:one-step}
The one-step extension $s'=\add(s,w)$ in $P(\vec F)$  is the atomic
non-direct extension, corresponding to the extension
in Prikry forcing which simply adds one new ordinal to the finite
sequence.   In $P(\vec F)$ it acts by merging   Prikry components
$a^{s,\tau}_{\nu,\nu'}$ of $s(\tau)$ into the corresponding Cohen components of
$s'(\tau)$.   The following preliminary definition specifies  the
conversion of $a^{s,\tau}_{\nu,\nu'}$ to a Cohen  condition. 

\begin{definition}
  \label{def:a2f}
  Suppose  $w\in A^{s,\tau}_{\gamma}$ and
  $\tau\geq\nu>\gamma\geq\nu'\geq\gamma_0$, and let
  $a=a^{s,\tau}_{\nu,\nu'}$ and $a'=a^{w}_{\nu,\nu'}$.    The Cohen
  condition $f_{a,a'}$ is defined as follows:

  First, we define, for any function $a$ with domain a set of
  ordinals, a map $\sigma_{a,r}\colon
  \card{\domain(a)}\cong\domain(a)$.%
  \footnote{This definition would be simplified if a Levy collapse of
    $\reals$ onto $\omega_1$ had been taken at the start so that 
    $M$ satisfies GCH and hence the Axiom of Choice.  Then $\sigma_a$
    can be defined as the least map
    $\card{\domain(A)}\cong\domain(A)$ and used
    in place of the set of maps $\sigma_{a,r}$.}
  Write $\phi_a$  for 
  the least
  $\Sigma_{0}$ formula, with ordinal parameters, such that for some
  $r\in\reals$ the equation
  \begin{equation}\label{eq:w}
    \sigma_{a,r}(\nu)=\xi\iff \phi_a(r,\nu,\xi)
  \end{equation}
  defines an enumeration $\sigma_{x,r}\colon\card {\domain(a)}\cong
  \domain(a)$, and write $R_a$ for the set of $r\in\reals$ such that
  this holds.
  
  If
  $r\in R_a\cap R_{a'}$  then $f_{a,a',r}$ is the Cohen condition
  defined by 
  \begin{equation}
    \label{eq:faa-def}
    f_{a,a',r}(\xi)=
    \begin{cases}
      a'(\sigma_{a',r}\circ \sigma_{a,r}^{-1}(\xi))&\text{if
        $\sigma_{a,r}^{-1}(\xi)<\forceKappa^{w}$ and $\nu'=\gamma$,}  
      \\
      h_{\gamma,\nu'} (\sigma_{a',r}\circ\sigma_{a,r}^{-1}(\xi'))  &\text{if
        $\sigma_{a,r}^{-1}(\xi)<\forceKappa^{w}$ and 
        $\nu'<\gamma$,} 
      \\
      0&\text{if $\sigma_{a,r}^{-1}(\xi)\ge\forceKappa^{w}$,}
    \end{cases} 
  \end{equation}
  using in the second case the second form~(\ref{item:fpeculiar}) of the Cohen condition
  from Definition~\ref{def:overview}.
  Then $f_{a,a'}$ is defined if  and only if
  $R_a=R_{a'}$ and $(\forall r,r'\in R_a)\;f_{a,a',r}=f_{a,a',r'}$, in
  which case $f_{a,a'}$ is this common  value of $f_{a,a',r}$.
\end{definition}

  \begin{proposition}\label{thm:faa_exists}
    Suppose that $F$ is an extender with critical point $\lambda$.
    \begin{enumerate}
    \item
      If 
      $\card{\domain(a)}=\lambda$ then $\set{a'\mid f_{a,a'}\text{
          exists}}\in \ufFromExt{F}{a}$.
    \item If $\card{\domain(a_0)}=\card{\domain(a_1)}=\lambda$ and
      $a_1\supseteq a_0$ then $\set{(a'_0,a'_1)\mid
        f_{a_0,a'_0}=f_{a_1,a'_1}\restrict\domain(a_0)}\in \ufFromExt{F}{(a_0,a_1)}$.
    \end{enumerate}
  \end{proposition}
  \begin{proof}
    For the first clause, note that the elementarity of $i^{F}$
    implies that $\set{a'\mid R_{a'}=R_a}\in \ufFromExt{F}{a}$.  Let
    $r$ and $r'$ be members of $R_a$.  To see that  $\set{(a,a')\mid
      f_{a,a',r}=f_{a,a',r'}}\in\ufFromExt{F}{(a,a')}$, set 
    $\pi_{a,r,r'}=\sigma^{-1}_{a,r'}\circ\sigma_{a,r}$
    and
    $\pi_{a',r,r'}=\sigma^{-1}_{a',r'}\circ\sigma_{a',r}$.  
    Then by elementarity $\set{a'\mid \pi_{a',r,r'}=\pi_{a,r,r}\restrict\card{\domain(a')}}\in
    \ufFromExt{F}{a}$, and if  $a'$ is any member of this set, then (letting
    $\lambda'=\card{\domain(a')}$ and  letting 
    $\xi\in\sigma_{a,r}[\lambda']$ be arbitrary),  
    \begin{align*}
      f_{a,a',r}(\xi)=\sigma_{a',r}\circ\sigma_{a,r}^{-1}(\xi)&=(\sigma_{a',r'}\circ\pi_{a',r,r'})\circ
                                                (\sigma_{a,r'}\circ\pi_{a,r,r'})^{-1}(\xi)\\
      &=(\sigma_{a',r'}\circ\pi_{a,r,r'}\restrict\lambda')\circ
                                                (\pi^{-1}_{a,r,r'}\circ
                                \sigma^{-1}_{a,r'})(\xi)\\
      &=\sigma_{a',r'}\circ\sigma^{-1}_{a',r'}(\xi)=f_{a,a',r'}(\xi).
    \end{align*}
    This completes the proof of Clause~(1) of the Proposition, and a similar
    argument proves Clause~(2).
  \end{proof}

\begin{definition}[The one-step extension]
  \label{def:one-step}
  Suppose that $w\in A^{s,\tau}_{\gamma}$ where
  $\gamma\notin\domain(s)$ and 
  $\tau=\min(\domain(s)\setminus\gamma)$.  Then $s'=\add(s,w)$ is the
  condition with $\domain(s')=\domain(s)\cup\sing{\gamma}$ defined as
  follows:

  \begin{compactenum}
  \item
      $s'(\gamma)=(\forceKappa_\gamma^{w},\vec F^{w},
      z^{w}\restrict\,[\gamma_0,\gamma], \vec A^{w})$.

  \item
    $s'(\tau)=(\forceKappa_\tau^{s},\vec F^{s',\tau}, z^{s',\tau},
    \vec A^{s',\tau})$ where
    \begin{enumerate}
    \item $\forceKappa_\tau^{s'}=\forceKappa_\tau^{s}$ and
      $\vec F^{s',\tau{}}=\vec F^{s,\tau}\restrict(\gamma,\tau)$,
    \item $z^{s',\tau}$ is obtained from
      $z^{s,\tau}\restrict(\gamma,\tau]$ by using
      Definition~\ref{def:a2f} to replace $f^{s,\tau}_{\nu\nu'}$ with
      $f^{s',\tau}_{\nu,\nu'}=f^{s,\tau}_{\nu,\nu'}\cup f_{a^{s,\tau}_{\nu,\nu'},a^{w}_{\nu,\nu'}}$
      whenever $\tau\geq\nu>\gamma\geq\nu'\geq\gamma_0$, and
    \item\label{item:AddA} if $\gamma<\nu<\tau$ then
      $A^{s',\tau}_{\nu}=\set{\sigma(w')\mid w'\in
        A^{s,\tau}_{\nu}\land
        \forceKappa_\gamma^{w}<\forceKappa_\nu^{w'}}$, where
      \begin{align}
        \label{eq:Addw}
        \sigma(w')\restrict[\gamma,\nu]
        &=\add(w'\restrict[\gamma_0,\nu],w\scutdown\nu)\\
        \sigma(w')\restrict(\nu,\tau)&=w'\restrict(\nu,\tau).\notag
      \end{align}
    \end{enumerate}
  \item
    $s'(\gamma')=s(\gamma')$     for all
    $\gamma'\in\domain(s')\setminus\sing{\gamma,\tau}$. 
  \end{compactenum}
\end{definition}

Note that Equation~\eqref{eq:Addw} uses recursion on the pair
$(\gamma,\tau)$, along with the fact that 
$w'\restrict[\gamma_0,\nu]\in P^*_{\nu}$.

If any part of the definition of $\add(s,w)$ cannot be carried out as
described, then  $\add(s,w)$ is undefined.   Note that the set of $w$
for which it is defined is a member of $U^{s,\tau}_{\gamma}$, so that
we can assume without loss of generality that $\add(s,w)$ is defined
for all $w\in A^{s,\tau}_{\gamma}$.

This completes the definition of the one-step extension.
\subsubsection{The direct extension order $\le^*$.}
The direct extension order $\leq^*$  is the restriction of the
forcing order $\le$ to the pairs $(s',s)$ such that
$\domain(s)=\domain(s')$.
Again, the definition uses  recursion on $\tau$:

\begin{definition}
  \label{def:star-order} 
  If $s',s\in P(\vec F)$ then $s'\leq^* s$ if $\domain(s')=\domain(s)$
  and $s'(\tau)\le^* s(\tau)$ for all $\tau\in\domain(s)$.   The
  ordering $s'(\tau)\leq^* s(\tau)$ on $P^*_{\tau}$ holds if and only
  if the following   conditions hold:
  \begin{enumerate}
  \item $\forceKappa^{s',\tau}=\forceKappa^{s,\tau}$ and $\vec
    F^{s',\tau}=\vec F^{s,\tau}$. 
  \item \label{item:leq-star-a-fctns-extend}
    $a^{s',\tau}_{\gamma,\gamma'}\supseteq
    a^{s,\tau}_{\gamma,\gamma'}$ for each pair $(\gamma,\gamma')$ for 
    which they are defined. 
  \item\label{item:leq-star-change-w-in-A} 
    For each
    $\gamma\in(\gamma_0,\tau)$ and each $w'\in A^{s',\tau}_{\gamma}$
    there is $w\in A^{s,\tau}_{\gamma}$ such that
    \begin{enumerate}
    \item\label{item:SD-recursion} $w'\restrict[\gamma_0,\gamma]\leq^*
      w\restrict[\gamma_0,\gamma]$ in $P^*_{\gamma}$.      
    \item\label{item:a-extends} $a^{w'}_{\nu,\nu'}\supseteq a^{w}_{\nu,\nu'}$ for
      $\tau\geq\nu>\gamma\geq\nu'\geq\gamma_0$.
    \item\label{item:f-extends} For all pairs $(\nu,\nu')$ with
      $\tau\geq\nu>\nu'\geq\gamma_0$ we have 
     $     f_{a_{\nu,\nu'}^{s',\tau},\,a^{w'}_{\nu,\nu'}}\supseteq
     f_{a^{s,\tau}_{\nu,\nu'},\, a^{w}_{\nu,\nu'}}$, where these two
     functions are as defined in Definition~\ref{def:a2f}.
    \end{enumerate}
  \item\label{item:z} $f^{s',\tau}_{\nu,\nu'}\supseteq f^{s,\tau}_{\nu,\nu'}$ for each
    pair $\nu,\nu'$ for which they are defined.
  \end{enumerate}
\end{definition}

Clause~\ref{item:leq-star-change-w-in-A} implies that $\add(s',w')\leq^*
\add(s,w)$.   This clause
 corresponds to the requirement
in Prikry forcing that $A^{s'}\subseteq A^{s}$; however the
ultrafilters $U^{s,\tau}_{\gamma}$ used in this forcing vary with
$s$.  Gitik \cite{Gitik2005No-bound-for-th} also has varying
ultrafilters, but takes them from a predefined set and uses
predefined witnesses to a Rudin-Keisler order on the ultrafilters.
Our definition could also be stated in terms of the Rudin-Keisler
order,
however the  ultrafilters would have to be defined on the
complete Boolean algebra induced by the ordering $(P^*_{\tau,\gamma},\leq^*)$.

\medskip{}

This completes the definition of  the forcing
$(P(\vec F), {\leq^*}, {\leq})$.

\subsection{Properties of the forcing $P(\vec F)$}
\label{sec:PFproperties}

\begin{definition}
  If $\vec w$ is a  sequence of length $n$, then we write
  $\add(s,\vec w)$ for the condition defined by recursion as
  $\add(s,\vec w)=s$ if $n=0$, and $\add(s,\vec w) =\add(\add(s,\vec
  w\restrict (n-1)), w_{n-1})$ if $n>0$.
\end{definition}

\begin{proposition}\label{thm:one-stepFirst}
  Suppose that $s\leq t$. Then
  there is $\vec z$ such that
  $s\leq^*\add(t,\vec z)\leq t$
\end{proposition}
\begin{proof}
  The proposition will  follow by an easy
  induction on the length of $\vec z$ once we show that for any $t'\leq^* t$ and
  $s=\add(t',w')\leq t'$, where $w'\in A^{t',\tau}_{\gamma}$,  there is  $ w\in A^{t,\tau}_{\gamma}$ such
  that $s\leq^* \add(t, w)<t$.    Clause~3 of
  the Definition~\ref{def:star-order} of  
  the direct ordering $\leq^*$ is designed to provide such a $w$:
  \begin{align*}
    s(\tau)&\le^* \add(t, w)(\tau)&&\text{by Clauses~(3b,c),}\\
    s(\gamma)&\leq^*\add(t,  w)(\gamma)&&\text{by Clause~(3a), and}\\
    s(\gamma')&=\add(t',w')(\gamma')=t'(\gamma')\\
    &\qquad\leq^*t(\gamma')=\add(t,w)(\gamma')                                  &&\text{for
                                     $\gamma'\in\domain(s')\setminus\sing{\tau,\gamma}$.}    
  \end{align*}
\end{proof}

\begin{proposition}\label{thm:one-step-commute}
  Suppose $s\leq t$ and  $\gamma\in\domain(s)\setminus\domain(t)$, and
  let 
  $\tau=\min(\domain(t)\setminus\gamma)$.   Then there is $w\in
  A^{t,\tau}_{\gamma}$  such that $s\leq\add(t,w)<t$.
\end{proposition}

\begin{proof}
  By using Proposition~\ref{thm:one-stepFirst}, we can find $\vec w$
  so that $s\leq^* \add(t,\vec w)\leq t$ for some sequence $\vec w$.
  Thus it only remains to show that the order of the sequence $\vec w$
  can be   permuted, that is, that there is $\vec w'$ such that $\add(s,\vec
  w)=\add(s,\vec w')$ and $w'_{0}   \in A^{s,\tau}_{\gamma}$.   

  This will follow by an easy induction once we show that the order
  of two consecutive one-step extensions can be reversed.    
  Thus suppose that 
  $s=\add(\add(t,w_0),w_1)$, with  $w_0\in
  A^{t,\tau_0}_{\nu_0}$ and $w_1\in A^{\add(t,w_0),\tau_1}_{\nu_1}$.
  We want to find  $w'_1\in A^{t,\tau'_1}_{\nu_1}$ and $w'_0\in A^{\add(t,w'_1),\tau'_0}_{\nu_1}$
  so that $s=\add(\add(t, w'_1),w'_0)$.
  We have three cases:

  \begin{case}
    {1}{$\nu_0<\nu_1<\tau_0$}
    In this case $\tau_1=\tau_0$, and 
    by definition~\ref{def:star-order}, there is $w'_1\in
    A^{t,\tau_0}_{\nu_1}$ 
    such that $w_1=
    (w'_1)\scutdown\nu_1$.
    Then $s=\add(\add(t,w_1'),\sigma_{\nu_0}(w_0))$, where
    $\sigma_{\nu_0}$ is as defined in Clause~3 of Definition~\ref{def:one-step}.
  \end{case}
  \begin{case}{2}{$\nu_1<\tau'_1=\nu_0$}
    By Definition~\ref{def:one-step}, $w_1=\sigma_{\nu_1}(w'_1)$ for some
    $ w'_1\in A^{t,\tau_0}_{\nu_1}$.   Then $s=\add(\add(t,w'_1),w_1\restrict (\nu_0,\tau_0])$.
  \end{case}
  \begin{case}{3}{$\nu_1>\tau_0$ or $\tau'_1<\nu_0$}
    In this case $\add(\add(t, w_0),w_1)=\add(\add(t,w_1),w_0)$ so we
    can take $w'_0=w_0$ and $w'_1=w_1$.
  \end{case}
\end{proof}

 We write $\below{P(\vec F)} s$ for $\set{s'\in P(\vec F)\mid s'\leq
  s}$.   The proof of the following proposition is straightforward.

\begin{proposition}[Factorization]  Suppose $s\in P(\vec F)$ and
  $\gamma\in\domain(s)$ for some $\gamma<\zeta$.  Then
  \label{thm:factorization}
    \begin{equation}
      \label{eq:factorizationexact}
      \below{P(\vec F)}{s} \text{ is a regular suborder
        of }\below{P(\vec F^{s,\gamma})}{s\restrict\gamma+1}\times
      P'
  \end{equation}
    where $P'=\set{q\restrict(\gamma,\zeta]\mid q\leq s}$.
    Thus  $\below{P(\vec F)}{s}$ can be written in the form
    \begin{equation}\label{eq:factorizationstar}
      \below{P(\vec F)}{s}\equiv
      \below{P(\vec   F^{s,\gamma})}{(s\restrict\gamma+1)}    \star\dot R
    \end{equation}
    where  $\dot R$ is a $\below{P(\vec
      F^{s,\gamma})}{s\restrict\gamma+1}$-name for a  Prikry style forcing order.
    \qed
\end{proposition}

This factorization property is an important property of
this Magidor-Radin style of Prikry forcing.   
Typically, equation~\eqref{eq:factorizationexact} would be an equality
rather than a subalgebra; however that is not true here
because  
of the peculiar form of the 
Cohen conditions $f^z_{\nu,\nu'}(\xi)=h_{\nu'',\nu'}(\xi'')$
in Clause~(\ref{item:fpeculiar}) of Definition~\ref{def:tableau}.  
When $\nu>\gamma\ge\nu''$,  the determination via Definition~\ref{def:a2f}
of the ultimate value of $h_{\nu,\nu'}$ depends on both
$\below{P(\vec   F^{s,\gamma})}{s \restrict\gamma+1}$ and $R$.
The generic $G\subseteq P(\vec F)$  obtained from a generic
$G_0\times G_1\subseteq P(\vec F^{s,\gamma})\times P'$ is obtained by resolving,
as specified in
equation~\eqref{eq:fpeculiareval}, the values of the
Cohen conditions  in $G_1$ which have the form described in
Definition~\ref{def:tableau}(\ref{item:fpeculiar}): that is, 
$f_{\nu,\nu'}(\xi)=h_{\nu'',\nu'}(\xi'')$ for some $\nu$,
$\nu''$ and $\nu'$ with 
$\nu>\gamma\ge\nu'$. 

Note that the forcing $P'$ in equation~\eqref{eq:factorizationexact}
is in fact identical to $P(\vec F)$ except that the domain of the
conditions is contained in the interval $[\gamma+1,\zeta]$  instead of
$[0,\zeta]$, and $\gamma+1$ is used instead of $0$ as the default
value of $\gamma_0$  in the definition of
$P^*_{\tau}$ when 
$\domain(s)\cap\tau=\emptyset$ (but the tableau of
figure~\ref{fig:1} retains all of its columns, starting with $0$).
Thus all of the properties proved of $P(\vec F)$ are also true of $P'$.
This 
factorization  will frequently
be used in proofs, sometimes implicitly, to justify simplifying
notation by proving that the result holds for the case when
$\domain(s)=\sing{\zeta}$.
The result then follows for arbitrary $s$ by a simple induction on
$\zeta$:  If $s$ is an arbitrary condition in $P(\vec F)$ and
$\gamma=\max(\domain(s)\cap\zeta)$ then the induction step uses 
the induction 
hypothesis for $P(\vec F^{s,\gamma})$ and the special case
$\domain(s)=\sing{\zeta}$ for $R$.

\begin{lemma}[Closure]
  \label{thm:closure}
    Suppose that $\seq{s_{\nu}\mid\nu<\beta}$ is a
  $<^{*}$-descending sequence of conditions in $P(\vec F)$.  
  \begin{enumerate}
  \item\label{item:closuresmall}($\kappa$ closure)
    If $\beta<\forceKappa^{s_0,\min(\domain(s_0))}$ then the infimum
    $\bigwedge_{\nu<\beta}s_\nu$ of this sequence exists.
  \item (Diagonal closure)\label{item:closurediagonal}
    Suppose that  $\beta=\forceKappa^{s_0,\min(\domain(s_0))}$.
    Then there is $s=\bigtriangleup_{\nu<\beta}s_\nu\leq^* s_0$ such that 
    $s\forces\forall\nu<\dot\forceKappa_0\;s_{\nu}\in\dot G$.
    \end{enumerate}
\end{lemma}
Note that for the factorization forcing $P'$ of
Proposition~\ref{thm:factorization}, $\forceKappa_0$ can be replaced
by $\forceKappa_{\gamma+1}$.
\begin{proof}
   The proof is by
  induction on $\zeta$, using Proposition~\ref{thm:factorization}.
  Thus we can  assume that $\domain(s_0)=\sing{\zeta}$.
  Since the first two coordinates of $s_{\nu}(\zeta)$ are fixed and
  the third, $z^{s,\nu}$, is $\kappa^+$-closed, the fourth coordinate,
  $\vec A^{s_\nu,\zeta}$, is the only problem.   
  
  If $w',w\in P_{\zeta,\eta}^*$ then we write $w'\leq^* w$ if the
  conditions of
  Definition~\ref{def:star-order}(\ref{item:leq-star-change-w-in-A}) hold.
  If $\zeta>\gamma>\eta$ then the induction hypothesis trivially
  extends to  sequences
  in $P^*_{\gamma,\eta}$, since only subclause~(\ref{item:leq-star-change-w-in-A}a) is problematic.

  Now, to prove Clause(\ref{item:closuresmall}) of the Lemma we need to define
  $A^{s,\zeta}_{\eta}$ for each $\eta<\zeta$.  We can assume that
  $\beta<\forceKappa^{w}_{\eta}$ for all $w\in
  A^{s,\zeta}_{\eta}$.  Set
  \begin{equation*}
    A^{s,\zeta}_{\eta}=\set{\bigwedge_{\nu<\beta}w_{\nu}\mid
    (\forall\nu<\beta)\;w_\nu\in A^{s_\nu,\zeta}_{\eta}\land
    (\forall\nu'<\nu<\beta)\; w_{\nu'}\leq^*w_{\nu}}.
  \end{equation*}
  To see that  $A^{s,\zeta}_{\eta}\in U^{s,\zeta}_{\eta}$ note that
  the induction hypothesis implies that  
  the infimum
  $w=\bigwedge_{\nu<\beta}(i^{F^{s,\zeta}_\eta}(s_\nu))\scutdown\eta$
  exists,  and
  $w\in i^{F^{s,\zeta}_\eta}(A^{s,\zeta}_{\eta})$.  

  This concludes the proof of Clause~(\ref{item:closuresmall}), and
  the proof of Clause~(\ref{item:closurediagonal}) is similar. 
\end{proof}

\begin{lemma}
  \label{thm:diagonal-closure}
  Suppose that $s\in P(\vec F)$ and for all $w\in
  A^{s,\zeta}_{\gamma}$ the set $D$ is open and dense  in
  $(P(\vec F),\leq^*)$ below $\add(s,w)$.   Then there is  a condition
  $s'\leq^* s$ such that   $s''\in D$ for all $s''<s'$ having $\gamma\in\domain(s'')$.
\end{lemma}
\begin{proof}{}
  By Proposition~\ref{thm:one-step-commute} it will be enough to show
  that there is $s'\leq^* s$ such that $\add(s',w)\in D$ for all
  $w\in A^{s',\zeta}_{\gamma}$.   In order to simplify notation,
  we assume that $\domain(s)=\sing{\zeta}$.

  By proposition~\ref{thm:enoughAC} we can assume that
  $A^{s,\zeta}_{\gamma}$ can be enumerated as $\set{w_{\nu}\mid
    \nu<\kappa}$  
  so that $\nu'\leq\nu$ implies
  $\forceKappa^{w_{\nu'}}_{\gamma}\leq\forceKappa^{w_{\nu}}_{\gamma}$.
  We will define  by recursion on $\nu$ a 
  $\leq^*$-decreasing sequence of conditions
  $\seq{s_{\nu}\mid\nu<\kappa}$ in $R$ so that
  $\add(s_{\nu},w_\nu)\leq^* \add(s,w_\nu)$ for all $\nu<\kappa$.
  At the same time we will define a function $\sigma\colon
  A^{s,\zeta}_{\gamma}\to P^{*}_{\zeta,\gamma}$ so that $s_\nu$ and
  $\sigma(w_\nu)$ satisfy the following conditions:
  \begin{enumerate}
  \item $s_0=s$,
  \item $s_{\nu}\scutdown\gamma=s\scutdown\gamma$ and $\vec
    A^{s_\nu,\zeta}\restrict\gamma+1=\vec
    A^{s,\zeta}\restrict\gamma+1$ for all $\nu<\kappa$, 
  \item $\add(s_{\nu+1},\sigma(w_\nu))\in D$, and 
  \item $s_{\nu'}\leq^*s_{\nu}$ for all $\nu'<\nu<\kappa$.
  \end{enumerate}
  Note that clause~(2) implies that $\add(s_{\nu},w)$ exists for all
  $\nu<\kappa$ and all $w\in A^{s,\zeta}_{\gamma}$.   Also,  clauses~(2)
  and~(4) imply that $\add(s_{\nu'},w_{\nu})\le^*
  \add(s_{\nu},w_\nu)\leq^* \add(s,w_\nu)$ for all $\nu<\nu'<\kappa$.
  
  To define the sequence, set $s_0=s$, and 
  if $\nu$ is a limit ordinal then set
  $s_{\nu}=\bigwedge_{\nu'<\nu}s_{\nu'}$.  For a successor ordinal
  $\nu+1$, since $\add(s_{\nu},w_\nu)\leq^* \add(s,w_\nu)$, the
  hypothesis implies that there is $t\leq^*\add(s_\nu,w_\nu)$ such
  that $t\in D$.

  Define $\sigma(w_\nu)$ by
  \begin{align*}
    \sigma(w_\nu)\restrict[\gamma_0,\gamma] &=
                                              (t\scutdown\gamma)\restrict[\gamma_0,\gamma]
    \text{, and}\\
    \sigma(w_\nu)\restrict(\gamma,\zeta] &=
                                           w_\nu\restrict(\gamma,\zeta].
  \end{align*}
  
  By clause~(2) we have $s_{\nu+1}\scutdown\gamma=s\scutdown\gamma$ and
  $\vec A^{s_{\nu+1},\zeta}\restrict\gamma+1=\vec A^{s}\restrict\gamma+1$.
  The remainder of $z^{s_{\nu+1}}$ is taken from $t$; that is:
  \begin{align*}
    a^{s_{\nu+1},\zeta}_{\eta,\eta'}&=a^{t,\zeta}_{\eta,\eta'}&&\text{if
                                                     $\eta'>\gamma$,}\\
    f^{s_{\nu+1},\zeta}_{\eta,\eta'}&=f^{t,\zeta}_{\eta,\eta'}&&\text{if
                                                                 $\eta'>\gamma$,
                                                           and}\\
    f^{s_{\nu+1},\zeta}_{\eta,\eta'}&=f^{t,\zeta}_{\eta,\eta'}\restrict(\kappa^{+}\setminus\domain(a^{s,\zeta}_{\eta,\eta'}))&&\text{if $\eta>\gamma\geq\gamma'$.} 
      \end{align*}

  The definition of $A_{\eta} ^{s_{\nu+1},\zeta}$ for
  $\zeta>\eta>\gamma$ is by recursion on $\gamma$.  For $w\in
  A^{s_\nu,\zeta}_{\eta}$ and $w'\in A^{t,\zeta}_{\eta}$, let us  write $w'\leq^* w$ if
  they satisfy
  Definition~\ref{def:star-order}(\ref{item:leq-star-change-w-in-A}),
  in which case  let $\pi_{w'}(w)$ be given by
  \begin{enumerate}
  \item
    $\pi_{w'}(w)\restrict[\gamma+1,\zeta]=w\restrict[\gamma+1,\zeta]$,
    and 
  \item 
    $\pi_{w'}(w)\restrict[\gamma_0,\gamma]$ is defined 
    in the same way as  $s_{\nu+1}$, but with  $w'\restrict[\gamma_0,\gamma]$, $w\restrict[\gamma_0,\gamma]$ and $\eta$ in
    place of  $t,s_{\nu}$ and  $\zeta$. 
  \end{enumerate}
  Then
  \begin{multline*}
    A_{\eta}^{s_{\nu+1},\zeta}=\set{\pi_{w'}(w)\mid 
      w'\in  A^{t}_{\eta}\land 
      \forceKappa^{w'}>\forceKappa^{w_\nu}\land
      w\in A^{s_\nu}_{\eta}
      \land w'\leq^* w}.
  \end{multline*}
  \medskip

  Now set $s_{\kappa}=\bigtriangleup_{\nu<\kappa}s_{\nu}$,  and set
  $\bar w   =[\sigma ]_{U^{s,\zeta}_{\gamma}}= 
  i^{F^{s,\zeta}_{\gamma}}(\sigma)(s\scutdown\gamma)$.
  Then clause~(2) of the initial conditions on $s_{\nu}$ allow $\bar
  w$ to be merged into $s_{\kappa}$, giving the desired extension
  $s'\leq^*s$.   We can assume without loss of generality that
  $w'\in A^{s_\kappa,\zeta}_{\eta}$ whenever $w\in A^{s_\kappa,\zeta}_{\eta}$ and $w'\leq^*w$ in the
  sense of Definition~\ref{def:star-order}(\ref{item:leq-star-change-w-in-A}).
  \begin{todoenv}
    (7/13/15) --- This might be simplified by assuming that $w'\leq^*
    w\in A^{s,\tau}_{\nu}$ implies that $w'\in A^{s,\tau}_{\nu}$.
    This would make the ultrafilter really an ultrafilter on a Boolean algebra.
  \end{todoenv}

  \begin{align*}
    A^{s',\zeta}_{\eta}&=
    \begin{cases}
      A^{\bar w}_{\eta}&\text{if $\eta<\gamma$},\\
      \set{\sigma(w)\mid w\in A^{s}_{\gamma}}&\text{if
        $\eta=\gamma$ and}
      \\
      A^{s_\kappa,\zeta}_{\eta}&\text{if $\eta>\gamma$};
    \end{cases}
    \\
    z^{s',\zeta}\scutdown\gamma &= \bar w\restrict[\gamma_0,\gamma],
    \\
    f^{s',\zeta}_{\eta,\eta'}&=f^{s_\kappa,\zeta}_{\eta,\eta'}&&\text{if
                                                     $\eta>\gamma$, and}
    \\
    a^{s',\zeta
}_{\eta,\eta'}&= a^{s_\kappa,\zeta}_{\eta,\eta'}&&\text{if $\eta'>\gamma$.}
  \end{align*}
\end{proof}

\subsection{The Prikry property}
\label{sec:prikry}
\begin{lemma}\label{thm:prikry}
  \begin{enumerate}
  \item\label{item:Prikrythm-decide} Let $\phi$ be a sentence and $s$ a condition in $P(\vec F)$.
    Then there is an $s'\leq^{*}s$ such that $s'$ decides $\phi$.
  \item\label{item:Prikrythm-inD} Let $D$ be a dense subset of $P(\vec F)$, and suppose $s\in
    P(\vec F)$. Then there is an $s'\leq^{*}s$ and a finite
    $b\subseteq\zeta+1$ such that any $s''\leq s'$ with
    $b\subseteq\domain(s'')$ is a member of $D$.
  \end{enumerate}
\end{lemma}

\begin{proof}[Proof of Lemma~\ref{thm:prikry}]
 
  The proof of Lemma~\ref{thm:prikry} is by induction on the length
  $\zeta$ of $\vec F$.   By the induction hypothesis and Proposition
  \ref{thm:factorization}
  we can simplify the notation by assuming that $\domain(s)=\sing{\zeta}$.
  The main part of the proof is the following claim:
  \begin{claim}
    \label{thm:Prikry-sublemma}
    Suppose that $D\subseteq P(\vec F)$ is dense and $s\in P(\vec F)$
    has domain $\sing{\zeta}$.  Then there is $s'\leq^* s$ such
    that either $s'\in D$ or for some $\gamma<\zeta$
    \begin{multline}
      \label{eq:Prikry-sublemma}
      s'\forces (\exists w\in A^{s',\zeta}_{\gamma})(\exists t\in \dot
      G\cap D)\;\Bigl(\domain(t)\subseteq
      (\gamma+1)\cup\sing\zeta\\
      \land t\leq\add(s',w)\land 
      t(\zeta)=\add(s',w)(\zeta)\Bigr)
    \end{multline}
  \end{claim}

  \begin{proof}
    For each $\gamma<\zeta$, define
    \begin{align*}
      D^{+}_{\gamma}&=\set{t\in P(\vec F)\mid t\forces(\exists t'\in\dot G\cap
                      D)\,\domain(t')\subseteq(\gamma+1)\cup\sing{\zeta}}
      \\
      D^{-}_{\gamma}&=\set{t\in P(\vec F)\mid t\forces\lnot(\exists t'\in\dot G\cap
                      D)\,\domain(t')\subseteq(\gamma+1)\cup\sing{\zeta}}
      \\
      E_{\gamma}&=
                  \set{t\in P(\vec F)\mid
                  (\forall t'\leq t)\;
                  \bigl( t'\in
                  D\land\domain(t')\subseteq(\gamma+1)\cup\sing\zeta
      \\&\hspace{5cm}  \implies (t'\restrict(\gamma+1)\cup 
          t\restrict\sing{\zeta})\in D\bigr)}.
    \end{align*}

    First, suppose that for all  $\gamma<\zeta$ the set 
    $(D^{+}_{\gamma}\cup D^{-}_{\gamma})\cap E_{\gamma}$ is $\leq^*$-dense below
    any condition $t\leq s$ with $\domain(t)=\sing{\gamma,\zeta}$.
    Then by Lemma~\ref{thm:diagonal-closure} there is $s'\leq^* s$
    such that for each $\gamma<\zeta$ and $w\in A^{s',\zeta}_{\gamma}$ we have  $\add(s,w)\in
    (D^{+}_{\gamma}\cup D^{-}_{\gamma})\cap E_{\gamma}$.   By shrinking the sets
    $A^{s',\zeta}_{\gamma}$ we can assume that for each $\gamma$,
    $\set{\add(s',w)\mid w\in A^{s',\zeta}_{\gamma}}$ is contained in one
    of $D^{+}_{\gamma}\cap E_\gamma$ or $D^{-}_{\gamma}\cap E_{\gamma}$.  Since $D$ is
    dense it follows that $\set{\add(s',w)\mid w\in
      A^{s',\zeta}_{\gamma}}\subseteq D^{+}_{\gamma}$ for some
    $\gamma<\zeta$, and it follows by
    Proposition~\ref{thm:factorization} that $s'$ satisfies the formula~\eqref{eq:Prikry-sublemma}.

    Now fix $\gamma<\zeta$ and $t\leq s$ with $\gamma\in\domain(t)$.   We will show that 
    $(D^{+}_{\gamma}\cup D^{-}_{\gamma})\cap E_{\gamma}$ is
    $\leq^*$-dense below $t$.
    First, note that by Proposition~\ref{thm:factorization}, the set $E_{\gamma}$
    is $\leq^*$-dense below any condition $t$ with
    $\gamma\in\domain(t)$. Now for $t\in E_{\gamma}$, consider the following
    formula  in the forcing language of $P(\vec F^{t,\gamma})$:
    \begin{equation}
      \exists
      t'\in \dot G\, (t'\cup t\restrict\sing\gamma\in D).\label{eq:Prikry.sublemma.a}
    \end{equation}
    By the induction hypothesis of
    Lemma~\ref{thm:prikry}(\ref{item:Prikrythm-decide}) there is
    $t''\leq^* 
    t\restrict\gamma+1$ which decides, in $P(\vec F^{t,\eta})$, the truth of
    formula~\eqref{eq:Prikry.sublemma.a}.   Then
    $t''\cup t\restrict\sing{\zeta}$ is in either $D^{+}_{\gamma}$ or
    in $D^{-}_{\gamma}$. 
  \end{proof}
  
  To complete the proof of
  Lemma~\ref{thm:prikry}(\ref{item:Prikrythm-decide}),  apply
  Claim~\ref{thm:Prikry-sublemma}  with  $D=\set{t\mid t\decides\phi}$.
  Since we are done if there is  $s'\leq^*
  s$ in $D$ we can assume by  
 Claim~\ref{thm:Prikry-sublemma}  that there is $s'\leq^* s$ and
  $\gamma<\zeta$ such that \eqref{eq:Prikry-sublemma} holds.

  By the induction hypothesis, for each $w\in A^{s',\zeta}_{\gamma}$
  there is $t_w\le^* \add(s',w)\restrict(\gamma+1)$ in $P(\vec F^{w})$ such that
  $t_w\cup\add(s',w)\restrict\sing{\zeta} 
  \decides\phi$.  Then  either $\set{w\in
    A^{s',\zeta}_{\gamma}\mid t_w\cup
    \add(s',w)\restrict\sing{\gamma}\forces\phi} \in
  U^{s',\zeta}_{\gamma}$ or $\set{w\in
    A^{s',\zeta}_{\gamma}\mid t_w\cup
    \add(s',w)\restrict\sing{\gamma}\forces\lnot\phi} \in
  U^{s',\zeta}_{\gamma}$.   Now reduce $A^{s',\zeta}_{\gamma}$ to
  whichever  set is in $U^{s',\zeta}_{\gamma}$, and apply
  Lemma~\ref{thm:diagonal-closure} to obtain $s''\leq^* s'$ such that
  $s''$ decides $\phi$.

 Lemma~\ref{thm:prikry}(\ref{item:Prikrythm-inD}) is proved similarly,
 applying Claim~\ref{thm:Prikry-sublemma} using the set $D$ given in
 the hypothesis.
\end{proof}

If $\gamma<\zeta$ and $G\subseteq P(\vec F)$ is generic, then set
$G\restrict \gamma+1=\set{s\restrict\gamma+1\mid
  \gamma\in\domain(s)\land s\in G}$.  Then
$G\restrict\gamma+1$ is a generic subset of $P(\vec F^{s,\gamma})$.

\begin{corollary}[No new bounded sets]
  \label{thm:noBoundedSets}
  Suppose  $x\in M[G]\setminus M$ and
  $x\subset\lambda<\forceKappa^{M[G]}_{\gamma+1}$. 
  Then $x\in M[G\restrict\gamma'+1]$ for some $\gamma'<\gamma$. 
\end{corollary}
\begin{proof}
  If $\gamma=\gamma'+1$ then Propositions~\ref{thm:factorization}
  and~\ref{thm:closure} imply that $\gamma'$ is as required.
  If $\gamma$ is a limit ordinal then take $\gamma'$ least such that
  $\forceKappa^{G}_{\gamma'}>\lambda$. 
\end{proof}

\begin{corollary}\label{thm:approx}
  If $\vec F$ is a suitable sequence with critical point $\kappa$ then
  $P(\vec F)$ has the $\kappa$-approximation property: if $G\subseteq
  P(\vec F)$ is $M$-generic then for any function $f\in M[G]$ with
  $\domain(f)=\kappa$ there is a set $A\in M$ with $\card A\leq\kappa$
  and $\range(f)\subseteq A$.
\end{corollary}
\begin{proof}
    Let $\dot f$ be the name of a function
  $f\colon \kappa=\forceKappa_{\zeta}\to\kappa^{+}$, and let $s$ be a
  condition, which  we will assume has domain $\sing{\zeta}$.   

  If $\zeta=\gamma+1$ then, for any condition $s$ with
  $\gamma\in\domain(s)$, factor $\below{P(\vec F)}{s}$ as
  $\below{P(\vec F^{s,\gamma})}{s\restrict\gamma+1}\times P'$.   Then
  $P'$ is $\kappa^{+}$-closed since $\vec F^{s,\gamma+1}=\emptyset$,
  so there is $s'\leq 
  s\restrict\sing{\zeta}$ such that for all $\alpha<\kappa$ there are
  $\beta$ and $t\in G\restrict\gamma+1$ such that $t\cup s'\forces
  \dot f(\alpha)=\beta$.  Thus we can take
  \begin{equation*}
    A=\set{\beta\mid (\exists\alpha<\kappa)(\exists
      t\in \below{P(\vec F^{s,\gamma})}{s\restrict\gamma+1})\;
      t\cup s'\forces \dot f(\alpha)=\beta}.
\end{equation*}
  If $\zeta$ is a limit ordinal then use
  Lemma~\ref{thm:closure} to define a $\leq^*$-decreasing
  sequence of conditions $s_\gamma\leq^* s$ such that $s_{\gamma}$ forces
  the following formula:
  \begin{multline*}
    (\forall\alpha<\dot{\forceKappa}_{\gamma})\forall\beta\;
    \bigl((\exists t\in \dot 
    G)\;\domain(t)\subseteq(\gamma+1\cup\sing\zeta)\land
    t\forces\dot f(\alpha)=\beta
    \implies
    \\
    (\exists w\in A^{s_{\gamma},\zeta}_{\gamma})(\exists t\in \dot
    G\restrict\gamma+1)\;
    (t\leq\add(s_{\gamma},w)\land
    \\
    t\cup\add(s_{\gamma},w)\restrict\sing{\zeta}\forces\dot f(\alpha)=\beta 
    \bigr).
  \end{multline*}
  Set $s'=\bigwedge_{\nu<\zeta} s_{\nu}$ and 
  \begin{multline*}
    A_{\gamma}=\set{\beta\mid \exists\alpha(\exists w\in
    A^{s',\zeta}_{\gamma})(\exists
    t<\add(s',w))\;\\t\restrict(\gamma,\zeta]=\add(s,w)\restrict(\gamma,\zeta]\land
    t\forces \dot f(\alpha)=\beta}.
  \end{multline*}
  Then $s'\forces\range(\dot f)\subseteq \bigcup_{\gamma<\zeta}A_{\gamma}$.
\end{proof}

\begin{corollary} \label{thm:PF-nocollapse}
  Forcing with $P(\vec F)$ does not collapse any cardinal which is not
  in the set $\bigcup_{\gamma\leq\zeta}[\forceKappa_{\gamma}^{++},   \forceKappa^{+\gamma+1}_{\gamma}]$.
\end{corollary}
\begin{proof}
  Suppose $\lambda$ is a cardinal of $M$ which is collapsed in $M[G]$,
  where 
  $G\subseteq P(\vec F)$ is  $M$-generic.    If $\lambda<\kappa=\forceKappa_{\zeta}$ then
  Corollary~\ref{thm:noBoundedSets} implies that the collapsing
  function is in $M[G\restrict\gamma+1]$ for some $\gamma<\zeta$.
  Thus we can assume without loss of generality that  $\gamma=\zeta$
  and $\lambda\geq\kappa$.   Also\,  $\lambda\leq\card{P(\vec
  f)}\leq\kappa^{+(\zeta+1) }$.   
  Finally, Lemma~\ref{thm:approx} implies that $\lambda\not=\kappa^{+}$.
\end{proof}

In the forcing of Gitik from which this forcing is derived, a
preliminary forcing is used to define a morass-like structure
which guides the main forcing so that no cardinals are collapsed.
We omit this preliminary forcing as unnecessary for the proof
of the main theorem; however as a
consequence we do not know whether the cardinals of $M_\ords$ which
are excepted in Lemma~\ref{thm:PF-nocollapse} are cardinals
in the Chang model. 

\subsection{Introducing the equivalence relation}\label{sec:gkeqDef}

We now proceed to the second part of the definition of the forcing by
adding a variant of Gitik's equivalence relation $\gkeq$ on $P(\vec
F)$.    Recall that if $F$ is an extender on $\lambda$ then $\ufFromExt{F}{b}$ is
the ultrafilter $\set{x\in V_\lambda\mid b\in i^{F}(x)}$.

\begin{definition}\label{def:gkeq-a-set}
  Suppose that $\vec F$ is a suitable sequence of extenders of length
  at least $\gamma+1$ on a cardinal $\lambda$, and $a,a'\colon x\to 
  \supp(F_{\gamma})$ for some $x\subseteq[\lambda,\lambda^{+})$ of
  size $\lambda$.    Set $Y=\bigcup_{\gamma<\gamma'}\supp(F_\gamma')$.
  \begin{enumerate}
  \item 
    $a\gkeq_{0}a'$ if 
    $\ufFromExt{F_\gamma}{y\cup\sing{a}}=\ufFromExt{F_\gamma}{y\cup\sing{a'}}$
    for all $y\in[Y]^{<\omega} $.
  \item 
    If $n\geq0$ then  $a\gkeq_{n+1}a'$ if for all $b\supseteq a$
    there is $b'\supseteq a'$ such that $b\gkeq_{n}b'$, and for all
    $b'\supseteq a'$ there is $b\supseteq a$ such that $b\gkeq_{n}b'$.
  \end{enumerate}
\end{definition}

\begin{definition}\label{def:gkeq-a-seq}
    We write $\mathcal{N}$ for the set of sequences
    $\vec n\in{^{\zeta}{\omega}}$ such that
    $\set{\iota<\zeta\mid n_\iota< m}$ is finite for each
    $m\in \omega$.
    Suppose that $\vec F$ is a suitable sequence of extenders on $\lambda$
    and $\vec a$ and $\vec a'$ are sequences with $\domain(\vec
    a)=\domain(\vec a')=\domain(\vec F)\subseteq\zeta$.
    \begin{enumerate}
    \item If $\vec n\in\mathcal{N}$ then $\vec a\gkeq_{\vec n}\vec a'$
      if $a_{\nu}\gkeq_{n_\nu}a'_{\nu}$ for all $\nu\in\domain(\vec
      F)$.
    \item $\vec a\gkeq\vec a'$ if there is some $\vec n\in\mathcal{N}$
      such that $\vec a\gkeq_{\vec n}\vec a'$.
    \end{enumerate}
  \end{definition}

\begin{definition}\label{def:gkeq}
  The extension of $\gkeq_{\vec n}$ to $P^*_{\gamma}$ is by recursion on
  $\gamma$:   we assume that its restriction to  $P^*_{\eta}$ is defined for all
  $\eta<\gamma$.
  
  If $\eta<\gamma$ and $w,w'\in P_{\gamma,\eta}^{*}$ then  $w\gkeq_{\vec n}w'$ if
  \begin{myinparaenum}
  \item
    $w\restrict[\gamma_0,\eta]\gkeq_{\vec n} w'\restrict[\gamma_0,\eta]$,
    as members of $P^*_{\eta}$,  and
  \item
    $w\restrict[\eta+1,\gamma)=w'\restrict[\eta+1,\gamma)$.
  \end{myinparaenum}

  Suppose $t,t'\in P^*_{\gamma}$.   Then $t\gkeq_{\vec n}t'$ if the
  following conditions hold:
  \begin{enumerate}
  \item  $\forceKappa^{t}=\forceKappa^{t'}$ and  $\vec F^{t}=\vec
    F^{t'}$.
    
  \item $f^{t}_{\nu,\nu'}=f^{t'}_{\nu,\nu'}$ for all $\nu,\nu'$  for
    which they are defined.
  \item     $a^{t}_{\mu,\nu}\gkeq_{n_\nu}a^{t'}_{\mu,\nu}$ for all $\gamma\geq\mu>\nu$.
  \item\label{item:gkeq_Aeq} $ [A^{t}_{\nu}]_{\vec n}=[A^{t'}_{\nu}]_{\vec n}$ for all
    $\nu\in\domain(\vec F^{t})$, where $[A]_{\vec
      n}=\set{[w]_{\gkeq_{\vec n}}\mid w\in A}$.
  \end{enumerate}
  Finally, 
  $s\gkeq_{\vec n} s'$ for conditions $s,s'\in P(\vec F)$ if
  $\domain(s)=\domain(s')$ and 
  $s(\gamma)\gkeq_{\vec n}s'(\gamma)$ for all $\gamma$ in their
  common domain.
\end{definition}

It is easy to see that $\gkeq$ is an equivalence relation.  

\begin{proposition}
  \label{thm:one-step-gkeq}
  Suppose that $\add(s, \vec z)\leq s\gkeq_{\vec n} t$.  Then
  there is $ \vec w$ such that $\add(s, \vec z)\gkeq_{\vec n}\add(t,  \vec w)\leq t$.
\end{proposition}

\begin{proof}
  We show that this is true when $\vec z$ has length one.   An
  induction will then show that it is true in general.     .

  Suppose that $\add(s,z)\leq s\gkeq_{\vec n} t$, with $z\in
  A^{s,\tau}_{\gamma}$.  By 
  definition~\ref{def:gkeq}(\ref{item:gkeq_Aeq})
  there is  
  $w\in A^{t,\tau}_{\gamma}$ such that $z\gkeq_{\vec n}w$.    Then
  the condition $z\restrict[\gamma_0,\gamma]\gkeq_{\vec
    n}w\restrict[\gamma_0,\gamma]$ implies that 
  $\add(s,z)(\gamma)\gkeq_{\vec n}\add(t,w)(\gamma)$, and the
  condition that $z\restrict[\gamma+1,\tau)=
  w\restrict[\gamma+1,\tau)$ implies that the Cohen functions induced in
  $\add(s,z)(\tau)$ and $\add(t,w)(\tau)$ 
  by Definition~\ref{def:a2f} are equal.  Therefore
  $\add(s,z)(\tau)\gkeq_{\vec n}\add(t,w)(\tau)$.  

  Since these are the only values of $s$ and $t$ which are changed in
  the extensions, 
  it follows that $\add(s,z)\gkeq_{\vec n}\add(t,w)$.
\end{proof}

\begin{proposition}
  \label{thm:direct-gkeq}
  Suppose $s'\leq^* s\gkeq_{\vec n} t$, and that $n_{\nu}>0$ for all
  $\nu\notin\domain(s)$.  Then there is $t'\leq^* t$ such
  that $s'\gkeq_{\vec m} t'$ for all
  $\nu<\zeta$, where $m_{\nu} = n_{\nu} - 1$ if
  $n_{\nu}>0$, and $m_{\nu}=0$ otherwise, 
\end{proposition}
\begin{proof}
  We will show by induction on $\gamma$ that, under the hypotheses of the Proposition, if   $\gamma\in\domain(s)=\domain(t)$ then there is  $t'(\gamma)$ such that
  $t'(\gamma)\leq^* t(\gamma)$ and $t'(\gamma)\gkeq_{m_{\nu}}s'(\gamma)$.    By the definition of $\gkeq$, 
  the sequence  $\vec F^{t',\gamma}$ and the functions
  $f^{t',\gamma}_{\nu,\nu'}$ must  be the same as $\vec
  F^{s',\gamma}$ and $f^{s',\gamma}_{\nu,\nu'}$.    This leaves the
  functions $a^{t',\gamma}_{\nu',\nu}$ and sets $A^{t'}_{\nu}$ to be
  defined.  

  To define $\vec a^{t',\gamma}$, pick for each $\nu$ in the interval
  $\gamma_0\leq\nu<\gamma$ some $b\supseteq a^{t,\gamma}_{\nu+1,\nu}$ such that
  $a^{s',\gamma}_{\nu+1,\nu}\gkeq_{m_{\nu}} b$.   This is possible by
  the definition of $\gkeq_{m_{\nu}+1}$, since $n_\nu\not=0$.   Now set
  $a^{t',\gamma}_{\nu+1,\nu}=b$.   By clause~(\ref{item:asubset}) of 
  the Definition~\ref{def:tableau} of the tableau, this determines
  $a^{t',\gamma}_{\mu,\nu}=a^{t',\gamma}_{\nu+1,\nu}\restrict\domain(a^{s',\gamma}_{\mu,\nu})$
  for $\mu>\nu+1$.  

  Finally, set $A^{t',\gamma}_{\nu}$ equal to  the set of all $w'$ such
  that $w'\leq^* w$ for for some $w\in A^{t,\gamma}_{\nu}$ and
  $w'\gkeq_{\vec m} v'$ for some $v'\in A^{s',\gamma}_{\nu}$.     Then
  $ [A^{t',\gamma}_{\nu}]=[A^{s',\gamma}_{\nu}]$ since for all
  $v'\in A^{s',\gamma}_{\nu}$ there is $v\in A^{s,\gamma}_{\nu}$ and
  $w\in A^{t,\gamma}_{\nu}$ such that $v'\leq^* v\gkeq_{\vec m}w$, and
  then the induction hypothesis implies that there is $w'\leq^* w$ with
  $w'\gkeq _{\vec m} v'$.     
\end{proof}
\begin{definition}
  \label{def:modgekq}
  We will write $[s]$ for $[s]_{\gkeq}=\set{t\mid s\gkeq t}$.   The
  ordering on $P(\vec F)\mgkeq$ is the smallest transitive relation such
  that $[s]\leq[t]$ holds if either $s\leq t$  or  $s\gkeq t$.
\end{definition}

\begin{proposition}
  \label{thm:leq_trans1}
  Suppose $[t]= [s]$ and $t'\leq t$.  Then there are $s''\leq s$ and
  $t''\leq t'$ such that $[s'']=[t'']$.
\end{proposition}
\begin{proof}
  Suppose that $t\gkeq_{\vec n}s$.   By using a
  further extension $t''=\add(t',\vec w)$ we can arrange that
  $\set{\nu\mid n_{\nu}=0}\subseteq\domain(t'')$.  By
  Proposition~\ref{thm:one-stepFirst} there is $\vec z$ so that
  $t''\leq^* \add(t,\vec z)\leq t$.  By
  Proposition~\ref{thm:one-step-gkeq} it follows that there is $\vec
  w$ so that $\add(t,\vec z)\gkeq_{\vec n}\add(s,\vec w)\leq s$.  Finally it   follows by Proposition~\ref{thm:direct-gkeq} that there is
  $s''\leq^* \add(s,\vec w)$ so that $s''\gkeq t''$.
\end{proof}
\begin{proposition}\label{thm:equivnormal}
  Suppose that $[t]\leq [s]$.  Then there is a condition $q\leq s$
  such that $[q]\leq [t]$.
\end{proposition}
\begin{proof}
  \newcommand{\ltgkeq}{\mathbin{\genfrac{}{}{0pt}{}{\raisebox{-3pt}{$<$}}{\raisebox{1pt}{$\scriptstyle\gkeq$}}}}
  If $[t]\leq[s]$ then there is a sequence 
  $t=t_0\ltgkeq t_1\ltgkeq\cdots\ltgkeq t_{k-1}\ltgkeq t_{k}=s$, 
  where we write $s\ltgkeq s'$ to mean that  either $s\le s'$  or $s\gkeq s'$.
  We prove the
  proposition by induction on the length of the shortest such
  sequence, assuming as an induction hypothesis that there is $\bar
  q\le t_{k-1}$ such that $[\bar q]\leq[t]$.
  
  If $t_{k-1}\leq s$, then it follows that  $\bar q\leq s$ and we
  can take $q=\bar q$.   Otherwise $\bar q\le t_{k-1}\gkeq s$, and
  Proposition~\ref{thm:leq_trans1} asserts that there is $q\leq s$
  and $q'\le\bar q$ such that $q\gkeq q'$.  But then $[q]=[q']\leq
  [t]$, as required.
\end{proof}

\begin{corollary}\label{thm:iterated}
  $P(\vec F)$ is forcing equivalent to $(P(\vec F)\mgkeq)*\dot R$ where 
  $\dot R$ is a $P(\vec F)\mgkeq$-name  for a partial order.\qed
\end{corollary}

\begin{corollary}\label{thm:nocollapse-PFgkeq}
  Forcing with $P(\vec F)\mgkeq$ does not collapse any cardinal which is not in the set $\bigcup_{\gamma\leq\zeta}[\forceKappa_{\gamma}^{++},
  \forceKappa^{+\omega_1}_{\gamma})$.
\end{corollary}

\begin{proof}
  By Corollary~\ref{thm:PF-nocollapse}  this is true in the extension by $P(\vec F)=
  (P(\vec F)\mgkeq)*\dot R$; 
  hence it is certainly true in the extension by $P(\vec F)\mgkeq$.
\end{proof}

\subsection{Constructing a generic set}
\label{sec:generic_set}

Much of the argument in this subsection  is basically the same as
Carmi Merimovich's first genericity argument 
\cite[Theorem 5.1]{Merimovich2007Prikry-on-exten}.  
In order to construct a $M_B$-generic set we need to move outside of
$M_B$: we work in $V[h]$,  where $h$ is a generic collapse of
$\reals$ onto $\omega_1$ so that $\card{M[h]}=\omega_1$.   Since this Levy
collapse does not add countable sequences of ordinals  the Chang
model is unchanged, the ordering
$\le^*$ of $P(\vec 
N\restrict\zeta)$ is still countably complete, and  $M$ is still closed under
countable sequences.   Furthermore, since $h$ is generic over $M$,
$M[h]\supseteq M(\reals)$ and $M[h]$ is mouse over $h$ which has all
of the required
properties of $M$. 

\begin{lemma}[Generic set construction]\label{thm:generic_in_V_1}
  Let $h$ be a generic collapse of $\reals$ onto $\omega_1$ with
  countable conditions, and    
  let $B$ be a countable subset of $I$ with $\otp(B)=\zeta$. 
  Then there is, in $V[h]$,  an $i_{\ords}(M_B)$-generic set $G\subseteq
  i_{\ords}(P(\vec E\restrict\zeta)\mgkeq)$  such that
  every countable subset of $M_{B}$ is contained in $M_B[G]$.
\end{lemma}

\subsubsection{Proof of Lemma~\ref{thm:generic_in_V_1}}
\label{sec:generic_in_V_1_proof}
  Since $M_B\cong M_{B(\zeta)}$, where $B(\zeta)=\set{\kappa_\nu\mid
    \nu<\zeta}$,  containing the first $\zeta$ members of $I$, it will be
  sufficient to prove this for the case where $B=B(\zeta)$.
  This will simplify notation, since then $M_{B}\cut\ords$ is
  transtive and $\forceKappa_\nu^{G}$ is equal to both the $\nu$th
  member $\kappa_\nu$ of $I$ and the $\nu$th member of $B$.

  We define a partial order $R$.   Our assumptions on $M$ are
  sufficiently generous that the definition of $R$ can be made inside
  $M$, using $\seq{N_{\xi}\cap H^{M}_{\tau}\mid \xi<\omega_1}$, for
  some sufficiently large cardinal $\tau$ of $M$, instead
  of $\seq{N_{\xi}\mid\xi<\omega_1}$.

  \begin{definition}
    $R=\bigcup_{\xi<\omega_1} R_{\xi}$, where $R_{\xi}$ is defined as
    follows: The members of $R_{\xi}$ are the pairs $([s],b)$ such that
    $[s]\in P(\vec E\restrict\delta)\mgkeq$ is a condition with $\domain(s)=\sing{\zeta}$ and
    $b=\seq{b_\gamma:\gamma<\zeta}$ where each $b_{\gamma{}}$ is a
    function in $N_{\xi}$ satisfying the following three conditions:
    \begin{enumerate}
    \item $\domain(b_\gamma)=\domain(a^{s,\zeta}_{\gamma+1,\gamma})$ for each
      $\gamma<\zeta$,
    \item $\range(b_\gamma)\subset [\kappa, \kappa^{+\omega_1})$ for each
      $\gamma<\zeta$, and
    \item \label{item:Rgkeq}
      $\seq{a^{s,\zeta}_{\gamma+1,\gamma}\mid \gamma<\zeta}\gkeq  b_{\gamma}$.
    \end{enumerate}
    The ordering of $R$ is $(s',b')\leq(s,b)$ if $[s']\leq [s]$ in $P(\vec
    N)\mgkeq$ and $(\forall\gamma<\zeta)\; b'_\gamma\supseteq
    b_\gamma$.  
  \end{definition}

  Clause~(\ref{item:Rgkeq}) requires some explanation, since
  $\range(b_{\gamma})\nsubset \supp(E_\gamma)=\supp(E)\cap N_\gamma$.
  The Definition~\ref{def:gkeq-a-set}   of  the relation
  $a\gkeq_n a'$ uses the parameter $\gamma$ in two ways.
  The first use is in the definition of $a\gkeq_{0}a'$, where the set
  $Y=\bigcup_{\gamma'<\gamma}\supp(F_{\gamma'})$ is used as the set of
  $y$ in the requirement
  $(F_{\gamma})_{y\cup\sing{a}}=(F_{\gamma})_{y\cup\sing{a'}}$.
  Here the same set $Y$ is used, and since
  $\ufFromExt{E_\gamma}{y\cup\sing{a}}=\ufFromExt{E}{y\cup\sing a}$
  the requirement can be altered to 
  $\ufFromExt{E_{\gamma}}{y\cup\sing{a}}=\ufFromExt{E}{y\cup
    \sing{b}}$. 

  The second way in which the parameter $\gamma$ is used is in the
  domain of the quantifiers.   In Clause~(\ref{item:Rgkeq}) the
  extensions $a'\supseteq a^{s,\zeta}_{\gamma+1,\gamma}$ are in
  $M_{\gamma}$, while the extension $b'\supseteq b_{\gamma}$ are in
  $M$.   We reconcile these demands by using the elementarity of
  $N_{\gamma}$, and this requires expressing Clause~(\ref{item:Rgkeq})
  as a first order statement.  This is achieved by the following
  Proposition, which is the reason for the requirement in
  Definition~\ref{def:Nsequence} that
  $\card{\bigcup_{\gamma'<\gamma}N_{\gamma'}}^{++}\subseteq N_\gamma$.  

  \begin{proposition}
    \label{thm:gkeq0inNgamma}
    For any $b\colon x\to [\kappa,\kappa^{+\omega_1})$ with $x\in
    [\kappa^{+}\setminus\kappa]^{\kappa}$, there is a
    formula $\phi(n,a)$, with parameters from $N_{\gamma}$, such that if 
    $a\colon x\to\supp(E_{\gamma})$ then 
    $a\gkeq _n b$ if and only if $N_\gamma\models \phi(n,a)$.
  \end{proposition}
  \begin{proof}
    For $n=0$, note that 
    the sequence of ultrafilters $\seq{\ufFromExt{E}{y\cup\sing{b}}\mid y\in
      [Y]^{<\omega}}$ can be coded as a subset of $[Y]^{<\omega}\times
    \ps(\kappa)$, which has cardinality $\card
    Y=\card{\bigcup_{\gamma'<\gamma}N_{\gamma'}}$.

    Working in $M$, define $T$ to be the tree
    of finite sequences of the form $\seq{[b_i]_{\gkeq_0}\mid i<k}$
    where $\seq{b_i\mid i<k}$ is a $\subseteq$-increasing sequence of
    functions $b_i\colon x_i\to[\kappa,\kappa^{+\omega_1})$ with  $x_i\in
    [\kappa^{+}\setminus\kappa]^{\kappa}$. 
    Since $T$ is at most $\card{\ps(Y)}$-branching,
    it has cardinality at most $\card {\ps^{2}(Y)}$, so
    Clauses~(\ref{item:cardNnuSubsetNnu}) 
    and~(\ref{item:Nseq-doublepluss}) of
    Definition~\ref{def:Nsequence} ensure that $T\in N_{\gamma}$.

    Write $T_{b}$ for the portion of $T$ above $\seq{[b]_{\gkeq_0}}$.
    Then  the conclusion of the proposition is satisfied by the
    formula $\phi(n,a)$, with parameter $T_{b}$, 
    which asserts that the first $n$ levels of $T_{b}$ and
    $(T_{a})^{N_\gamma}$ are equal.   Since
    $[\supp(E_\gamma)]^{\kappa}\cap N_\gamma\in N_\gamma$, this is a
    first order formula over $N_\gamma$.
  \end{proof}
  \begin{todoenv}
    (7/21/15) --- for future --- Note that the $\le^*$ forcing order
    is homogeneous: In particular to conditions with function $f$ and
    $a$ are essentially independent of the domain of $f$ and $a$
    except for terms that explicitly ask for $f_{\zeta,\nu}(\xi)$ for
    $\xi$ in the domain of one of the affected functions.
  \end{todoenv}
  \begin{todoenv}
    (7/3/15) --- For future --- Note that the $\gkeq_{n}$-class is
    essentially a property of $[\supp(E)]^{\kappa}$.    We can, for
    example, assume that $c_0=\emptyset$ and $c_{i+1}\setminus c_i$
    has domain 
    $[\nu_i,\nu_{i+1})$ where $\nu_0=\kappa$ and $\nu_{i+1}=\nu_i
    +\otp(\range(c_{i+1}\setminus\range(c_{i}))$ and $c_{i+1}\setminus
    c_{i}$ is the increasing enumeration.   Then the $\gkeq_{n}$-types
    of a sequence of $c'_i$ with the same ranges is determined by that
    of $\seq{c_i}_i$ together with the function $\rho$ such that
    $c'_i=c_i\circ \rho$.
  \end{todoenv}

  \begin{lemma}\label{thm:density-in-R}
    \begin{enumerate}
    \item \label{item:x} $\xset{([s],b)}{s\in D}$ is dense in $R$ for
      each $\leq^*$-dense set $D\subseteq P(\vec E\restrict\zeta)$ in
      $M$.
    \item
      \label{item:y}
      Suppose $\gamma<\zeta$ and $\beta\in [\kappa,\kappa^{+\omega_1})$.
      Then there is a dense subset of conditions $([s],b)\in R$ such
      that $b(\xi)=\beta$ for some $\xi\in\domain(a^{s,\zeta}_{\zeta,\gamma})$.
    \end{enumerate}
  \end{lemma}

  \begin{proof}
    For clause~(\ref{item:x}), 
    let
    $([s],b)\in R$ be arbitrary and set $\vec
    a=\xseq{a^{s,\zeta}_{\gamma+1,\gamma}}{\gamma<\zeta}$.  We may assume that
    $a_{\gamma}\gkeq_{1}b_{\gamma}$ for each $\gamma<\zeta$; if not,
    then replace each such $a_{\gamma}$ with some $a'_{\gamma}$ such that
    $a'_{\gamma}\gkeq_{0}a_{\gamma}$ and
    $a'_{\gamma}\gkeq_{1}b_{\gamma}$.  This is possible by
    Proposition~\ref{thm:gkeq0inNgamma} and the
    elementarity of the structures $N_{\xi}$, since $b_{\gamma}$ has
    the desired properties.  This change
    only involves finitely many of the functions $a_{\gamma}$, so the
    condition obtained from $s$ by making this substitution is still
    in $[s]$.

    Now pick $s'\leq^* s$ in $D$.  Because of the assumption we made
    on $s$, Proposition~\ref{thm:direct-gkeq} implies that there is
    $b'\gkeq a^{s',\zeta}$ such that
    $([s'],b')\leq([s],b)$.
    \smallskip{}

    The proof for clause(\ref{item:y}) is 
    similar.  Fix $([s],b)\in R$, and assume that
    $a^{s,\zeta}_{\gamma+1,\gamma}\gkeq_{1}b_{\gamma}$ for all $\gamma<\zeta$. 
    Now fix $\mu<\omega_1$ so that $\sing{b,\eta}\subset N_{\mu}$ and 
    extend $b$ to $b'\in N_\mu$ by setting $b'_\nu(\xi)=\eta$ for some 
    $\xi$ which is not  in the domain of any function in~$s$.    Then there
    is $ a'_{\nu}\supset  a_{\nu}$ so that 
    $ a'_{\nu}\gkeq_{0} b_{\nu}$.   Now extend $s$ to $s'$ by
    setting $a^{s',\zeta}_{\gamma',\gamma}=a'(\xi)$ for all $\gamma'\in(\gamma,\zeta]$.
  \end{proof}

  The ordering $(P(\vec N)\mgkeq, \leq^*)$ is not countably complete:
  it is easy to find an infinite descending sequence of conditions
  $\seq{[s_n]\mid n<\omega}$ such that any lower bound would require an ultrafilter
  concentrating on non-well founded sets of ordinals.  However the
  partial order $R$ is countably complete due to the guidance of the
  second coordinate $b$:

  \begin{lemma}\label{thm:Rcomplete}
    The partial order $R$ is countably closed.
  \end{lemma}

  \begin{proof}
    Suppose that $\seq{([s_n], b_n)\mid n<\omega}$ is a descending
    sequence in $R$.  We define a lower bound
    $([s_{\omega}],b_{\omega})$ for this sequence.  The definition of $R$ determines 
    $b_{\omega,\nu}=\bigcup_{n<\omega}b_{n,\nu}$, and determines all of 
    $s_{\omega}$  except for the functions $a^\omega_\nu=a^{s_{\omega},\zeta}_{\nu+1,\nu}$.
    It also determines  $\domain(a^\omega_\nu)=
    \domain(b_{\omega,\nu})= \bigcup_{n<\omega}\domain(a^{s_n,\zeta}_{\zeta,\nu})$. 
    Pick any  $\vec n=\seq{n_\nu\mid\nu<\zeta}\in\mathcal{N}$, and for each
    $\nu<\zeta$ pick $a^\omega_\nu\in N_\nu$ so that
    \begin{equation*}
      a^\omega_\nu\restrict\domain(a^n_\nu)\gkeq_{k_{n,\nu}}a^n_\nu
      \quad\text{and}\quad 
      a^\omega_\nu\gkeq_{n_{\nu}}b_{\omega,\nu}
    \end{equation*}
    where $a^n_\nu\gkeq_{k_{n,\nu}} b_{n,\nu}$.  This
    is possible by the elementarity of the models $N_{\xi}$, since
    $b_{\omega,\nu}$ satisfies these conditions.    Then
    $([s_{\omega}],b_{\omega})\in R$ and
    $([s_{\omega}],b_{\omega})\le ([s_{n}],b_{n})$ for each $n\in\omega$.
  \end{proof}
  
  We are now ready to construct the desired $M_B$-generic set 
  $G\subset i_{\ords}(P(\vec E\restrict\zeta)\mgkeq)$, where $\zeta=\otp(B)$.
    
  \begin{definition}[The generic set $G$]
    \label{def:G}
    Let  $H\subset R$ in $V[h]$ be an
    $M$-generic set.   Such a set can be constructed in $V[h]$ using 
    Lemma~\ref{thm:Rcomplete},  since   $\card M^{V[h]}=\omega_1$ and
    and $^{\omega}M\subseteq   M$. 
    
    We set
    \begin{equation}\label{eq:a1}
      G=\set{[s']\mid (\exists
      ([s],b)\in H)(\exists\vec \gamma\in[\zeta]^{<\omega} )\, 
      s'\geq^* \add(i_{\ords}(s),\vec w(s,b,\vec\gamma))}
  \end{equation}
  where  $\vec w(s,b,\vec \gamma)$ is defined as follows: Set
  $n=\len(\vec\gamma)$.   Then 
  $\vec w(s,b,\vec\gamma)=\seq{i_{\gamma_i}(w_i)\mid  i<n}$ , where 
  \begin{align}
    w_i\restrict[0,\gamma_i]&=\add(s,w_i)\restrict[0,\gamma_i]
    && \text{and}\notag\\ 
    a^{w_i}_{\gamma,\gamma_i}&=
                               b_{\gamma,\gamma_i}\restrict\domain(a^{s,\gamma_i}_{\gamma,\gamma_i})&
                             &\text{for $\zeta\geq\gamma>\gamma_i$.}
                               \label{eq:a2}
  \end{align}
  \end{definition}

  Note that $w_i\gkeq s\scutdown\gamma_i$ and therefore
  $[\add(i_{\ords}(s), i_{\gamma_i}(w(s,b,\vec\gamma)))]\leq [i_{\ords}(s)]$.
  The effect of the substitution used  in equation~\eqref{eq:a2} to define $w_i$ is that
  \begin{equation*}
  [\add(i_{\ords}(s), i_{\gamma_i}(w_i))]\forces
  h_{\zeta,\gamma_i}(\xi)=b_{\gamma_i}(\xi)\qquad \text{for all}\qquad
  \xi\in\domain(a^{s,\zeta}_{\zeta,\gamma_i}). 
\end{equation*}
  In looking at the Chang model inside of $M_B[G]$, it is important to
  recall that the set $T$ terms specified for the sharp of $\chang$
  provides a set, inside $M$, of names for the members of $\chang_B$.   
  Definition~\ref{def:standard-names} below makes this more specific,
  and provides a set of names inside $M$ for the members of $M_B$ and
  for $\chang^{M_B}$, and then provides standard \emph{forcing} names
  which are useful inside $M_B[G]$; however the notation in the next
  definition is sometimes useful.
  \begin{definition}
    \label{def:namenotation}
    We write $\bar i_{\gamma}$ for the embedding $i_{\gamma'}$ where
    $\gamma'$ is the ordinal such that the
    $\gamma$th member $\forceKappa_{\gamma}$ of $B$ is equal to
    $\kappa_{\gamma'}$.   

    If $\tau$ is an expression then we write $\gn{\tau}$ to indicate
    that $\tau$ is being used as a name for the value of the expression.
  \end{definition}

  \begin{definition}
    \label{def:standard-names}
    A  \emph{standard name} for a member of $M_B$ is a term
    obtained recursively as follows:
    \begin{enumerate}
    \item\label{item:stdname_MB_gen}
      If $\gamma\leq\zeta$ and $\bar\beta\in
      [\kappa,\kappa^{+\omega_1})$ then $\gn{\bar
        i_{\gamma}(\bar\beta)}$ is a standard name for the generator
      $\beta=\bar i_{\gamma}(\bar\beta)$ belonging to $\forceKappa_{\gamma}$.
      
    \item \label{item:stdname_MB_fctn}
      If $f\in M$ and $x$ is a finite sequence  of standard
      names of  generators $\beta_i$ in $M_B$, then  $\gn{i_{\ords}(f)(x)}$ is 
      a standard name for the value $i_{\ords}(f)(\vec{\beta})$.
    \end{enumerate}
    
    A standard name for a member of $\chang$ is a term obtained
    recursively using clause~(1) above and the following two operations:
    \begin{enumerate}
      \setcounter{enumi}{2}

      \item[$\arabic{enumi}'$.]  \label{item:stdname_C_fctn}
      If $\alpha$ is an ordinal, then a standard name for
      $\alpha\in M_B$ from clause~(2) above is also a standard name for $\alpha\in\chang$.
      
    \item\label{item:stdname_C_recur}
     Suppose that $i$ is a standard name for an ordinal $\iota$ and
     that  $\vec\tau$ is a countable sequence of standard forcing
     names for ordinals $\vec\beta=\seq{\beta_k\mid k\in\omega}$.
       Then 
      $\gn{\set{x\in\chang_{i}\mid
          C_i\models\phi(x,\vec\tau)}}$ is a standard name for $\set{x\in
        C_{\iota}\mid C_\iota\models \phi(x,\vec\beta)}$. 
    \end{enumerate}
    
    The definition of a  \emph{standard forcing name} is identical in
    both cases, 
    except that clause~1 is replaced with the following:
    \begin{enumerate}
    \item[$1'$.]
      Suppose $([s],b)\in H$,
      $\xi\in\domain(a^{s,\gamma}_{\zeta,\gamma})$ and
      $b_{\gamma}(\xi)=\bar \beta$, so 
      that
      \begin{equation*}
        ([s],b)\forces_{R} M_B[G]\models h_{\zeta,\gamma}(\xi)=i_{\gamma}(\bar\beta).
      \end{equation*}
      Then $\gn{h_{\zeta,\gamma}(\xi)}$ is a standard forcing name for
      $\beta=i_{\gamma}(\bar\beta)$, and is said to be
      \emph{established} by the condition $([s],b)$.
    \end{enumerate}
    An arbitrary  standard forcing name $\tau$ is established by $([s],b)$ if
    this condition establishes all names $\gn{h_{\zeta,\gamma}(\xi)}$ occurring in $\tau$.
  \end{definition}

  \begin{claim}\label{thm:Ggeneric-claim}
    $G$ is $M_B$-generic.
  \end{claim}
  \begin{proof}
    Let $D\subseteq i_{\ords}(P(\vec E\restrict\zeta)\mgkeq)$ be
    dense, and let \[\dot D=i_{\ords}(d) (\seq{h_{\zeta,\gamma_i}(i_{\ords}(\xi_i))\mid i<k})\] be
    a standard forcing name for $D$, established by a condition
    $(s,b)\in R$.    Thus for any $\vec
    w\in\prod_{i<k}A^{s,\zeta}_{\gamma_i}$,  the condition
    $\add(s,\vec w)$ decides the values  of each of the $(P(\vec
    E\restrict\zeta)\mgkeq$)-names
    $h_{\zeta,\gamma_i}(\xi_i)$ and hence determines the value of
    $d(\seq{h_{\zeta,\gamma_i}(\xi_i)\mid i<k})\subseteq P(\vec
    E\restrict\zeta)\mgkeq$. 
    We write $d(\vec
    w)$ to denote this value.     

    Since $D$ is dense, 
    \begin{align*}
      A=\set{\vec w\in\prod_{i<k}A^{s,\zeta}_{\gamma_i}\mid d_{s}(\vec w)
      \text{ is dense in }P(\vec E\restrict\zeta)\mgkeq}\in\prod_{i<k}U^{s,\zeta}_{\gamma_i}
    \end{align*}
    so we can assume that $d_{s}(\vec w)$ is dense in $P(\vec
    E\restrict\zeta)\mgkeq$ for all $\vec w\in\prod_{i<k}A^{s,\zeta}_{\gamma_i}$.
    Using Lemma~\ref{thm:prikry}(\ref{item:Prikrythm-inD})
    and Lemma~\ref{thm:closure}(\ref{item:closurediagonal}), it can be
    shown that there is $s'\leq^*s$ such that 
    \begin{equation*}
(\forall \vec 
    w\in\prod_{i<k}A^{s',\zeta}_{\gamma_i})(\exists e\in[\zeta]^{<\omega})(\forall t\leq
    s')\;\bigl(e\subseteq\domain(t)\implies [t]\in d_{s}(\vec w)\bigr). 
  \end{equation*}

  Since $[\zeta]^{<\omega}$ is countable, 
  we can furthermore assume that $e$ does not depend on $\vec w$.
  But     now we are done, for if $b'$ is such that $(s',b')\leq(s,b)$ 
  in $R$ and $e\subseteq\vec\gamma $ then 
  $\add(i_{\ords}(s'), \vec w(s',b',\vec\gamma))\in D\cap G$.
\end{proof}

This completes the proof of Lemma~\ref{thm:generic_in_V_1}.
\subsubsection{Defining $\chang_{B}$ in  $M[G]$}

It follows immediately from the genericity of $G$ that 
  \begin{corollary}
    \label{thm:suitableCB}
   $\chang_{B}=\chang^{M_B[G]}$  for any suitable sequence $B$.\qed
  \end{corollary}

  Here $\chang^{M_B[G]}=\chang_B^{M_B[G]}$ is the set defined inside
  $M_B[G]$ using the definition of $\chang$ given in the first
  paragraph of this paper.  
  The more important case of a limit suitable set $B$ is more delicate since
  $M_{\tilde B}$ is not definable inside $M_B[G]$ for suitable $\tilde
  B\subset B$.   The following is the promised precise definition of
  $\chang_B$:
  \begin{definition}
    \label{def:changBdef}
    Suppose $B$ is a \LS\ set,  and let $B'\subset B$ be the set of heads of
    gaps in $B$.   Call a countable set $v\in M_B$ of ordinals
    \emph{$B$-bounded} if for all $\lambda\in B'$ and 
    $f\colon[\ords]^{<\omega}\to\ords$ in $M_B$,  the set
    $f[\,[v]^{<\omega}]\cap\lambda$ is bounded in $M_B\cap\lambda$.
    Let $\mathcal{C}$ be the set of $B$-bounded sets.
    Then $\chang_B$ is the set $L_{\ords}^{M_B[G]}(\mathcal{C})$,
    constructed by recursion over the
    ordinals in $M_B\cap\ords$ as in the first paragraph of this paper
    using countable sequences from 
    $\mathcal{C}$.
  \end{definition}
  Note that $\chang_{B}$ is definable inside $M_B[G]$.   The following
  Proposition implies that  Definition~\ref{def:changBdef}
  is equivalent to the more
  informal one given in section~\ref{sec:proof-start}.

\begin{proposition}
  \label{thm:limitSuitableC_Bdefinable}
  A countable sequence $\vec\nu$ of ordinals in $M_B$ is $B$-bounded
  if and only  if there is a suitable $\tilde B\subset B$ such that
  $\vec\nu\in M_{\tilde B}$.
\end{proposition}
\begin{proof}
  It is easy to see that if $\tilde B$ is suitable then every
  countable  $\vec\nu\subset M_{\tilde B}$ is $B$-bounded.    For the
  converse, suppose that $\vec\nu$ is $B$-bounded and take for
  each $\nu_k\in \vec\nu$ a function $g_k\in M$ and finite sequence of
  generators $e_k$ for cardinals in $B$ such that
  $\nu_k=i_{\ords}(g_k)(e_k)$; taking for each $k$ the least possible
  sequence $e_k$ in the 
  usual well order of finite sets of  ordinals: 
  $e'\prec e\iff\max((e\cup e')\setminus (e\cap e'))\in e$.   
  
  Now let $f_k$ be the pseudoinverse of $g_k$ defined by setting
  $f_k(\nu)$ equal to 
  the $\prec$-least finite sequence $e$ such that $\nu=g_k(e)$.
  Then every member of $i_{\ords}(f_k)(\nu_k)$ is a generator for some
  member of $B$, for otherwise let $\xi$ be the largest counerexample, 
  $\xi=\max(i_{\ords}(f_k)(\nu_k)\setminus B)$.  Then there is a
  function $h\in M$ 
  and a set $e''\subset\xi$ of generators for members of $B$ such
  that 
  $\xi=i_{\ords}(h)(e'')$, but
  $e''\cup  f_{k}(\nu_k)\setminus\sing{\xi}\prec f_k(\nu_k)\preceq
  b_k$,   contradicting the minimality of $b_k$.
  
  Now $\vec \eta=\bigcup_{k\in\omega}f_{k}(\nu_i)$ is $B$-bounded: suppose to
  the contrary that $f[\vec\eta]$ is unbounded in $\lambda\cap M_B$ where
  $\lambda$ is the head of a gap  in $B$. Then $f\circ g[ \vec\nu]$ is
  also unbounded in $\lambda$, where
  $g(\nu)=\sup_{k\in\omega}(f_k(\nu)\cap\lambda)$, and this contradicts
  the assumption that $\vec\nu$ is $B$-bounded.   Finally, the set of
  $\lambda\in B$ which have a generator in $\vec\eta$ is also
  $B$-bounded, and it follows that it is contained in a suitable
  subset $\tilde B\subset B$.  
\end{proof}

\subsection{Proof of the Main Lemma}
\label{sec:proof-main-lemma}

The purpose of this subsection is to prove 
Lemma~\ref{thm:mainlemma} under the  additional assumption that
$\kappa=\kappa_0$ is a member of the \LS{} set $B$.   The following
Subsection~\ref{sec:finite-exceptions} will  complete the proof of
Lemma~\ref{thm:mainlemma}, and hence of 
Theorem~\ref{thm:main}, by removing this
assumption.   
In the process it wil indicate the technique used to prove the
stronger result Theorem~\ref{thm:modified-suitable}.

Before beginning the proof, we state two general facts about iterated
ultrapowers.    Both are well known facts, but we need to verify that
they are valid in the context in which they will be used.

A full statement of the conditions under which these properties hold
is somewhat delicate, so we will restrict consideration to the
iterated ultrapowers needed here. If $k$ and $k'$ are iterated
ultrapowers, then we write $k[k']$ for the copy map, that is, the
direct limit of the maps $i^{k(U)}$ where $U=(F)_b$ for some extender $F$ used
in the iteration $k'$ and some generator $b$ for $F$.  Every extender 
$F$ used satisfies $k(F) = k[F]$ for any iteration $k$ such that  $\crit(F)$ is not
moved.   In the following the term \emph{extender} means an extender
with this property which does not overlap any measurable cardinals.

\begin{lemma}
  \label{thm:commute}
  Suppose $\kappa'\leq\kappa$, $E'$ is an extender on $\kappa'$, and
  $E$ is an extender on $\kappa$.
  Suppose further that if $\kappa'=\kappa$ then $E'\tless E$.
  Then the following diagram commutes:
\begin{equation}
        \label{eq:diagram_commute_1}
        \begin{tikzpicture}[>=angle 90,baseline]
         \matrix(m)[matrix of math nodes,
         row sep =2em, 
         column sep=4.5em, text height = 1.5ex, text depth=0.25ex
         ]
           {\ult(V,E)&\ult(V,E\times E')\\ V&\ult(V,E') \\};
          \path[->] (m-2-1)  edge node[auto, swap] {$i^{E'}$} (m-2-2) 
                                               edge node  [auto] {$i^E$}    (m-1-1) 
                             (m-2-2)     edge node  [auto] {$i^{E}$}   (m-1-2)
                             (m-1-1)     edge node[above] {$i^{E}[i^{E'}]$}(m-1-2.west);
           \path[->, dashed] (m-2-1) edge node[above] {$i^{E\times E'}$}(m-1-2);
        \end{tikzpicture}
    \end{equation}
  \end{lemma}
  \begin{proof}
      The diagram~\eqref{eq:diagram_commute_1} is the direct limit of the same diagram for the 
      ultrafilters $\ufFromExt{E}{a}$ and  $\ufFromExt{E'}{b}$, where $a$ and $b$ are generators of $E$ and $E'$ respectively. 
    \end{proof}

    \begin{corollary}
      \label{thm:reorderiteration}
      Any iteration  can be rearranged to an equivalent iteration with strictly increasing critical points.\qed
    \end{corollary}

The second statement is a variant of Kunen's result in 
\cite{Kunen1973A-model-for-the} that for any ordinal $\alpha$ there
are at most finitely many cardinals having a measure $U$ such that
$i^U(\alpha)>\alpha$.  The statement of the  following lemma is tailored to its use
in the proof of the main Lemma:
\begin{todoenv}
  (7/13/15) --- 
  Note that a more general version of this next lemma is actually
  used.   I can't think of a good way to state it.
\end{todoenv}
    \begin{lemma}\label{thm:M_bDoesntmove}
      Suppose that $b$ is a finite subset of $I$, $B\subset I$ is suitable, and
      $k$ is an iteration in $M_\ords[B]$ of length less than
      $\omega_2$ which  uses only
      extenders of the form $i_{\nu}(\ufFromExt{E}{\alpha})$ where
      $\kappa_\nu\in B\setminus b$ and $\alpha<\omega_1$.   Then
      $k\restrict (\ords\cap M_b)$ is the identity.    .
    \end{lemma}
    The proof uses the following lemma.   We write $\Crit(k)$ for the set of  critical
    points of the extenders in the iteration $k$.  Note that the
    hypothesis implies that $k'[k]=k'(k)$ for any iteration $k'$ which is
    the identity on $\Crit(k)$.

   \begin{lemma}\label{thm:M_bDoesntmove_helper}
      Suppose $b\subseteq I$ is finite and $\alpha\in M_b$.   Then
      there is a sequence
      $\vec\nu=\seq{\nu_\lambda\mid \lambda\in b\cup\sing{0}}$ in
      $M_b$ satisfying 
$
        (\forall\lambda\in
        b)\;\lambda\leq\nu_\lambda<\min(\sing{\ords}\cup
        b\setminus\lambda+1)      
$
      which has the following property:  Let  $k\in
      M_\ords$ be any iteration of length less than $\kappa_0$ such
      that $\Crit(k)\cap[\lambda,\nu_\lambda]=\emptyset$ for all
      $\lambda\in\sing{0}\cup b$.  Then $k(\alpha)=\alpha$.
    \end{lemma}
    Note that the statement of this lemma is first order, and hence
    it is also valid (using the image of the same sequence $\vec\nu$)
    in any iterated ultrapower of $M_{\ords}$.
\begin{proof}
  The proof closely follows that of Kunen.  We will work inside
  $M_{\ords}$, but the fact that  $M_b\prec M_\ords$ ensures that the ordinals
  $\nu_\lambda$ are members of $M_b$.

  We will suppose that the lemma is false for
  $b$ and $\alpha$.
  Set $\bar b=\sing{0}\cup b\cap\tau$, where $\tau\in b$ is
  least such that there is no sequence
  $\seq{\nu_{\lambda'}\mid\lambda\in \sing{0}\cup b\cap\tau}$ which satisfies the
  conclusion for iterations $k$ with $\Crit(k)\subset\tau$. Note
  that $\tau\leq\max(b\cap\alpha)$, since $\nu_{\max(b\cap\alpha)}$
  can be $\alpha$.  Set $\bar\tau=\max(\bar b)$, 
  let
  $\seq{\nu^{0}_{\lambda}\mid \lambda\in\bar b\cap\bar\tau}$ witness that
  $\tau$ is minimal, and set $\nu^0_{\bar\tau}=\cof(\alpha)$ if
  $\max(\bar b)\leq\cof(\alpha)\leq\max(b)$, and 
  $\nu^0_{\bar\tau}=\bar\tau$ otherwise.   Following Kunen, the
  failure of the lemma implies that there is  an infinite sequence
  $\seq{\kappa_n\mid n\in\omega}$ of iterations such that
  \begin{gather*}
    (\forall n\in\omega)\;k_n(\alpha)>\alpha,\\
    (\forall\lambda\in\bar
    b)\;\min(\Crit(k_0))\setminus\lambda>\nu^{0}_{\lambda},\qquad\text{and}\\
    (\forall\lambda\in\bar b)(\forall 
    n\in\omega)
    \;\min\bigl(\Crit(k_{n+1})\setminus\lambda\bigr)>\sup\bigl(\Crit(k_{n})\cap\min(b\setminus\lambda+1)\bigr).   
  \end{gather*}
  Now set $k'_{0,1}=k_0\colon M_{\ords}=N_0\to N_1$ and
  $k'_{n,n+1}=k'_{n}[k_{n}]\colon N_{n}\to N_{n+1}$.   Then the direct
  limit $N_{\omega}$ of these iterations is well founded; however the
  following claim implies that $\seq{k'_{n,\omega}(\alpha)\mid
    n\in\omega}$ is strictly descending.  This contradiction will
  complete the proof of Lemma~\ref{thm:M_bDoesntmove_helper}.
  \begin{claim}\label{thm:M_bDoesntmove_claim}
    $k'_{n,n+1}(\alpha)>\alpha$ for each $n\in\omega$.
  \end{claim}
  \begin{proof}[Proof of Claim~\ref{thm:M_bDoesntmove_claim}]
    Set $\ell=k'_{n}$ and $\ell'=k_{n+1}$, and write $\ell=\ell_{1}\circ
    \ell_{0}$ and $\ell'=\ell_{1}'\circ \ell_{0}'$, where $\ell_0$ and
    $\ell_0'$ use the extenders below $\bar\tau$, while $\ell_1$
    and $\ell'_1$ use the extenders above $\bar{\tau}$.   Now
    consider the following diagram,  which is obtained using
    Corollary~\ref{thm:reorderiteration}.
    \begin{equation}
      \label{eq:zzzz}
        \begin{tikzpicture}[>=angle 90,baseline]
         \matrix(m)[matrix of math nodes,
         row sep =1.5em, 
         column sep=4.4em, text height = 1.5ex, text depth=0.25ex
         ]
           {
             			&  M^{\ell_0}  &		   		&  M^{\ell_1,h_0}				\\
             M		&			&M^{h_0}		&				&M^{h_1 }	\\
             			& M^{\ell'_0}	&				& M^{\ell'_1,  h_0}				\\
                                &			& M^{\ell'}										\\
         };
          \path[->] 
                         (m-1-2) edge node[auto,sloped,pos=.2]{$\ell_0[\ell'_0]$} (m-2-3)
                         (m-1-4) edge node[auto,sloped,pos=.2] {$\ell[\ell'_1]$} (m-2-5)
                         (m-2-1) edge node[auto]{$\ell_0$} (m-1-2)
                                       edge  node[auto,swap]{$\ell'_0$}(m-3-2)
                                       edge[dotted] node[auto]{$h_0$} (m-2-3)
                         (m-2-3) edge node[auto,sloped,pos=.8] {$ \ell'_0[\ell_1]$} (m-1-4)
                                       edge node[auto,swap,sloped,pos=.8]  {$ \ell_0[\ell'_1]$} (m-3-4)
                                       edge[dotted] node[auto]{$h_1$} (m-2-5)
                         (m-3-2) 	edge node[auto,pos=.2,swap,sloped]{$\ell'_0[\ell_0]$} (m-2-3)
                         		edge node[auto,swap]{$\ell'_1$} (m-4-3)
                         (m-3-4) edge node[auto,swap,sloped,pos=.2] {$\ell'[\ell_1]$} (m-2-5)
                         (m-4-3) edge node[auto,swap,sloped,pos=.2] {$\ell'[\ell_0]$} (m-3-4);
        \end{tikzpicture}
      \end{equation}

      The choice of $\seq{\nu^{0}_{\lambda}\mid \lambda\in\bar b}$
      implies that $h_0(\alpha)=\alpha$,  so
      \begin{equation*}
        \ell_0[\ell'_1](\alpha)=
        \ell_0[\ell'_1]\circ\ell'_0[\ell_0](\alpha)=\ell'[\ell_0]\circ\ell'_1(\alpha)
        \geq\ell'_1(\alpha)=\ell'(\alpha)>\alpha. 
      \end{equation*}
    
        We will embed $\ell_0[\ell'_1](\alpha)$ into
      $\ell'_0[\ell_1](\ell'_1)(\alpha)$, showing that the latter is also
      greater than $\alpha$.  To this end, let $g$ and $\gamma$ be a
      function in $M^{h}$ and a generator of $\ell_0[\ell'_1]$ such
      that $\ell_0[\ell_1'](g)(\gamma)<\ell_0[\ell_1'](\alpha)$.   We will define
      a function $\bar g\in M^{\ell_1,h_0}$, and the desired embedding will
      be given by $\ell_0[\ell_1'](g)(\gamma)\mapsto
      \ell_1[\ell'_1](\bar g)( \ell'_0[\ell_1](\gamma))$.

      For each $\nu\in\domain(g)$, let the function $f_{\nu}$ and the
      generator $\beta_\nu$ of $\ell'_0[\ell_1]$ be
      such that $g(\nu)= \ell'_0[\ell_1](f_{\nu})(\beta_\nu)$.  Note
      that  $\ell'_0[\ell_1]\in M^{h_0}$, so 
      the function $h(\nu,\xi)=f_{\nu}(\xi)$ is also in $M^{h_0}$.
      Also $\seq{\beta_\nu\mid \nu\in\domain(g)}\in M^{h_0}$, and
      since $\sup(\Crit(\ell'_0[\ell_1]))<\min(\crit(\ell_0[\ell'_1])$, there is
      some $\beta$ such that $\beta_\nu=\beta$ for almost all $\nu$;
      that is, $\gamma\in \ell_0[\ell'_1](\set{\nu\mid \beta_\nu=\beta})$.    

      Now set $\bar g(\xi)= \ell'_0[\ell_1](h)(\beta,\xi)$, so 
      $\bar g(\ell'_0[\ell_1](\nu))=g(\nu)$ for almost all $\nu$.

      This completes the proof of Claim~\ref{thm:M_bDoesntmove_claim}
      and hence of Lemma~\ref{thm:M_bDoesntmove_helper}.
\end{proof}
\let\qed\relax
\end{proof}

\begin{proof}[Proof of Lemma~\ref{thm:M_bDoesntmove}]
  We will show that for any finite $b\subseteq I$ and $\alpha\in M_b$
  the sequence $\vec\nu$ given by Lemma~\ref{thm:M_bDoesntmove_helper} is also valid for
  iterations $k$ as in Lemma~\ref{thm:M_bDoesntmove}.  Note that such
  $k$, having all critical points in $M_{B\setminus b}$, satisfy the
  constraint given by  $\vec\nu$.

  Supposing the contrary, let $b$ be a sequence for which the claim
  fails, let
  $\alpha$ the least ordinal for which it fails, and let $k\in M_B$
  witness this failure.
  Set $\zeta=\otp(B)$, and let $G\subset i_{\ords}(P(\vec E\restrict\zeta)\mgkeq)$ 
  be the generic set 
  constructed in Subsection~\ref{sec:generic_set}, so that $k\in
  M_\ords[G]$. 
  Then there is a condition $s\in G$ such that
  $\set{\forceKappa^{s,\nu}\mid\nu\in\domain(s)}\subseteq
  b\cup\sing{\ords}$ which forces that $\alpha$ is the least
  counterexample and that $\dot k$ is a name for a witness to this failure. 
  
  The choice of $\vec\nu^{0}$ ensures that $k$ is continuous at $\alpha$,
  and therefore there is some $\alpha'<\alpha$ such that
  $k(\alpha')\geq\alpha$. 
  By Lemma~\ref{thm:prikry}(\ref{item:Prikrythm-inD}) there is a
  condition $s''\leq^*s$ in $G$ and a finite $e\subseteq\zeta$ such that any
  $s'\leq s$ with $e\subseteq\domain(s'')$  determines 
  $\alpha'$.   Fix $s'\leq s''$ in $M_b$ with $e\subseteq\domain(s')$
  and $\nu_{\lambda}<\forceKappa^{s',\nu}$ whenever $\lambda\in
  b\cup\sing0$ and $\lambda<\forceKappa^{s',\nu}$.

  Now let $ j\colon M_{\ords}\to M^{ j}$ be the iteration of $M_{\ords}$
  by the extenders
  \begin{equation}
    \label{eq:www}
    \seq{F^{s',\nu}_{\xi}\mid\nu\in\domain(s')\setminus\sing\ords\land
      \forceKappa^{s',\nu}\notin\lim(B)
      \land \xi\in\domain(\vec F^{s',\nu})}.
  \end{equation}

  Now  construct, as in
  Subsection~\ref{sec:generic_set} (except that the second component
  $\vec b$ of the conditions 
  of $R$ is modified
  appropriately),  $G'\subseteq  j\circ i_{\ords}(P(\vec 
  E\restrict\otp(B))\mgkeq)$ with $s'\in G'$.      Instead of taking all
  indiscernibles from $I$, this construction uses the iteration
  $j\circ i_{\ords}$, substituting the critical
  point of $F^{s,\nu}_{\xi}$ for the corresponding member of $B$
  whenever $F^{s,\nu}_{\xi}$ is  in the sequence~\eqref{eq:www}.

  Now factor $\dot k^{G'}$ as $\ell_1\circ\ell_0$ where $\ell_0$ uses the
  extenders of $\dot k^{G'}$ which are
  in $M_b$ and $\ell_1$ uses the remainder.
  Note that
  since $M_b$ is closed under countable sequences, $\ell_0\circ j\in
  M_b$, and since $\ell_0\circ j$ obeys $\vec\nu$ it follows that $\ell_0\circ j(\alpha)=\alpha$.    
  Therefore $\ell_0\circ j(\alpha')<\alpha$, but  $(\ell_1\circ
  \ell_0)(j(\alpha'))\geq j(\alpha)=\alpha$, so
  $\ell_1(\ell_0\circ j(\alpha'))>\ell_0\circ j (\alpha')$.    Since the map $j$ is
  elementary, this contradicts the minimality of $\alpha$.  
\end{proof}

\medskip

This concludes the preliminary observations, and 
we are now ready to continue with the proof of the Main 
Lemma,~\ref{thm:mainlemma}.  As was stated earlier, this proof is an
induction on the lexicographic ordering of pairs 
$(\iota, \phi)$ in order to prove that for all  \LS\ sequences $B$ and
all $x$ in $\chang_\iota\cap  M_B$,
\begin{equation}\label{eq:indeq}
  \chang_{B}\modelsCi \phi(x) \quad\text{if and only
    if}\quad \modelsCi\phi(x).
\end{equation}
Here and for the remainder of the paper we write $P\modelsCi\sigma$ to
mean that  $(\chang_{\iota})^{P}\models\sigma$.

The statement~\eqref{eq:indeq} uses the induction hypothesis: $\chang_{B}$ is not, by
its definition, a subset of $\chang$; however by the induction
hypothesis there is an embedding $\pi\colon
(\chang_{\iota})^{\chang_B}\to \chang_{\iota} $, which is the identity
on ordinals and is defined in general  by
setting $\pi(\set{y\in (\chang_{\iota'})^{M_B}\mid
  (\chang_{\iota'})^{M_B}\models \phi(y, a)})=
\set{y\in \chang_{\iota'}\mid
  \chang_{\iota'}\models \phi(y, \pi(a))})$.    For the rest of
this section we will identify $(\chang_{\iota})^{M_B}$ with the range
of $\pi$.

We will need an additional induction hypothesis in order to carry out the
proof of Lemma~\ref{thm:mainlemma}.
This hypothesis is rather
technical and uses notation which will be developed during the proof of the induction
step for Lemma~\ref{thm:mainlemma}, so we defer its statement, as
Lemma~\ref{thm:mainlemma2}, until it is needed to complete that proof.

By standard arguments, the only problematic part of the proof of the induction step
for Lemma~\ref{thm:mainlemma}  is the assertion 
that the existential quantifier is
preserved downwards:   We assume that $\psi(x,y)$ is a formula which
satisfies~\eqref{eq:indeq}, and want to prove that
\begin{equation}\label{eq:existsImplies}
  \forall x\in \chang_{B} \, \bigl(\modelsCi\exists
  y\,\psi(x,y)\implies \chang_{B}\modelsCi\exists y\,\psi(x,y)\bigr).
\end{equation}

Since the basic problem in the proof is dealing with gaps in $B$, it
will be helpful to  introduce some terminology to describe their
structure.   A \emph{gap} of $B$ is a 
maximal nonempty  interval of $I\setminus B$.     
For a \LS{} set $B$, the gap
will be a half open interval $[\sigma,\delta)$ where $\sigma$ is the
supremum of an $\omega$-sequence of members of $B$, and $\delta$ is
either $\min(B\setminus\sigma)$ or $\ords$.     We call $\delta$ the
\emph{head} of the gap.

Let $\delta'=\sup(\sigma\cap I\setminus B)$, or $\delta'=0$ if $I\cap
\sigma\subseteq B$.   Then $[\delta',\sigma)\cap I\subseteq B$; we
refer to this interval as the \emph{block} of $B$ corresponding to
the gap, and to $\delta'$ (which either is $0$ or is also the head of
a gap below $\delta'$)  as the foot of the block.  
If $\sigma'=\sup((B\cap\lim(I))\cap \delta)$ then
$B\cap(\sigma',\sigma)=I\cap(\sigma',\delta)$
is an $\omega$ sequence of successor members
of $B$; we will refer to this interval as  \emph{the tail} of the
gap.   If $\gamma$ is any member of this tail then we will refer
to the interval $[\gamma,\sigma)\cap B$ as \emph{the tail  of $B$ above $\gamma$}.

Call a set $b\subseteq B$ a \emph{tail traversal} of $B$ if it contains
exactly one point from the tail of each gap in $B$.  
Then $b$ determines a suitable subsequence $\tilde B\subseteq B$ as follows: let
$\delta$ be the head of a block in $B$, let $\delta'$ be the foot of the
associated block, and let $\gamma$ be the unique member of
$b\cap[\delta',\delta)$.   Then we regard $\gamma$ as dividing this
block of $B$ into three parts: the closed interval
$[\delta',\gamma)\cap B$, which we will call a \emph{closed block of $B$
  below $\gamma$},  
the singleton $\sing{\gamma}$, and the tail $(\gamma,\delta)\cap B$,
which we will call the \emph{tail above $b$}.
The suitable subsequence $\tilde B$ determined by $b$ is the union of
the closed blocks of $B$ below the members of $b$.

The maximal suitable subsequences of $B$ are those which are
determined by some tail traverse of $B$.    Note that any suitable
subsequence of $B$ is contained in a maximal subsequence, and hence in
dealing with $\chang_B$ we only need to consider maximal suitable subsequences.

\medskip{}
We are now ready to begin the proof of  the induction step for Lemma~\ref{thm:mainlemma}.
Suppose that $\phi(x)$ is the formula $\exists y\,\psi(x,y)$ and
is true in $\chang_{\iota}$, and that $B$ is a \LS\ sequence with
$x\in \chang_B$. 
Fix a tail traversal $b$ of $B$ such that
$\sing{x,\iota}\subseteq\chang_{\tilde B}$, 
where $\tilde B$ is the suitable subsequence of $B$ determined by $b$.
Pick $y$ so $\modelsCi\psi(x,y)$ and let $B'\supseteq B$ be a
\LS\ sequence with $y\in \chang_{B'}$.   By the induction hypothesis
$\chang_{B'}\modelsCi\psi(x,y)$.

We will define an iteration map $k$ and an isomorphism
$\sigma$ as in   Diagram~\eqref{eq:diagram_1}. 
\begin{equation}
  \label{eq:diagram_1}
  \begin{tikzpicture}[>=angle 90,baseline]
    \matrix(m)[matrix of math nodes, row sep =2.6em, column sep=2.8em, text height = 1.5ex, text depth=0.25ex]
    {%
      &&M_{B'}&&M_{B'}\xre\\
      M&M_{\tilde B}&M_B&M_k&M_k\xre\\
    };
    \path  [right hook->]  (m-2-2) edge (m-2-3)
    (m-2-3) edge (m-1-3);
    \path  [left hook->]   (m-1-5) edge (m-1-3)
    (m-2-5) edge (m-2-4);
    \path[->] (m-1-5) edge[decorate,decoration={snake,segment
      length=1.2mm, amplitude=.2mm}]  node[auto]{$\sigma$}  (m-2-5); 
    \path  [->]                        (m-2-3) edge  node[auto]{$k$}       (m-2-4)
    (m-2-1) edge  node[auto]{$i_\ords$}  (m-2-2);
  \end{tikzpicture}
\end{equation}

The map $k$ will be an iterated ultrapower using iterated extenders with critical points in $b$.  It has length greater than $\omega_1$, but is definable in $M_B[c]$ from
a countable sequence $c\in M_B$ of ordinals.  
The iteration $k$ has two purposes:
\begin{enumerate}\item 
  It includes one iteration step for each member of $B'\setminus\tilde B$
  (excluding a tail in $B'$ of  each gap of $B$).

\item 
  For each gap in $B'$ which does not correspond to a gap of $B$, it
  includes an
  $\omega_1$-sequence of iteration steps  inserted in order to  emulate
  this gap inside $M_k$.
\end{enumerate}
The submodel $M_{k}\xre$ of $M_k$ will be obtained by using
only the iterations from clause~1, omitting those from clause~2.   The
isomorphism $\sigma$ will map members of $B'\setminus B$ to the
corresponding critical points of ultrapowers in clause~1, and the
submodel $M_{B'}\xre$ of $M_{B'}$ will be obtained by taking only the
generators belonging to members of $B'\setminus B$ which correspond to
generators of 
extenders used in the iteration steps from clause~1.

The iteration $k$ will be such that Lemma~\ref{thm:M_bDoesntmove}
implies that the restrictions of $k$ and $\sigma$ to ordinals in
the suitable submodel $M_{\tilde B}$ are the identity.
The iteration  $k$  can be defined in $M_B[c]$, for a countable
sequence $c$ of ordinals, and thus is definable in the extension $M_B[G]$. 
The models $M_B$ and $M_k$ have the same ordinals and the 
same associated Chang model $\chang_B=\chang_k$.
Thus  Diagram~\eqref{eq:diagram_1}  induces  the following diagram:

\begin{equation}
  \label{eq:diagram_2}
  \begin{tikzpicture}[>=angle 90, baseline]
    \matrix(m)[matrix of math nodes, row sep =2.6em, column 
    sep=2.8em, text height = 1.5ex, text depth=0.25ex]
    {%
      &\chang_{B'}&&\chang_{B'}\xer\\
      \chang_{\tilde B}&\chang_B&\chang_k=\chang_B&\chang_{k}\xer\\
    };{
      \path  [right hook->] (m-2-1) edge (m-2-2)
      (m-2-2) edge (m-1-2);
      \path  [left hook->]  (m-1-4) edge (m-1-2)
      (m-2-4) edge (m-2-3);
      \path[->] (m-1-4) edge    
      [decorate,
      decoration={snake,segment length=1.2mm, 
        amplitude=.2mm}
      ] node[auto]{$\sigma$}  (m-2-4);
      \path  [->]   (m-2-2) edge  node[auto]{$k$} (m-2-3) ;
    }\end{tikzpicture}
\end{equation}
Once this machinery has been put into place, we will be able to  complete the
proof of the induction step for Lemma~\ref{thm:mainlemma}: we are
assuming $\modelsCi\psi(x,y)$, with $x$ and $y$ in $\chang_{B'}$, so by
the 
induction hypothesis $\chang_{B'}\modelsCi\exists y\psi(x,y)$.
An easy proof will give
Lemma~\ref{thm:M-B-eta-elem}, which implies that 
$\chang_{B'}\xre\prec \chang_{B'}$, so $\chang_{B'}\xre
\modelsCi\exists y\psi(x,y)$.   Fix $y\in \chang_{B'}\xre$ so that
$\chang_{B'}\xre\modelsCi\psi(x,y)$.   Since $\sigma$ is an
isomorphism, $\chang_{k}\xre\modelsCi\psi(x,\sigma(y))$.    

Now we want to conclude that  $\chang_{B}\models \psi(x,\sigma(y))$,
but unlike the case in the previous paragraph, we don't know of a direct proof that
$\chang_{B}\xre\prec 
\chang_{B}$.    Instead we will state a slightly generalized form of the
needed fact  as Lemma~\ref{thm:mainlemma2}, and with this as an
additional induction hypothesis conclude the
proof of the induction step for Lemma~\ref{thm:mainlemma}.        We
then use the induction hypothesis (including the just proved fact that
Lemma~\ref{thm:mainlemma} holds for the pair $(\iota,\phi)$) to prove that
Lemma~\ref{thm:mainlemma2} holds for $(\iota,\phi)$; this will
complete the proof of Lemmas~\ref{thm:mainlemma}
and~\ref{thm:mainlemma2}, and thus of Theorem~\ref{thm:main}, except
for the assumption that $\kappa_0\in B$.
\medskip{}

We now give the details of the construction of Diagram~\eqref{eq:diagram_1}.  
We already have the four models on the left of the diagram: $B$ is the
given \LS\ sequence, $\tilde B\subset B$ is a suitable subsequence with
$x\in M_{\tilde B}$ which is characterized by a tail traversal $b$ of $B$, and
$B'\supseteq B$ is a \LS\ sequence with a witness $y$ to $\exists
y\;\psi(x,y)$.  
The following definition is more general than needed here.  The added
generality is used in the proof of Lemma~\ref{thm:mainlemma2}.

\begin{definition}\label{def:vgsequence}
  A \emph{\vg{} sequence} for $B$ is a triple $(b, \vec\eta,g)$
  satisfying the following four conditions:
  \begin{enumerate}
  \item The set $b$ is a tail traversal sequence of $B$.
  \item $\vec\eta$ is a function with $\domain(\vec\eta)=
    \set{(\lambda,\xi)\mid \lambda\in b\land
      \xi<\nu_\lambda}$, where  $\nu_\lambda$ is a countable ordinal
    for each $\lambda\in b$.
  \item $g\subset\domain(\vec\eta),$ and if $(\lambda,\xi)\in g$ then
    $\xi$ is a limit ordinal.
  \item \label{item:vg_forcing}
    Define an order $\precdot$ on $B\cup\domain(\vec\eta)$ using the
    ordinal order on   $B$, the  lexicographic order on
    $\domain(\vec\eta)$, and setting 
    $\lambda'\precdot(\lambda,\xi)\precdot\lambda$ when $\lambda'<\lambda\in B$
    and $(\lambda,\xi)\in\domain( \vec\eta)$.

    Then $\eta_{\lambda,\xi}>\otp(\set{z\in B\cup\domain(\eta)\mid
      z\precdot(\lambda,\xi)})$.
  \end{enumerate}

  We will say that  $(b,\vec\eta,g)$ is a \vg\ sequence for $B'$ over  $B$
  if in addition the following four conditions hold:
  \begin{myinparaenum}
  \item 
    $B'$ and $B$ are \LS\ sequences with $B'\supset B$.
  \item\label{item:vgover-tauiso} $B'$ has the
    same order type as $(B\cup\domain(\eta),\precdot)$.   In the
    following, we write
    $\tau\colon (B\cup\domain(\eta))\to B'$ for  the order isomorphism. 
  \item\label{item:vgover-identity}
    $\tau$ is the identity on the suitable subsequence $\tilde B$ of
    $B$ determined by $b$.
  \item \label{item:vgover-gisright}
    if $\gamma\in B'\setminus B$ then $\tau(\gamma)\in g$ if and
    only if $\gamma$ is the head of a gap in $B'$.
  \end{myinparaenum}
\end{definition}
Note that if  $(b,\vec\eta,g)$ is a \vg\ sequence for $B'$ over  $B$
then $b'=\tau^{-1}[b]$ is a traversal of the
tails in $B'$ of the gaps of $B$, and that if $\lambda\in b'$ then
$\tau$ maps the tail above $\lambda$ in $B'$ to the tail above
$\tau(\lambda)$ in $B$.

For the construction of Diagram~\eqref{eq:diagram_1},  we use the following \vg\ sequence
$(b,\vec\eta,g)$ for $B'$ over $B$:   The function $\vec\eta$ is 
a constant function, with the constant value $\eta$ to be specified
later.  Fix a traversal $b'$ of the tails in $B'$ belonging to  gaps of
$B$.  Then
\begin{myinparaenum}
\item $\domain(\vec\eta)=\set{(\lambda,\xi)\mid \lambda\in b\land
    \xi\leq\otp(B'\cap[\lambda,\lambda')}$, where $\lambda'$ is
  the member of $b'$ in the tail in $B'$ of the same gap as
  $\lambda$, and 
\item $g=\set{\tau(\gamma)\mid \gamma\in B'\setminus B\land \gamma
    \text{ is the head of a gap in }B'}$
\end{myinparaenum}

\begin{definition}
  \label{def:Bprime_vg}
  If $(b,\vec\eta,g)$ is a \vg\ sequence for $B'$ over $B$, then
  $M_{B'}\xre=\set{j_{\ords}(f)(a)\mid f\in M\land a\in
    [\mathcal{G}]^{<\omega}}$ where $\mathcal{G}$ is the following set
  of generators:  Let $\kappa_\nu$ be a member of $B'$ and let $\beta=i_{\nu}(\bar\beta)$
  be a generator belonging to $\kappa_{\nu}$.      Then
  \begin{equation*}
    \beta\in\mathcal{G} \iff \bigl(
    \tau(\kappa_\nu)\in B  \lor
    (\tau(\kappa,\nu)=(\lambda,\xi)\in\domain(\vec\eta)\land \bar\beta\in\supp(E_{\eta_{\lambda,\xi}}))\bigr).
  \end{equation*}
\end{definition}

Note that $M_{B'}\xre\prec M_{B'}$,  that $M_{\tilde B}\subseteq
M_{B'}\xre$ and, that if $\eta$ is chosen sufficiently large then $y\in
M_{B'}\xre$.    This is the first of two criteria for the choice of
$\eta$; the other is that $\eta>\omega^{\omega}\cdot \otp(B')$.

\begin{lemma}
  \label{thm:M-B-eta-elem}
  If $(b,\eta_{\lambda,\xi},g)$ is a \vg\ sequence for $B'$ over $B$
  then $\chang_{B'}\xre\prec\chang_{B'}$. 
\end{lemma}

\begin{proof}
  The construction of Subsection~\ref{sec:generic_set} can
  be carried out to obtain a $M_{B'}\xre$-generic subset $G\in
  i_{\ords}(P(\vec E\restrict\otp(B')\mgkeq)$.  The only change needed is
  that the range of the coordinate $b_{\gamma}$ in a condition of $R$ is restricted
  to $\supp(E_{\eta_{\lambda,\xi}})$ whenever
  $(\lambda,\xi)\in\domain(\vec\eta)$ and $\kappa_\gamma$ is the
  $\xi$th member of $B'$ above $\lambda$.

  Now let  $\phi$ be a formula which is true in $\chang_{B'}\xre$. 
  Then there is a condition $([r],b)$ in the forcing $R$ for $M_{B'}\xre$
  which establishes the parameters of $\phi$ and forces $\phi$ to be
  true.   This condition is also a condition in the forcing $R$ for
  $M_{B'}$, it
  establishes the parameters in the same way, and it forces that $\phi$
  holds in $\chang_{B'}$.
\end{proof}

Note that condition~\ref{item:vg_forcing} of
Definition~\ref{def:vgsequence}
is used here to ensure that the enough of the image of $E$ is
present at each of the $\kappa_\nu\in B'\setminus B$ to construct
the generic set as in section~\ref{sec:generic_set}. 

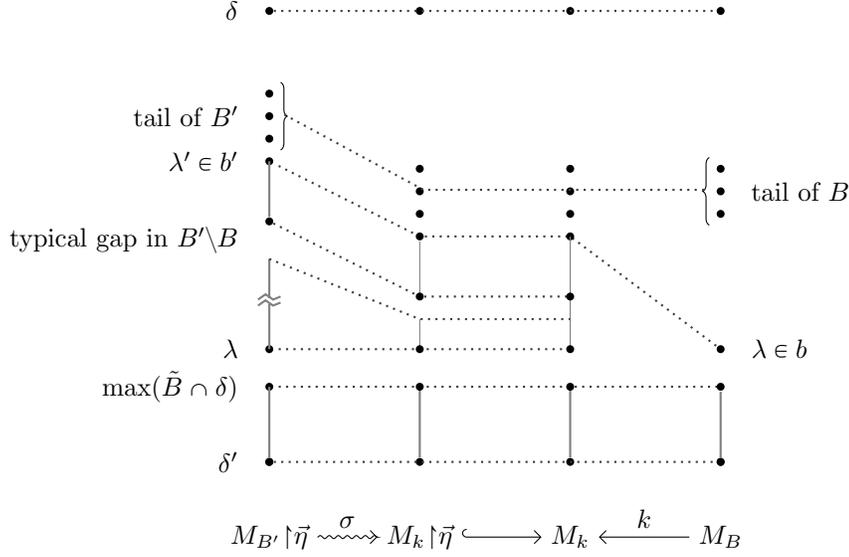
\begin{figure}[t]
  \[      
  \begin{tikzpicture} [>=angle 90, baseline,xscale=.5]
    \tikzstyle{viz}=[]   
    \tikzstyle{math}=[]
    \tikzstyle{mapping}=[dotted,thick,black!75!white]
    \tikzstyle{left-label}=[black,math,left=8]
    \tikzstyle{right-label}=[black,math,right=8]
    \tikzstyle{mydot}=[fill=black,circle,inner sep=0,minimum size=3pt]
    \tikzstyle{inmodel}=[thick,black!50!white]
    \tikzstyle{modeled}=[black!50!white]
    \def\colsep{4}
    \newcommand{\mytail}[1]{#1 node[mydot]{} +(0,.3) node[mydot]{} ++(0,.6) node[mydot]{}}
    \node[viz,mydot](a1) at (0,0) {};
    \node[viz,mydot](a9) at (0,6) {};
    \node[viz,mydot](b1) at (\colsep,0) {};
    \node[viz,mydot](b9) at (\colsep,6) {};
    \node[viz,mydot](c1) at (2*\colsep,0) {};
    \node[viz,mydot](c9) at (2*\colsep,6) {};
    \node[viz,mydot](d1) at (3*\colsep,0) {};
    \node[viz,mydot](d9) at (3*\colsep,6) {};
    \node[math](Bp)  [below of = a1] {$M_{B'}\xre$};
    \node[math](Bph) [below of = b1] {$M_{k}\xre$};
    \node[math](Bh)   [below of = c1] {$M_k$};
    \node[math](Bk)  [below of = d1] {$M_B$};
    \draw[->] (Bp) edge [decorate,decoration={snake,segment length=1.2mm,amplitude=.2mm}] 	     
    node[auto]{$\sigma$} (Bph);
    \draw[->]  (Bk) edge node[auto,swap]{$k$}(Bh);
    \draw[right hook->]   (Bph)  edge (Bh);
    %
    \draw[mapping]   (a1) node[mydot]{} node[left-label]{$\delta'$}   --  (d1) node[mydot]{}
    (a9) node[left-label]{$\delta$}  -- (d9);
    \draw (a1)  +(0,1) node(a2){}     +(0,1.5) node(a3){}  +(0,4) node(a4){}   +(0,4.5) node(a5){}
    +(0,5) node(a6){};
    \draw[mapping]  (a2)   node[left-label]{$\max(\tilde B\cap\delta)$} node[mydot]{} --
    ++ (3*\colsep,0)  node[mydot](d2){} ;
    \draw[mapping]  (a3) node[left-label]{$\lambda$} node[mydot]{} -- + (2*\colsep, 0)
    +(3*\colsep,0) node[mydot](d3){};
    \draw[mapping] (a4)   node[left-label]{$\lambda' \in b'$}
    node[mydot]{} -- ++(\colsep,-1) node[mydot](b4){}  
    -- ++(\colsep,0) node[mydot](c4){}
    -- (d3) node[right-label]{$\lambda\in b$} ;
    \draw               (a3)  +(0, .6) node(a3c){}            
    +(0,1.2) node(a3a){}             
    ++(0,1.7) node[mydot](a3b){};   
    \newcommand\brokenline{ +(-.3,-.02) -- +(-.1,+.05) -- +(.1, -.02) -- +(0.3, +.02) +(0,0)}
    \draw[inmodel] (a3.center) --   (a3c.center) 
    \brokenline   ++(0,0.08) \brokenline 
    --  (a3a.center)  +(0,.25)   node[left-label]{typical gap in $B'\setminus B$}   (a3b) -- (a4.center); 

    \draw[modeled] (b1 |- a3)node(b3)[mydot]{} -- ++(0,0.4)node(b3a){}       ++(0,0.3) node[mydot](b3b){} -- (b4)
    (c1 |- a3)node[mydot]{}       --  (c1 |- b4);
    \draw[mapping]   (a3b) -- (b3b) -- (b3b -| c1) node[mydot]{}
    (a3a.center) -- (b3a.center) -- (b3a -| c1);                              
    \draw[inmodel] (a1) edge (a2.center) (b1)-- (b1 |- a2)node[mydot]{}      (c1) -- (c1 |- a2)node[mydot]{}   (d1) -- (d2);
    \draw   (a4) \mytail{++(0,0.3)} +(0,-.3) node[left-label]{$\text{tail of }B'$};
    \draw[snake=brace] (a4)   +(.3, 1.05) -- +(.3, 0.15);
    \draw              (b4) \mytail{++(0,0.3)}     (c4)  \mytail{++(0,0.3)};
    \draw   (c4) ++(\colsep,0) \mytail{++(0,0.3)} +(0,-.3)
    node[right-label]{$\text{tail of }B$}; 
    \draw[snake=brace] (c4) ++(\colsep,0)  +(-.3, 0.15) -- +(-.3, 1.05);
    \draw[mapping]  (a4) ++(.4,0.62) -- ++(\colsep - .3, -1.0) --
    ++(\colsep, 0) -- +(\colsep - 0.45, 0); 
  \end{tikzpicture}
  \]
  \caption{The maps $\sigma$ and $k$ inside the block between $\delta'$
    and $\delta$ which is associated
    with the gap in $B$ headed  by $\delta$.    The dotted lines represent the
    maps $\sigma$ and $k$; the heavier vertical lines represent intervals of $I$
    contained in  the indicated models and the lighter ones represent the indiscernibles added by the iteration $k$.}
  \label{fig:k_def}
\end{figure}

We can now complete the construction of the elements of
Diagram~\eqref{eq:diagram_1} by defining $k$ and $\sigma$.    This
construction is illustrated in Figure~\ref{fig:k_def}.

\begin{definition}
  \label{def:k}
  We define by recursion on $z\in (B\cup\domain(\vec\eta),\precdot)$
  a sequence of embeddings         $k_{z}\colon M_{B}\to M^*_{z}$.  
  We will describe the construction on one of the blocks of $B$. 
  Thus, suppose that $\delta\in B$ is the head of a gap and
  $\delta'\in B\cup\sing{0}$ is the foot of the block of $B$ below it.
  We assume that $k_z\colon M_B\to M^*_z$ has been defined for all $z\precdot\delta'$. 
  Let $\lambda$ be the unique member of $b\cap[\delta',\delta)$.     
  \begin{compactenum}[(i)] 
  \item 
    $M^*_0=M_B$, and if $\delta'>0$ then 
    $M^*_{\delta'}=\dirlim\seq{(M_z^*;k_z):z\precdot \delta'}$.
  \item If $\nu\in \tilde B\cap
    [\delta',\delta)=B\cap[\delta',\lambda)$ then
    $M^*_\nu=M^*_{\delta'}$. 
  \item If $\nu\in B\cap(\lambda,\delta)$ then
    $M^*_\nu$  is the
    direct limit of the embeddings $k_z$ for $z\precdot
    \lambda$. 
  \item If $(\lambda,\xi)\in \domain(\vec\eta)\setminus g$ and
    $\xi$ is a limit ordinal then
    $M^*_\nu$ is the direct limit of the embeddings $k_z$ for
    $z\precdot(\lambda,\xi)$.  
  \item If $z=(\lambda,\xi+1)\in \domain(\vec\eta)$, or if
    $z=\lambda$ and $(\lambda,\xi)$ is its predecessor in
    $\precdot$, then $M^*_z=\ult(M^*_{(\lambda,\xi)}, 
    E^*_{\eta_{\lambda,\xi}})$ where, letting $\gamma$ be such that
    $\delta'=\kappa_{\gamma}$, we write $E^*_{\alpha}$ for
    $k_{\lambda,\xi}\circ i_{\gamma}(E_\alpha)$.
  \item If $z=(\lambda,\xi)\in g$, then set $\bar k^*_z\colon M_B\to
    M^{**}_z=\dirlim_{z'\precdot z}M^*_{z'}$.
    Then $M^*_z$ 
    is an iterated ultrapower of $M^{**}_z$
    of length $\omega_1$, using extenders $\bar k^{*}_z(i_{\gamma}(\vec
    F))$ where $\lambda=\kappa_{\gamma}$ and $\vec F\in M$ is an
    arbitrary but fixed cofinal subsequence of the 
    sequence of extenders below $E$ on $\kappa$ in $M$.
  \end{compactenum}
  If $\gamma\in B'$ and $\tau(\gamma)=(\lambda,\xi)\in \domain(\vec\eta)$, then
  $\sigma(\gamma)$ is equal to  the critical point of the ultrapower of 
  $M^*_{\tau(\gamma)}$.
\end{definition}

\begin{definition}  
  \label{def:sigma}
  The restriction of $\sigma$ to $B'$ is determined by the map
  $\tau$ specified in the Definition~\ref{def:vgsequence} of a \vg\
  sequence for $B'$ over $B$: if $\tau(\gamma)\in B$ then
  $\sigma(\gamma)=k(\tau(\gamma))$, and if
  $\tau(\gamma)=(\lambda,\xi)$ then $\sigma(\gamma)$ is the
  $\xi$th critical point of the iteration steps of $k$ using  extenders on $\lambda$.
  The restriction of $\sigma$ to $B'$ determines its restriction
  to generators of $M_{B'}\xre$, and this restriction  determines
  the remainder of $\sigma$.
\end{definition}

The particular choice of the sequence $\vec F$ of extenders will not
matter; a suitable choice for $F_\nu$ would be the least
$\kappa^{+(\nu+1)}$-strong   extender on $\kappa$.
It is important that $\vec F\in M$, for that implies that
$M_k$ is in $M_B[B,\vec\eta]$ and hence is in the generic
extension $M_B[G]$ of $M_B$ described in
section~\ref{sec:generic_set}; we  use this fact to identify the ordinals
of $M_k$ with those of $M_B$.    It is also important that
$\vec F$
is cofinal among the extenders below $E$  in $M$, and hence
$i_{\gamma}(\vec F)$ is cofinal among  the extenders on $\lambda$
in $M_B$:  this fact ensures (using
Lemma~\ref{thm:M_bDoesntmove})
that the restriction of $k$ to  the ordinals of $M_B$ is
independent of the choice of $\vec F$.    

\begin{todoenv}
{6/18/15  --- This actually goes beyond
    Lemma~\ref{thm:M_bDoesntmove}.}
\end{todoenv}
This completes the definition of the elements of
Diagram~\eqref{eq:diagram_1}, and the extension to the Chang model in 
Diagram~\eqref{eq:diagram_2} is straightforward.   
We have already observed that the Chang model $\chang_k$ built on
$M_k$ is the same as $\chang_B$, giving the identity  on the
bottom.    Lemma~\ref{thm:M-B-eta-elem} asserts that $C_{B'}\xre$ is
an elementary substructure of $\chang_{B'}$, and $\sigma\colon
C_{B'}\xre\to \chang_{k}\xre$ is an isomorphism.    It follows that
$\chang_{k}\xre\modelsCi \psi(x,\sigma(y))$, and
we will be finished if we can conclude from this that  that
$\chang_{B}\modelsCi \psi(x,\sigma(y))$.   This is implied by the case
$(\iota,\psi)$ of 
Lemma~\ref{thm:mainlemma2}, which is the promised addition to the induction
hypothesis to be used in  the proof of Lemma~\ref{thm:mainlemma}.   Thus this
concludes the proof of the induction step for
Lemma~\ref{thm:mainlemma}.

\begin{lemma}\label{thm:mainlemma2}
  Suppose that $B\subseteq B'$ are \LS\ sequences and $\vec \eta$ is a virtual gap
  construction sequence for $B'$ over $B$ such that
  $\eta_{\lambda,\xi}\geq\omega^{n}\cdot\otp(B\cup\domain(\vec\eta),
  \precdot)$ for all $(\lambda,\xi)\in\domain(\vec\eta)$ and $n<\omega$.
  Let $k\colon M_B\to  M_k$ be 
  the  \vg{} iteration, and    let $\chang_k\xre\subseteq 
  \chang_k$ be as given in Diagram~\eqref{eq:diagram_2}.
  Then $\chang_{k}\xre\prec \chang$.
\end{lemma}

\begin{proof}
  As was stated earlier, this proof is a simultaneous induction along with 
  Lemma~\ref{thm:mainlemma}.     We have completed
  the proof that Lemma~\ref{thm:mainlemma} holds
  for $(\iota,\phi)$, using as an induction hypothesis that Lemmas~\ref{thm:mainlemma} and~\ref{thm:mainlemma2} hold for all smaller pairs.
  We now use this same induction hypothesis, together
  with the fact that Lemma~\ref{thm:mainlemma} holds for $(\iota,\phi)$,  to prove that Lemma~\ref{thm:mainlemma2} holds
  for  $(\iota,\phi)$: that is, if
  $B$, $k$  and $\eta$ are as in Lemma~\ref{thm:mainlemma2}
  and $x$ is an arbitrary member of
  $C_{k}\xre$ such that $\modelsCi\exists y\psi(x,y)$, then $\chang_k\xre\modelsCi\exists y\psi(x,y)$.   
  By the newly proved case of  Lemma~\ref{thm:mainlemma},
  $\chang_{B}\modelsCi\exists y\,\psi(x,y)$. 
  Fix $y_0\in\chang_{B}$ so that $ \modelsCi\psi(x,y_0)$.
  We now define an extension $\vetap$ of the \vg{} sequence
  $\vec \eta$ such that $y_0\in C_{B}\xrep$.
  The sequence $\vetap$ will have the same sets $b$
  and $g$ as $\vec\eta$, but the domain    of $\vec \eta'$ will be enlarged
  by adding an $\omega$ sequence of new elements below each $(\lambda,\xi)\in g$.
  Thus, for each $\lambda\in b$ define a map $t_{\lambda}$ with
  $\domain(t_{\lambda})=\len(\vec\eta_\lambda)$  by 
  \begin{equation*}
    t_{\lambda}(\xi) =
    \begin{cases}
      0&\text{if $\xi=0$,}\\
      t_{\lambda}(\xi')+1&\text{if $\xi=\xi'+1$,}\\
      \sup_{\xi'<\xi}t_{\lambda}(\xi')&\text{if $\xi$ is a limit and
        $(\lambda,\xi)\notin g$}\\
      \sup_{\xi'<\xi}t_{\lambda}(\xi')+\omega&\text{if     $(\lambda,\xi)\in g$}.
    \end{cases}
  \end{equation*}
  Now we define $\vetap$, using an ordinal $\eta'\in\omega_1$ to be
  determined shortly: 
  \begin{gather*}
    \domain(\vetap)=\set{(\lambda,\xi)\mid\xi<\sup\range(t_{\lambda})}\\
    b^{\vetap}=b^{\vec\eta}\text{, and }
    g^{\vetap}=\set{(\lambda,t_{\lambda}(\xi)\mid (\lambda,\xi)\in
      g^{\vec\eta}}, \text{ and}\\
    \eta'_{\lambda,\xi} =
    \begin{cases}
      \eta_{\lambda,\xi'}&\text{if $\xi=t(\xi')$ }\\
      \eta'&\text{if $(\lambda,\xi)\notin\range(t)$}.
    \end{cases}
  \end{gather*}
  The first condition on $\eta'$ is that $\eta'\geq
  \omega^{n}\cdot\otp\bigl(B\cup\domain(\vetap,\precdot)\bigr)$ for each
  $n\in\omega$, and the second condition is that $y_0\in \chang_{B}\xrep$.
  It is possible to satisfy the second condition  since
  $\chang_B=\bigcup_{\eta'<\omega_1}\chang_B\xrep$.  Notice that the
  first condition implies that $\vetap$ satisfies the hypothesis of
  Lemma~\ref{thm:mainlemma2}, since if  $\xi=t_{\lambda}(\xi')$ then 
  $\eta'_{\lambda,\xi}=\eta_{\lambda,\xi'}>
  \omega^{n+1}\cdot\otp(B\cup\domain(\vec\eta),\precdot)=\omega^{n}\cdot
  \omega\cdot\otp(B\cup\domain(\vec\eta),\precdot)
  \geq\omega^{n}\cdot\otp(B\cup\domain(\vetap),\precdot)$.

  \newcommand\taumap{\tau}
  
  \begin{equation}\label{eq:mainlemma2-diag}
    \begin{tikzpicture}[>=angle 90,baseline]
      \tikzstyle{mysnake}=[decorate,decoration={snake,segment length=1.2mm,amplitude=.2mm}]; 	  
      \matrix(m)[matrix of math nodes, row sep =2.6em, column sep=2.8em, text height = 1.5ex, text depth=0.25ex]
      {&M_{B''}&&M_{B''}\etarestrict\vetap\\
        M_{B'}&&M_{B'}\etarestrict\vec\eta\\
        M_B&M_k&M_k\etarestrict\vec\eta \\
        &M_{k'}&&M_{k'}\etarestrict\vetap\\
      };
      \path  [right hook->](m-2-1) edge                     (m-1-2)
      (m-2-3.north east) edge                     (m-1-4)
      (m-3-1) edge                     (m-2-1);
      \path[->]            (m-3-1) edge node[auto] {$k$}    (m-3-2)
      edge   node[auto] {$k'$}    (m-4-2)
      (m-2-3) edge [mysnake] node[auto] {$\sigma$} (m-3-3)
      (m-3-2) edge node[auto] {$\taumap$}  (m-4-2)
      (m-1-4) edge [mysnake] node[auto] {$\sigma'$} (m-4-4);            
      \path[right hook->]   (m-3-3.south east) edge  node[sloped,above] {$\taumap$} (m-4-4);
      \path[left hook->]   (m-3-3) edge                     (m-3-2)
      (m-4-4) edge                     (m-4-2)
      (m-1-4) edge                     (m-1-2) 
      (m-2-3) edge                     (m-2-1);
    \end{tikzpicture}
  \end{equation}
  For the remainder of the proof we refer to
  Diagram~\eqref{eq:mainlemma2-diag}.   The inner rectangle is the same
  as Diagram~\eqref{eq:diagram_1}.   
  The map $\taumap$ is determined by using the map
  $(\lambda,\xi)\mapsto(\lambda,t_{\lambda}(\xi))$ to map the generators
  of indiscernibles from $\vec\eta$ into those of $\vetap$.
  As with Diagrams~\eqref{eq:diagram_1} and~\eqref{eq:diagram_2}, 
  Diagram~\eqref{eq:mainlemma2-diag} induces a similar diagram for the
  corresponding Chang models.    

  \begin{todoenv} {6/18/15 --- MAYBE --- Again,
      Lemma~\ref{thm:M_bDoesntmove} doesn't really cover this, as
      stated.}
  \end{todoenv}

  We claim that
  $\taumap\restrict(\chang_k\xre)$ is the identity.   First, 
  Lemma~\ref{thm:M_bDoesntmove} implies that the restriction of $\taumap$ to the ordinals of
  $M_k\xre$ is the identity.   
  Now every member of $\chang_k\xre$ is represented by a term
  $w=\set{z\in\chang_{\iota'}\!\mid\quad\models_{\chang_{\iota'}}\phi(z,a)}$,
  where $\iota'\in M_k\xre$ and $a$ is a sequence of ordinals from
  $M_{k}\xre$.  Thus $\taumap(w)$ is represented by the same term in
  $\chang_{k'}\xre$.   But $\chang_{k}=\chang_{k'}=\chang_{B}$, so this
  term represents the same set $w$ in $\chang_{k'}$. 

  Now define $B''$ to be $B'$ together with the next $\omega$-many
  members of $I$ from each of the gaps of $B'$ which are not gaps of
  $B$.
  The right-hand trapezoid commutes, and in particular
  $\sigma^{-1}(x)=(\sigma')^{-1}\circ \tau(x)=(\sigma')^{-1}(x)$.    Now
  $\chang_{k'}\xrep\modelsCi \psi(x,y_0)$, and since $\sigma'$
  is an isomorphism it follows that
  $\chang_{B''}\xrep\modelsCi\psi(\sigma^{-1}(x),(\sigma')^{-1}(y_0))$.
  It follows by Lemma~\ref{thm:M-B-eta-elem} that $\chang_{B''}$ satisfies the
  same formula,  by the induction hypothesis Lemma~\ref{thm:mainlemma} for
  $(\iota,\phi)$ it follows that $\modelsCi\exists
  y\;\psi(\sigma^{-1}(x),y)$, and by 
  another application of the same
  induction hypothesis $\chang_B$ satisfies the same formula.
  By Lemma~\ref{thm:M-B-eta-elem},  $\chang_{B'}\etarestrict\vec\eta$
  does as well, so   let $y_1$ be such that $M_{B'}\xre\modelsCi\psi(\sigma^{-1}(x),y_1)$.  Then   
  $\chang_k\xre\modelsCi \psi(x,\sigma(y_1))$, so $\chang_k\xre\modelsCi\exists y\psi(x,y)$, as required.
\end{proof}

\subsection{ Finite exceptions and  $\kappa_0\notin B$}\label{sec:finite-exceptions}
In the last subsection we assumed that $\kappa_0=\kappa$ is a member
of $B$;  here we indicate how this extra assumption can be eliminated.  The
same argument is used in the proof of
Theorem~\ref{thm:modified-suitable}  to support the provision allowing finitely many exceptions.

The reason that the previous argument fails when $\kappa_0\notin B$ is that $\kappa_0$ may be a member of the extended model $B'$ of diagram~\ref{eq:diagram_1}.   In this case the definition of  
the map $k$ in Diagram~\eqref{eq:diagram_1} fails because there is no tail of $B$ in this first gap.

To conclude the proof of Theorem~\ref{thm:main}(\ref{item:main-upper}), 
suppose that $B=\set{\lambda_\nu\mid\nu\leq\zeta}$ is a \LS\ set
with $\lambda_0>\kappa_0$, that 
$x\in\chang_{B}$, and that $\chang\models\phi(x)$.    We want to show
that $\chang_{B}\models\phi(x)$.  
Let $B'=B\cup\set{\kappa_{n}\mid n<\omega}$, a \LS\ sequence of length $\omega+\delta$.   
Since $\kappa_0\in B'$,  the version of
Theorem~\ref{thm:main}(\ref{item:main-upper})
already proved implies
that $\chang_{B'}\models\phi(x)$.
Let $G$ be the $M_{B'}$-generic subset of 
$i_\ords(P(\vec E\restrict(\omega+\delta))\mgkeq)$ constructed 
in section~\ref{sec:generic_set},
and set \[G_1=\set{[p\restrict(\omega,\omega+\delta)]\mid [p]
  \in G \land \omega\in\domain(p)}. \]
Then $G_1$ is an $M_B$-generic subset of 
$i_\ords(P'\mgkeq)$, where $P'$ is the forcing described following
Lemma~\ref{thm:factorization} such that $P(\vec
E\restrict(\omega+\delta))\equiv P(\vec
E\restrict\omega)*\dot R$ is a regular suborder of  
$P(\vec E\restrict\omega+1)\times P'$.

Now let $[q]\in G$ be a condition such that $[q]\forces
\chang_{B'}\models\phi(x)$.  We may assume that
$\omega=\min(\domain(q))$.  Let $G_0$ be a $M_B$-generic subset of
$P(\vec F^{q,\omega})$ with $q\restrict\omega+1\in G_0$, and let
$\tilde G$ be the resulting 
$M_B$-generic subset of $i_\ords(P(\vec E\restrict(\omega+\delta))\mgkeq)$.
Then $[q]\in \tilde G$, so $M[\tilde G]\models
\chang_{B''}\models\phi(x)$, where $B''$ is the set
$\set{\forceKappa_n\mid n\in\omega}\cup B$, interpreted as having,
like $B'$, a gap headed by $\lambda_0$.  Now the forcing does add a
new countable sequence of ordinals, as  $M[\tilde
G]\models\cof(\lambda_0)=\omega$.  However, 
$\lambda_0$ is being interpreted as the head of a gap and therefore $\chang_{ B''}=\bigcup\set{\chang_{\tilde B}\mid \tilde B\subset B\land \tilde B\text { is suitable}}$.  
Since the forcing $P(\vec F^{q,\omega})\mgkeq$ does not add bounded subsets of $\lambda_0$, this implies that $\chang_{B''}$, as defined inside $M[\tilde G]$, is equal to 
$\chang_B$.   This concludes the proof that  $\chang_B\models\phi(x)$.

\bigskip{}
It is critical to this argument that there are only a finite number of intervals
(in this case, only one interval) of $B$ which need special attention.
Finitely many such special cases can be dealt with a condition $q$
obtained, as in the proof, by finitely many one-step extensions, but
infinitely many would involve
adding Prikry type sequences, which  requires the use of the 
iteration to obtain genericity.

\section{Questions and Problems}
\label{sec:questions}
This study leaves a number of questions open.  Two which were
mentioned in the introduction essentially involve filling gaps in this paper:

\begin{question}\label{q:exact}
  Exactly what is the large cardinal strength of a sharp for $\chang$?
\end{question}
Theorem~\ref{thm:main} puts it between a mouse over the reals satisfying
$o(\kappa)=\kappa^{+(\omega+1)}+1$ and a sufficiently strong mouse
over the reals
satisfying $o(\kappa)=\kappa^{+\omega_1}+1$.   The second question
asks whether this procedure truly gives a sharp for the Chang model:
\begin{question}
  Can the restricted formulas be removed from the
  definition~\ref{def:Csharp} of the sharp for the Chang model?  That
  is, can the added Skolem functions be made full-fledged members of
  the language?
\end{question}

The next questions ask for more detailed information about the
structure of the sharp:
\begin{question}
  What is $K(\mathcal{R})^{\chang}$?   Is it an iterate (not moving
  members of $I$) of
  $M_{\Omega}\cut\Omega$ for some mouse $M$ over the reals?  If so, is
  this iteration definable in
  $L[M,\set{\lambda\mid\cof(\lambda)=\omega}]$? 
\end{question}

\begin{question}
  What is the core model $K^{\chang}$ of the Chang model?
  How does it relate to   $K^{L(\mathcal{R})}$ and to $K(\mathcal{R})^{\chang}$?
\end{question}
\begin{todoenv}
  Again, a further question: how much of the covering lemma survives?
\end{todoenv}
\begin{question}
  Is it true that the  measurable cardinals of 
  $K(\reals)^{\chang}$ are exactly the regular cardinals of
  $K(\reals)^{\chang}$ which have countable cofinality in $V$?
\end{question}
\medskip{}

The final question is about the next step from the Chang model.   The
$\omega_1$-Chang model $\omega_1$-$\chang$ is obtained by closing
under $\omega_1$-sequences of ordinals.

\begin{question}
  What can be said about the $\omega_1$-Chang model?
\end{question} 
The question is due to Woodin  (personal communication), as is most 
of the known information.  Gitik has pointed out that (contrary to
my earlier belief) his technique of recovering extenders from threads, or
strings of indiscernibles, appears to be essentially unlimited for
strings whose length has uncountable cofinality.  It follows that
the lower bound, the counterpart to
Theorem~\ref{thm:main}(\ref{item:main-lower}), is probably at least as
large as any cardinal for which there is a pure extender model.

There is one minor caveat to this statement:
\begin{proposition}\label{thm:omega1-sharp-lower-bound}
  Suppose that $V=L[\mathcal{E}]$ is an extender model, and that
  there is an iterated ultrapower $i\colon V\to M$ where $M$ is a
  definable submodel of $\omega_1$-$\chang$.  Then there is no
  strong cardinal in $V$.
\end{proposition}
\begin{proof}
  Suppose the contrary, and let $\kappa$ be the smallest strong
  cardinal.  Then $i(\kappa)$ is the smallest strong cardinal in
  $M$.  However, since $\kappa$ is strong there is an extender $E$
  with critical point $\kappa$ such that $i^{E}(\kappa)>i(\kappa)$
  and
  ${\vphantom{\bigl)}}^{\omega_1}{\ult(L[\mathcal{E}],E)}\subseteq\ult(L[\mathcal{E}],E)$.
  Then
  $\omega_1$-$\chang=(\text{$\omega_1$-$\chang$})^{\ult(L[\mathcal{E}],E)}$,
  but the smallest strong cardinal in the latter is
  $i^{E}(i(\kappa))\ge i^E(\kappa)>i(\kappa)$.
\end{proof}

However this observation has no implications for the existence of a
sharp for $\omega_1$-$\chang$.    For example, if $V=L[\mathcal{E}]$
where $\mathcal{E}$ is a proper set, then so long as
$K^{\omega_1\text{-}\chang}$ exists and is sufficiently iterable, Gitik's
technique gives an iterated ultrapower from $L[\mathcal{E}]$ to $K^{\omega_1\text{-}\chang}$.

Woodin has observed that the existence of a sharp for
$\omega_1$-$\chang$ would imply  the Axiom of Determinacy, which
 implies that there is no embedding from $\omega_1$ into
the reals in $\omega_1$-$\chang$, and hence none in $V$.   Thus a
sharp for $\omega_1$-$\chang$ is inconsistent with the Axiom of
Choice in $V$.
However it would be of interest to find a sharp for the
$\omega_1$-Chang model as defined inside an inner model which satisfies the
Axiom of Determinacy.
\bibliographystyle{abbrv}
\bibliography{logic}
\todos
\end{document}
